\documentclass[12pt]{amsart}

\usepackage{epsfig,amsxtra,amssymb,amsthm,amsmath,amscd,url,listings}
\usepackage[utf8]{inputenc}
\usepackage{fullpage}
\usepackage{scrtime}

\usepackage{scrtime}
\usepackage[colorlinks]{hyperref}
\usepackage{mathrsfs}
\usepackage{tensor}

\usepackage{enumitem}
\setenumerate{itemsep=5pt}
\setitemize{itemsep=5pt}

\newlist{enumth}{enumerate}{1}
\setlist[enumth]{label=\emph{(\arabic*)}, ref=(\arabic*)}
\DeclareMathSymbol{A}{\mathalpha}{operators}{`A}%
\DeclareMathSymbol{B}{\mathalpha}{operators}{`B}%
\DeclareMathSymbol{C}{\mathalpha}{operators}{`C}%
\DeclareMathSymbol{D}{\mathalpha}{operators}{`D}%
\DeclareMathSymbol{E}{\mathalpha}{operators}{`E}%
\DeclareMathSymbol{F}{\mathalpha}{operators}{`F}%
\DeclareMathSymbol{G}{\mathalpha}{operators}{`G}%
\DeclareMathSymbol{H}{\mathalpha}{operators}{`H}%
\DeclareMathSymbol{I}{\mathalpha}{operators}{`I}%
\DeclareMathSymbol{J}{\mathalpha}{operators}{`J}%
\DeclareMathSymbol{K}{\mathalpha}{operators}{`K}%
\DeclareMathSymbol{L}{\mathalpha}{operators}{`L}%
\DeclareMathSymbol{M}{\mathalpha}{operators}{`M}%
\DeclareMathSymbol{N}{\mathalpha}{operators}{`N}%
\DeclareMathSymbol{O}{\mathalpha}{operators}{`O}%
\DeclareMathSymbol{P}{\mathalpha}{operators}{`P}%
\DeclareMathSymbol{Q}{\mathalpha}{operators}{`Q}%
\DeclareMathSymbol{R}{\mathalpha}{operators}{`R}%
\DeclareMathSymbol{S}{\mathalpha}{operators}{`S}%
\DeclareMathSymbol{T}{\mathalpha}{operators}{`T}%
\DeclareMathSymbol{U}{\mathalpha}{operators}{`U}%
\DeclareMathSymbol{V}{\mathalpha}{operators}{`V}%
\DeclareMathSymbol{W}{\mathalpha}{operators}{`W}%
\DeclareMathSymbol{X}{\mathalpha}{operators}{`X}%
\DeclareMathSymbol{Y}{\mathalpha}{operators}{`Y}%
\DeclareMathSymbol{Z}{\mathalpha}{operators}{`Z}%

\renewcommand{\leq}{\leqslant}
\renewcommand{\geq}{\geqslant}

\numberwithin{equation}{section}

\newcommand{\uple}[1]{\text{\boldmath${#1}$}}

\def\setminus{\mathchoice
    {\mathbin{\vrule height .62ex width 1.61ex depth -.38ex}}% 12
    {\mathbin{\vrule height .62ex width 1.61ex depth -.38ex}}% 12
    {\mathbin{\vrule height .50ex width 0.85ex depth -.28ex}}%  9
    {\mathbin{\vrule height .20ex width 0.570ex depth -.24ex}}%  7
}

\renewcommand{\mathcal}{\mathscr}
\newcommand{\sheaf}[1]{\mathcal{{#1}}}

\newcommand{\bfb}{\uple{b}}
 
\newcommand{\bfx}{\uple{x}}
\newcommand{\bfy}{\uple{y}}

\newcommand{\Cc}{\mathbf{C}}

\newcommand{\Aa}{\mathbf{A}}

\newcommand{\Zz}{\mathbf{Z}}

\newcommand{\Rr}{\mathbf{R}}
\newcommand{\Gg}{\mathbf{G}}

\newcommand{\Qq}{\mathbf{Q}}
\newcommand{\Fp}{{\mathbf{F}_p}}
\newcommand{\Fpt}{{\mathbf{F}^\times_p}}

\newcommand{\Ff}{\mathbf{F}}

\newcommand{\bQl}{\overline{\Qq}_{\ell}}
\newcommand{\Qlb}{\bQl}

\newcommand{\mcV}{\mathscr{V}}
\newcommand{\mcW}{\mathscr{W}}

\newcommand{\mcH}{\mathscr{H}}
\newcommand{\mcZ}{\mathscr{Z}}

\def\loccit{loc.\kern3pt cit.{}\xspace}
\def\cf{see\kern.3em}
\def\Cf{See\kern.3em}
\def\eg{e.g.\kern.3em}
\def\ie{i.e.,\ }
\def\resp{\text{resp.}\kern.3em}

\newcommand{\mods}[1]{\,(\mathrm{mod}\,{#1})}

\newcommand{\lra}{\longrightarrow}

\DeclareMathOperator{\frob}{\mathrm{Fr}}

\DeclareMathOperator{\Spec}{Spec}

\DeclareMathOperator{\rank}{rank}

\DeclareMathOperator{\Kl}{\mathrm{Kl}}

\DeclareMathOperator{\supp}{supp}

\DeclareMathOperator{\Tr}{Tr}

\DeclareMathOperator{\codim}{codim}
\DeclareMathOperator{\ft}{FT}

\DeclareMathOperator{\dual}{D}

\renewcommand{\rho}{\varrho}

\DeclareMathOperator{\Un}{\mathbf{U}}

\theoremstyle{plain}
\newtheorem{theorem}{Theorem}[section]
\newtheorem*{theorem*}{Theorem}
\newtheorem{lemma}[theorem]{Lemma}
\newtheorem{corollary}[theorem]{Corollary}

\theoremstyle{remark}

\theoremstyle{definition}

\newtheorem{definition}[theorem]{Definition}

\newtheorem{example}[theorem]{Example}
\newtheorem{remark}[theorem]{Remark}

\newcommand{\mcM}{\mathscr{M}}
\newcommand{\mcL}{\mathscr{L}}

\newcommand{\mcF}{\mathscr{F}}

\newcommand{\mcX}{\mathscr{X}}
\newcommand{\mcY}{\mathscr{Y}}

\newcommand{\pH}{\tensor*[^{\mathfrak{p}}]{\mathscr{H}}{}}
\newcommand{\pt}{\tensor*[^{\mathfrak{p}}]{t{}}{}}

\usepackage{xcolor}

\begin{document}

\title{Stratification theorems for exponential sums in families}
 
\author{Dante Bonolis}
\address{Duke University, 120 Science Drive, Durham NC 27708}
\email{dante.bonolis@duke.edu}

\author{Emmanuel Kowalski}
\address{ETH Z\"urich -- D-MATH\\
  R\"amistrasse 101\\
  CH-8092 Z\"urich\\
  Switzerland} \email{kowalski@math.ethz.ch} 

\author{Katharine Woo}
\address{Department of Mathematics, Stanford University, Stanford, CA 94305}
\email{khwoo98@stanford.edu}

\date{\today,\ \thistime} 


\begin{abstract}
  We survey some of the \emph{stratification theorems} concerning
  exponential sums over finite fields, especially those due to
  Katz--Laumon and Fouvry--Katz, as well as some of their applications.
  Moreover, motivated partly by recent work of Bonolis, Pierce and Woo,
  we prove that these stratification statements admit uniform variants
  in families, both algebraically and analytically.
  \par
  The paper includes an Appendix by Forey, Fresán and Kowalski, which
  provides an elementary intuitive introduction to trace functions in
  more than one variable over finite fields.
\end{abstract}

\maketitle

\begin{flushright}
  \textit{Dedicated to the memory of Gérard Laumon}
\end{flushright}

\setcounter{tocdepth}{1}
\tableofcontents

\section{Introduction}

\subsection{General context}

Exponential sums over finite fields occur in an enormous variety of
problems in number theory. A few notable examples include:
\begin{enumerate}
\item Gauss sums and Jacobi sums, which are closely related to the
  theory of cyclotomic fields, leading for instance to proofs of the
  law of quadratic reciprocity or of Fermat's theorem on primes which
  are sums of two squares of integers (see,
  e.g.,~\cite[Ch.\,8]{ireland-rosen});
\item Kloosterman sums, whose ``omnipresence'' has often been noticed --
  they appear for instance in Fourier coefficients of modular forms, in
  Kloosterman's work on quaternary quadratic forms, or in the proof of
  equidistribution of angles of Gauss sums (see for
  instance~\cite[\S\,14.2,\S\,20.3]{iwaniec-kowalski}
  and~\cite[Ch.\,13]{katz-gkm});
\item Additive character sums in one or many variables play a
  prominent role in the analysis of the ``major arcs'' in the circle
  method (see, e.g.,~\cite[Ch.\,20]{iwaniec-kowalski});
\item Weyl sums that appear in equidistribution problems over finite
  fields (see, e.g.,~\cite[Ch.\,21]{iwaniec-kowalski}).
\end{enumerate}

Abstractly, maybe the most general form of exponential sum over finite
fields would be an expression of the type
\begin{equation}\label{eq-sum-general}
  \sum_{x\in V(k)}t_M(x;k),
\end{equation}
where $k$ denotes a finite field, $V$ is a (non-empty) algebraic
variety\footnote{\ See below for the technical convention we use (a
  separated and reduced scheme of finite type over~$k$); here one can
  think of the zero locus of finitely-many polynomials with coefficients
  in~$k$.}  defined over~$k$ (for instance, the affine $d$-dimensional
space, for which $V(k)=k^d$), and $t_M$ is the ``trace function'' of a
suitable object defined on~$V$.  In the simplest case, this could be a
function of the form
\begin{equation}\label{eq-char-sum}
  t_M(x;k)=\chi(g(x))\psi(f(x)),
\end{equation}
where $f$ and~$g$ are (polynomial) functions on~$V$, while~$\chi$ is a
non-trivial character of the group~$k^{\times}$, extended by taking
the value~$0$ at~$0$, and~$\psi$ is a character of~$k$. The
sum~(\ref{eq-sum-general}) is then one of the classical character sums
\[
  \sum_{x\in V(k)}\chi(g(x))\psi(f(x))
\]
which occur most frequently in applications.

In the great majority of cases,\footnote{\ One notable exception is the
  computation of the sign of quadratic Gauss sums.}  the main goal is to
obtain a non-trivial estimate for the exponential sum, reflecting
oscillations of the values of the trace function as~$x$ varies. The
function~$t_M$ usually happens to be uniformly bounded, say by some
constant~$c\geq 0$ (as in the example above, where it is of
modulus~$\leq 1$), so that the trivial bound is
\[
  \Bigl|\sum_{x\in V(k)}t_M(x;k)\Bigr|\leq c|V(k)|.
\]

As also happens for more general oscillatory sums, we expect significant
cancellation unless the trace function has some specific features which
make it constant or ``almost constant'' on~$V(k)$.  In the context of
sums over finite fields, this expectation can be phrased in much more
precise terms, due to the truly remarkable connections between
exponential sums over finite fields and algebraic geometry, which are
one of the great discoveries of 20th Century number theory, arising
through the work of Hasse, Weil, Grothendieck, Deligne, Katz, Laumon and
others.

The first major step of this connection can be stated as a
``structural'' result: a combination of the Grothendieck--Lefschetz
trace formula and of Deligne's proof of the Riemann Hypothesis over
finite fields~\cite{deligne} leads to the following:

\begin{theorem}[Structure theorem]
  Let~$d\geq 0$ be the dimension of~$V$. There exist non-negative
  integers~$b_i(M)$, defined for~$0\leq i\leq 2d$, and for each~$i$,
  there exist complex numbers~$\alpha_{i,j}$ and integers~$w_{i,j}$
  for $1\leq j\leq b_i(M)$, such that
  \[
    \sum_{x\in V(k)}t_M(x;k)= \sum_{0\leq i\leq 2d}(-1)^i\sum_{1\leq
      j\leq b_i(M)}\alpha_{i,j},
  \]
  and $|\alpha_{i,j}|=|k|^{w_{i,j}/2}$. Moreover, for trace functions
  of the form~\textup{(\ref{eq-char-sum})}, we have $w_{i,j}\leq i$.

  In fact, more generally, for any finite extension $k_n/k$ of
  degree~$n$, we have
  \[
    \sum_{x\in V(k_n)}t_M(x;k_n)= \sum_{0\leq i\leq 2d}(-1)^i\sum_{1\leq
      j\leq b_i(M)}\alpha_{i,j}^n.
  \]
\end{theorem}

The condition $w_{i,j}\leq i$ is valid for character sums but also for
much more general trace functions (those which are ``mixed of
weights~$\leq 0$''); it is a direct reflection of the Riemann Hypothesis
over finite fields in the most general form proved by Deligne.

Assuming that this condition $w_{i,j}\leq i$ holds, the structure
theorem leads to the bound
\[
  \Bigl|\sum_{x\in V(k)}t_M(x;k)\Bigr|\leq C|k|^{w/2},
\]
where
\[
  C=\sum_{i=0}^{2d}b_i(M),\quad\quad w=\max_{ \substack{0\leq i\leq
      2d\\b_i(M)\not=0}}\max_{1\leq j\leq b_i(M)} w_{i,j}.
    % \,\mid\, i\text{
    % satisfies } b_i(M)\not=0,\ 1\leq j\leq b_i(M)\}.
\]

It follows that this approach leads to non-trivial bounds for the
exponential sums, provided two conditions are met:
\begin{enumerate}
\item One can give an upper-bound for the numbers $b_i(M)$ (which are
  called \emph{Betti numbers}) or at least for their sum~$C$;
\item One can prove that $b_i(M)$ is zero for $i$ large, and especially
  (and at least) for~$i=2d$ (since in general the summation set $V(k)$
  will have size of order of magnitude about $|k|^{d}=|k|^{2d/2}$, so
  that having $w=2d$ leads to a trivial bound).
\end{enumerate}

These two questions are somewhat independent, and each of them is a
significant problem. We begin by illustrating the first issue, although
this will not be the main focus of this paper. We define as usual
$e(z)=e^{2i\pi z}$ for $z\in\Cc$. We then consider the sum
\[
  \sum_{x\in \Ff_p}e\Bigl(\frac{x^d}{p}\Bigr),
\]
where~$p$ is a prime number and~$d\geq 1$ is an integer. It is, in this
case, straightforward to obtain the structural result above, using Gauss
sums. Assuming~$p\equiv 1\mods{d}$ for simplicity, so that all $d$-th
roots of unity are in $\Ff_p$, we have a basic formula
\[
  |\{x\in\Ff_p\,\mid\, x^d=y\}|= \sum_{\chi^d=1}\chi(y),
\]
where~$\chi$ runs over characters of~$\Ff_p^{\times}$ of order
dividing~$d$, which can be used to detect the number of solutions of
$x^d=y$.\footnote{\ We extend~$\chi$ to $\Ff_p$ by $\chi(0)=0$, except
  if $\chi$ is trivial, in which case $\chi(0)=1$.} Applying this
formula, we obtain
\[
  \sum_{x\in \Ff_p}e\Bigl(\frac{x^d}{p}\Bigr)=
  \sum_{y\in \Ff_p}e\Bigl(\frac{y}{p}\Bigr)
  |\{x\in\Ff_p\,\mid\, x^d=y\}|
  =\sum_{\substack{\chi^d=1\\\chi\not=1}} \sum_{y\in
    \Ff_p}\chi(y)e\Bigl(\frac{y}{p}\Bigr)
\]
(the trivial character does not appear in the final expression because
its contribution would be~$0$). Each of the sums over~$y$ is a Gauss
sum, and is well-known to have modulus~$p^{1/2}$. The structure theorem
holds here with
\[
  d=1,\quad\quad b_0=b_2=0,\quad\quad b_1=d-1,
\]
where the $\alpha_{1,j}$ are the $d-1$ Gauss sums, each with
$w_{1,j}=1$. Thus we get
\[
  \Bigl|  \sum_{x\in \Ff_p}e\Bigl(\frac{x^d}{p}\Bigr)
  \Bigr|\leq (d-1)p^{1/2}.
\]

This bound is also one of the simplest cases of the Weil bound for
additive character sums in one variable. It is remarkably strong, but
this is really only the case when~$d$ is small enough compared to $p$,
in particular when~$d$ is fixed: note that the estimate becomes trivial
if~$d>\sqrt{p}$.

Similar issues may occur for many exponential sums. Nevertheless, it
turns out that for most sums appearing in solutions of classical
problems of analytic number theory, the constant~$C$ can be efficiently
estimated. The most important issue is what happens when, in a character
sum with summands~(\ref{eq-char-sum}), we take $k=\Ff_p$ and vary the
prime number~$p$ while the polynomial functions are ``defined
over~$\Zz$''. In this context, the key goal is to obtain a bound for~$C$
which is independent of~$p$. This was first achieved for additive
character sums (i.e. with $g=1$) by Bombieri, and generalized by Katz to
all character sums (see~\cite{katz-betti}).

More significantly, however, the recent development of Sawin's
\emph{Quantitative Sheaf Theory} (see the presentation by Sawin with
Forey, Fresán and Kowalski~\cite{qst}) has provided a very general
framework to handle questions related to sums of Betti numbers in
classical problems of analytic number theory, so that this first
important problem arising from the Structure Theorem can now be
considered to be secondary in most cases.

The present paper will be concerned with one specific aspect of the
second problem above, which has to do with providing criteria for the
\emph{vanishing} of certain Betti numbers~$b_i(M)$, especially when~$i$
is large.

We begin however with the case $i=2d$, which is of course of particular
importance, since we obtain a non-trivial bound whenever $w<2d$. It can
indeed be achieved in very general situations.

For instance consider a sum with summands~(\ref{eq-char-sum}), and
assume that $V$ is the affine $d$-dimensional space for some integer
$d\geq 1$, so that we have a ``classical'' (multi-variable) character
sum. One can then show using the basic underlying formalism\footnote{\
  The ``coinvariant formula'' for the topmost Betti number, which is
  part of the foundational theory of étale cohomology.} that $w<2d$ if
one of the following conditions holds:
\begin{enumerate}
\item The additive character~$\psi$ is non-trivial and~$f$ is a
  polynomial function of degree~$<p$, where~$p$ is the characteristic
  of~$k$;
\item The multiplicative character~$\chi$ is of order~$d$ and~$g$ is
  not a $d$-th power of another polynomial function.
\end{enumerate}

This fact can be interpreted in the following remarkable property of
character sums over finite fields: \emph{either} the
summand~(\ref{eq-char-sum}) is literally constant (this is what happens
if properties (1) and (2) fail), \emph{or} there is non-trivial
cancellation, in the sense that
\[
  \Bigl|\sum_{x\in k^d}\chi(g(x))\psi(f(x))\Bigr|\leq C|k|^{d-1/2}.
\]

In applications, the minimal requirement $w<2d$ is often not sufficient
however, and what is needed is the \emph{square-root cancellation}
property, which states that $w\leq d$. This corresponds to the vanishing
property
\begin{equation}\label{eq-square-root}
  b_i(M)=0\text{ for all } i>d,
\end{equation}
and to a bound
\[
  \Bigl|\sum_{x\in V(k)}t_M(x;k)\Bigr|\leq C|k|^{d/2},
\]
which is in many cases essentially optimal (on probabilistic grounds for
instance). In the case of sums with one variable ($V$ is a curve, so
that $d=1$), this condition is the same as $w<2=2d$, and this
coincidence explains why one-variable sums (either character sums, or
sums of trace functions) are so useful in so many problems of analytic
number theory.

In the case of genuinely multi-variable sums, however, the issue is much
more complicated. Elementary examples show that although the condition
$w\leq d$ holds very often, it can sometimes fail (many concrete
examples will appear in Section~\ref{sec-examples}). This direct
observation immediately gives an idea of how to proceed: even
if~(\ref{eq-square-root}) does fail for some particular sum, it is also
a fact that in many applications, one has to deal with more than one
exponential sum of the same type, and it may be enough to prove that
``most'' instances do satisfy~(\ref{eq-square-root}). This leads to what
are known as \emph{stratification} results for families of exponential
sums. This is the main topic of this paper.

Our first goal is to recall some of the known stratification results,
focusing on the results of Fouvry and Katz~\cite{fouvry-katz}, which
built on the foundational pioneering work of Katz and
Laumon~\cite{katz-laumon}. The second goal is to complement these
results by providing \emph{uniform} statements about such
stratifications. These are of importance in recent work of Bonolis,
Pierce and Woo~\cite{bpw}, which provided the motivation to write this
survey.

The exponential sums we will consider are special cases of the general
form above. Precisely, they will be of the form
\[
  \sum_{x\in V(k)}t_M(x;k)\psi(h_1x_1+\cdots+h_nx_n)
\]
where $V$ is an algebraic variety over a finite field~$k$ and the
parameter~$h=(h_i)_{1\leq i\leq n}$ varies over~$\Ff_p^n$. For each
value of~$h$, this expression is an exponential sums in the general
sense above, and the principle described previously, that ``most''
exponential sums satisfy the square-root cancellation
property~(\ref{eq-square-root}) will refer to these parameters~$h$: for
``most'' values of $h$, there is square-root cancellation.

% Such sums occur extremely often in number theory, including in areas
% such as the circle method (...), certain studies of Cohen--Lenstra
% heuristics (...).  A crucial aspect is that the variety~$V$ usually
% has dimension $\geq 2$, and that the key goal is to obtain some form
% of \emph{square-root cancellation}: if~$V$ has dimension~$d$, then we
% would like to have a bound of the type
% \[
%   \sum_{x\in V(\Fp)}\chi(g(x))\psi(f(x)+h_1x_1+\cdots+h_nx_n)\ll
%   p^{d/2}.
% \]

% Such a goal is realistic in view of Deligne's most general form of the
% Riemann Hypothesis over finite fields, and is morally equivalent to
% the vanishing of certain étale cohomology groups, as we will recall
% below. However, it is also well-understood that this cannot usually
% hold for all values of the various parameters, as ``diagonal''
% configurations may exist for which no better bound than
% \[
%   \sum_{x\in V(\Fp)}\chi(g(x))\psi(f(x)+h_1x_1+\cdots+h_nx_n)\ll
%   p^{d-1/2}
% \]
% can be achieved (but conversely, very light non-degeneracy conditions
% imply such a bound, and this may already be useful).

More precisely, the theory of Katz and Laumon, and its refinement by
Fouvry and Katz, shows that there is \emph{always} a ``cascading''
sequence of bounds, starting from the ideal case of square-root
cancellation which holds for ``generic'' choices of~$h$, and then become
weaker and weaker, but only for fewer and fewer values
of~$h$. Crucially, these evolving bounds are controlled by the vanishing
of auxiliary polynomials -- in the language of algebraic geometry,
``generic'' means for a Zariski-dense open set, and ``fewer and fewer''
is measured by parameters lying in subvarieties of higher and higher
codimension. This fundamental structural property of the evolving bounds
is essential in many applications.

\begin{remark}\label{rm-ag}
  In the remainder of this paper, we will use the language of
  algebraic geometry (using the book~\cite{g-w} of Görtz and Wedhorn
  as our primary reference). We just recall here some basic language
  which should help to at least understand some of the basic
  statements.  The Appendix contains some more explanations on an
  intuitive level for trace functions and sheaf theory.
  \par
  For $n\geq 1$, a \emph{closed subscheme} $X$ of~$\Aa^n_{\Zz}$ is the
  data of the solution sets $X(A)$, for any commutative ring~$A$, of a
  finite system of polynomial equations in $n$ variables with integral
  coefficients. A \emph{locally closed subscheme} $X$ of~$\Aa^n_{\Zz}$
  is defined similarly, except that in addition to equations $F(x)=0$,
  one allows (finitely many) inequalities $G(x)\not=0$, with $G$ again
  an integral polynomial in~$n$ variables.
  
  Looking at the reductions modulo a prime~$p$ of the defining
  equations and inequations, we obtain the corresponding subschemes
  over the finite field~$\Ff_p$, and in particular the sets $X(\Ff_p)$
  of rational points.

  There is an algebraic notion of dimension for~$X$, which has some of
  the expected properties, except that (in the case of~$\Aa^n$
  itself), we have~$\dim(\Aa^n_{\Zz})=n+1$, which takes into account
  the fact that~$\Zz$ itself plays the role of a one-dimensional
  object in algebraic geometry (its points being essentially the prime
  numbers); the \emph{relative dimension} corrects this, so
  that~$\Aa^n_{\Zz}$ has relative dimension~$n$, as expected. More
  precisely, for a scheme $X$ over a base~$Y$ (which will usually just
  be the spectrum of~$\Zz$ or~$\Zz[1/N]$ for some~$N$), the
  \emph{relative dimension} of~$X$ over~$Y$ (or just \emph{relative
    dimension}, if~$Y$ is clear from context) is the maximum of the
  dimension of the fibers of~$X\to Y$ (each fiber being a scheme over
  the residue field at a point of~$Y$). If all fibers are of the same
  dimension (i.e., if~$Y\to X$ is equidimensional), then the relative
  dimension is equal to $\dim(X)-\dim(Y)$, at least in the cases of
  interest to us (see, e.g.,~\cite[Cor.\,14.116]{g-w}). We will
  sometimes write $\dim_{Y}(X)$ for the relative dimension.

  We will also use the fact that for a scheme~$X$ of finite-type and
  dominant over~$\Zz$ (meaning that its reduction modulo primes are
  non-empty for all but finitely many primes), the ``generic''
  dimension of the reduction $X_{\Fp}$ of~$X$ modulo~$p$ (which we
  will also sometimes denotes $X\otimes\Fp$)
  % ; see for instance~\cite[Cor.\,14.116\,(1)]{g-w} for this fact. It
  is also the dimension of the complex algebraic variety $V_{\Cc}$.
\end{remark}

\subsection{Stratification statements}

We begin this paper by recalling the statement of Theorem 1.1 %%and 1.2
in the paper of Fouvry and Katz~\cite{fouvry-katz} (compare
with~\cite[Prop.\,1.0]{fouvry}).
% (TODO: explain in concrete terms all algebraic-geometric language.)

\begin{theorem}[Fouvry--Katz]\label{th-fk}
  Let $n$ and $d$ be positive integers. Let $V$ be a locally closed
  subscheme of $\Aa^n_{\Zz}$ such that $\dim(V_{\Cc})\leq d$. Let
  $f\in\Zz[x_1,\ldots,x_n]$ be given.
  \par
  Then there exist constants $C$ and~$N$, depending on $(n,d,V,f)$,
  closed subschemes $X_j\subset \Aa^n_{\Zz}$ for $1\leq j\leq n$, of
  relative dimension $\leq n-j$, such that
  $$
  \Aa^n_{\Zz}\supset X_1\supset\cdots\supset X_n
  $$
  with the following property: for any invertible function $g$ on $V$,
  for any prime number $p\nmid N$ such that $V$ modulo~$p$ has
  dimension~$\leq d$, for any $h\in (\Aa^n-X_j)(\Fp)$, for any
  non-trivial additive character $\psi$ of $\Fp$ and for any
  multiplicative character $\chi$ of $\Fpt$, we have
  \begin{equation}\label{eq-sum}
    \Bigl|\sum_{x\in
      V(\Fp)}\chi(g(x))\psi(f(x)+h_1x_1+\cdots+h_nx_n)\Bigr|
    \leq Cp^{d/2+(j-1)/2}.
  \end{equation}
\end{theorem}

\begin{remark}
  (1) We see for instance that if $h\notin X_1(\Ff_p)$, which is a
  ``generic'' condition, we have the square-root cancellation bound
  \[
    \Bigl|\sum_{x\in
      V(\Fp)}\chi(g(x))\psi(f(x)+h_1x_1+\cdots+h_nx_n)\Bigr| \leq
    Cp^{d/2},
  \]
  recalling that if $V$ modulo~$p$ has dimension~$d$, then the number
  of $\Ff_p$-points is of order of magnitude $p^d$.

  (2) The actual statement in the paper of Fouvry and Katz does not
  require the restriction on the dimension of $V$ modulo~$p$, which is
  satisfied for all but finitely many primes. We will see however that
  it is convenient to incorporate this condition for the later uniform
  statements.
\end{remark}

We are interested in providing uniform and quantitative versions of this
stratification result when the various data (especially $V$ and $f$)
vary. In order to facilitate the discussion, we will use the following
terminology: given the basic data $(V,f)$ (which implicitly also fixes
$n$ and~$d$), a \emph{Katz--Laumon stratification datum} (abbreviated
KL-datum) for an invertible function $g\colon V\to\Gg_m$ is a triple
$(\mathcal{X},N,C)$ where $N\geq 1$ is an integer, $C>0$ is a real
number, and $X_j\subset \Aa^n_{\Zz}$ are closed subschemes of relative
dimension $\leq n-j$ over~$\Zz[1/N]$, such that
\[
  \Aa^n_{\Zz}\supset X_1\supset\cdots\supset X_n
\]
and such that the estimate~(\ref{eq-sum}) holds in the situations
indicated in Theorem~\ref{th-fk}.  Thus Theorem~\ref{th-fk} states the
\emph{existence} of a KL-datum for any pair $(V,f)$ and any~$g$.

\begin{remark}
  In fact, the result states the stronger fact that there is a
  \emph{common} KL-datum which ``works'' for all~$g$, but below when
  considering families of varieties, it will be necessary to restrict
  the choices of~$g$, and for this reason we find it cleaner to quantify
  separately over the function~$g$. This is further justified by the
  fact that, in many applications, this uniformity with respect to~$g$
  is not really exploited.
\end{remark}

% . In fact, it will be
% necessary below to sometimes restrict the range of functions $g$ which
% are allowed, and we would then speak of a \emph{KL-datum of $(V,f)$ for
%   functions $g$ satisfying some extra conditions}.
% \item Given a KL-stratification $(\mathcal{X}$ of $(V,f)$, the infimum
%   of all constants $C\geq 0$ for which~(\ref{eq-sum}) holds for the
%   indicated parameters is denoted $\norm{\mathcal{X}}$.
% \end{itemize}

The quantitative problems we want to address are the following:
\begin{enumerate}
\item Given $(V,f)$ (and possibly~$g$), can one bound the ``complexity''
  of a KL-datum $(\mathcal{X},N,C)$ of~$(V,f)$?  Concretely, this means
  finding a bound for $N$, $C$ and for the coefficients of a family of
  polynomials whose common zeros define the subvarieties $X_j$, in terms
  of a bound on the size of the coefficients of the equations which
  define~$V$.
\item If we have an ``algebraic'' family of data $(V_a,f_a)$,
  parameterized by $a\in\Zz$ (or some other parameter space), how does a
  KL-datum $(\mathcal{X}_a,N_a,C_a)$ for $(V_a,f_a)$ vary with~$a$?
  Here, the way $g$ varies with~$a$ will need some care.
\end{enumerate}

We will explain briefly in Section~\ref{sec-motivation} the original
motivation for these questions. Our main goal in this paper is to answer
them, at least in some cases. Although there might seem to be little
connection between the two questions, and the second may seem possibly
less relevant to analytic number theory than the first, the key idea
behind the proofs of the quantitative bound for the complexity of the
stratifications will be deduced from the algebraic answer to the second
question.

Our first main theorem concerns the algebraic uniformity. First we
introduce the setting.  We consider non-negative integers~$n$, $d$
and~$r$ and a locally closed subscheme
\[
  W\subset \Aa^n_{\Zz}
\]
of relative dimension $d+r$ over~$\Zz$ (i.e., with $\dim_{\Zz}(W)=d+r$),
given with a morphism $\Delta\,:\, W\lra \Aa^r_{\Zz}$.  For any
$a\in \Aa^r$, we denote by~$V_a$ the fiber $\Delta^{-1}(a)$ of~$\Delta$
over~$a$. Informally (at least if $\Delta$ is surjective, or dominant),
the family of fibers $(V_a)$ is an ``$r$-parameter family of
varieties'', parameterized by~$a\in \Aa^r$.  If~$a\in\Zz^r$, we can
therefore apply Theorem~\ref{th-fk} to each $V_a$. We ask if the
corresponding KL-datum also varies algebraically.

% We call the data of $(W,\Delta)$ an
% \emph{$r$-parameter algebraic family}.
% such that $\dim V_{a,\Cc}\leq d$ for
% all $a$, where $V_a=\Delta^{-1}(a)$ is the fiber of $\Delta$ above $a$
% (viewed as a point of $\Aa^1_{\Zz}$). We may then apply
% Theorem~\ref{th-fk} for any $a\in\Zz$ to the data $(V_a,f)$.

The statement we will prove in this setting is the following.

\begin{theorem}[Algebraically uniform
  KL-stratifications]\label{th-algebraic-dependency}
  Let $W$ be as above, with $W$ of relative dimension $d+r$ over~$\Zz$,
  and let $\Delta\colon \Aa^n_{\Zz}\to \Aa^r_{\Zz}$ be a morphism such
  that the restriction of $\Delta$ to~$W$, still denoted $\Delta$, is
  dominant.
  % and $\Delta$ dominant. Assume that $\Delta$ is the restriction to
  % $W$ of a morphism $\Delta\,:\, \Aa^n_{\Zz}\lra \Aa^r_{\Zz}$.

  Let~$f\colon W\to \Aa^1_{\Zz}$ be a function on~$W$.  Let
  $g\colon W\to\Gg_m$ be an invertible function on~$W$.
  
  There exist 
  \begin{enumth}
  \item An integer $N\geq 1$ and a real number~$C>0$,
  \item Closed subschemes $Y_j\subset \Aa^{r+n}_{\Zz}$ defined for
    $1\leq j\leq r+n$ with
    \[
      \dim_{\Zz}(Y_j)\leq r+n-j
    \]
    such that
    \[
      \Aa^{r+n}_{\Zz}\supset Y_1\supset \cdots \supset Y_{r+n},
    \]
  \item A proper closed subscheme $A$ of~$\Aa^r_{\Qq}$,
  \item A non-zero polynomial~$\varphi\in \Zz[X_1,\ldots,X_r]$,
  \end{enumth}
  which satisfy the following: defining closed subschemes $X_{j,a}$ for
  $a\in\Aa^r_{\Zz}$ by
  \[
    X_{j,a}=\{h\in \Aa^{n}_{\Zz}\,\mid\, (a,h)\in Y_j\}\subset
    \Aa^n_{\Zz},
  \]
  then for all $a\in \Zz^r$ outside of~$A$, the integer $\varphi(a)$
  is non-zero and
  $((X_{j,a})_{j\leq n},N\varphi(a),C)$ is a KL-datum
  for~$(V_a,f|V_a)$ and for the invertible
  function~$g_a\colon V_a\to \Gg_m$ obtained by restriction of~$g$
  to~$V_a$.

  Furthermore, one can find a single tuple $((Y_j),N,C,A,\varphi)$ for
  which the above holds for all invertible
  functions~$g\colon W\to \Gg_m$.
  % \footnote{\ EK: check issue with which functions $g$ can be put in
  % the multiplicative character.}  and with the property that
  % Theorem~\ref{th-fk} holds for $(V_a,f)$ with the closed subschemes
  % $$
  % X_j(a,f)=\{h\in \Aa^{n}_{\Zz}\,\mid\, (a,h)\in Y_j\}
  % $$
  % for all but finitely many $a\in \Zz$, the exceptional $a$ depending
  % only on $(W,\Delta,f)$, provided the sums
  % $$
  % \sum_{x\in V_a(\Fp)}\chi(g(x))\psi(f(x)+h_1x_1+\cdots+h_nx_n)
  % $$
  % are considered only for $g\,:\, W\lra \Aa^1_{\Zz}$ invertible on $W$.
  % \par
  % In particular, in this situation, we have $|X_j(a,f)(\Fp)|\ll p^{n-j}$
  % where the implied constant depends only on $(W,\Delta,f)$.
\end{theorem}

\begin{remark}
  % (1) The meaning of the last restriction on~$g_a$ is that we obtain
  % bounds for
  % \[
  %   \sum_{x\in V_a(\Fp)}\chi(g_a(x))\psi(f(x)+h_1x_1+\cdots+h_nx_n)
  % \]
  % \emph{provided} $g_a$ is the restriction of an invertible function
  % defined on \emph{all} of~$W$. Note that we can certainly extend any
  % invertible function on~$V_a$ to a function on~$W$, since these are
  % just polynomials, but we cannot necessarily ensure that there is an
  % invertible extension. On the other hand, taking $g=1$ is always
  % possible.
  (1) Concretely, this means (in particular) that there exist two
  non-zero polynomials $F_1\in\Zz[a]$ (in $r$ variables) and
  $F_2\in\Zz[a,h]$ (in $r+n$ variables) such that we have square-root
  cancellation
  \[
    \Bigl| \sum_{\substack{x\in
        W(\Fp)\\\Delta(x)=a}}\chi(g(x))\psi(f(x)+h\cdot x)\Bigr|\leq
    Cp^{d/2}
  \]
  if $\psi$ is a non-trivial additive character modulo~$p$, for
  \begin{itemize}
  \item all~$a$ such that $F_1(a)\not=0$,
  \item all primes~$p$ such that $V_a$ modulo~$p$ has dimension~$d$ and
    such that $p\nmid\varphi(a)$,
  \item and for all~$h\in \Fp^r$ sch that $F(a,h)\not=0$.
  \end{itemize}

  (Recall from our discussion of Theorem~\ref{th-fk} that the bounds for
  exponential sums are restricted to primes~$p$ for which $V_{a,\Fp}$
  has dimension~$d$.)  Moreover, the polynomials $F_1$, $F_2$ and
  $\varphi$ can be chosen to be independent of the function~$g$, which
  is invertible on all of~$W$.

  (2) The requirement that $V_{a,\Fp}$ has dimension~$d$ is intuitively
  necessary. More precisely, if $V_{a,\Fp}$ has higher dimension (which
  is the only possibility, since $d$ is the ``generic'' dimension), then
  an estimate like
  \[
    \Bigl| \sum_{\substack{x\in
        W(\Fp)\\\Delta(x)=a}}\chi(g(x))\psi(f(x)+h\cdot x)\Bigr|\leq
    Cp^{d/2}
  \]
  would be ``better'' than square-root cancellation, and therefore
  unlikely to hold. A similar phenomenon explains the requirement that
  $p$ does not divide some integer (namely $\varphi(a)$) depending
  on~$a$: this is similar to the restriction that $p$ be coprime to~$a$
  for the Weil bound
  \[
    \Bigl| \sum_{x\in\Ff_p^{\times}} \psi(ax+x^{-1})\Bigr|\leq 2\sqrt{p}
  \]
  to hold.

  (3) We restrict the indices to define $(X_{j,a})$ to $j\leq n$ instead
  of $j\leq n+r$ because, for $a$ outside of~$A$, the relative dimension
  of $X_{j,a}$ would be $<0$, and hence these are empty.  In further
  statements below, we will omit the restriction $j\leq n$ for
  simplicity.
\end{remark}

We believe that this theorem will be useful in a number of
applications. However, in the motivating problem of Bonolis, Pierce and
Woo, and for the more general question of analytic uniformity, it is not
sufficient because one wants to handle \emph{all} parameters, and not
only those outside of a codimension one subscheme.  To handle this, we
will need the following generalization of the previous theorem, where
the parameter space is not assumed to be all of~$\Aa^r_{\Zz}$, but can
possibly be smaller. Apart from this, the statement is identical.

\begin{theorem}[General algebraically uniform
  KL-stratifications]\label{th-general-kl}
  Let $W\subset \Aa^n_{\Zz}$ be a locally closed subscheme, let
  $\mathbf{M}\subset\Aa^r_{\Zz}$ be a reduced closed subscheme and
  $\Delta\colon W\to \mathbf{M}$ a dominant morphism. Assume that $W$
  is of relative dimension $d+\dim(\mathbf{M})$ over~$\Zz$.
  Let~$f\colon W\to \Aa^1_{\Zz}$ be a function on~$W$.  Let
  $g\colon W\to\Gg_m$ be an invertible function on~$W$.

  There exist
  \begin{enumth}
  \item An integer $N\geq 1$ and a real number~$C>0$,
  \item Closed subschemes $Y_j\subset \mathbf{M}\times \Aa^{n}_{\Zz}$
    defined for $1\leq j\leq n+\dim_{\Zz}(\mathbf{M})$ with
    \[
      \dim_{\Zz}(Y_j)\leq \dim_{\Zz}(\mathbf{M})+n-j
    \]
    such that
    \[
      \mathbf{M}\times \Aa^{n}_{\Zz}\supset Y_1\supset \cdots \supset
      Y_{\dim_{\Zz}(\mathbf{M})+n},
    \]
  \item A proper closed subscheme $A$ of~$\mathbf{M}_{\Qq}$,
  \item A non-zero polynomial~$\varphi\in \Zz[X_1,\ldots,X_r]$,
  \end{enumth}
  which satisfy the following: defining closed subschemes $X_{j,a}$ for
  $a\in \mathbf{M}$ by
  \[
    X_{j,a}=\{h\in \Aa^{n}_{\Zz}\,\mid\, (a,h)\in Y_j\}\subset
    \Aa^n_{\Zz},
  \]
  then for all $a\in \mathbf{M}(\Qq)\cap \Zz^r$ outside of~$A$, the
  integer $\varphi(a)$ is non-zero and the triple
  $((X_{j,a})_{j\leq n},N\varphi(a),C)$ is a KL-datum
  for~$(V_a,f|V_a)$ and for the function~$g_a\colon V_a\to \Gg_m$
  obtained by restriction of~$g$ to~$V_a$.

  Furthermore, one can find a single tuple $((Y_j),N,C,A,\varphi)$ for
  which the above holds for all invertible
  functions~$g\colon W\to \Gg_m$.
  % There exist integers $N$ and~$C$ and closed subschemes $Y_j$ of
  % $\mathbf{M}\times \Aa^{n}_{\Zz}$ for $1\leq j\leq n+\dim(\mathbf{M})$,
  % with relative dimension $\leq n+\dim(\mathbf{M})-j$,
  % such that
  % \[
  %   \mathbf{M}\times\Aa^{n}_{\Zz}\supset Y_1\supset \cdots \supset
  %   Y_{n+\dim(\mathbf{M})}
  % \]
  % which satisfy the following property: defining closed subschemes
  % \[
  %   X_{j,a}=\{h\in \Aa^{n}_{\Zz}\,\mid\, (h,a)\in Y_j\}\subset
  %   \Aa^n_{\Zz}
  % \]
  % for~$a\in\mathbf{M}$, then for all $a\in\mathbf{M}$ outside of a
  % codimension~$1$ subvariety of~$\mathbf{M}$, the triple
  % % $\mathcal{X}_{a}=(X_{j,a})$ of closed subschemes defined by
  % % \[
  % %   X_{j,a}=\{h\in \Aa^{n}_{\Zz}\,\mid\, (h,a)\in Y_j\}\subset
  % %   \Aa^n_{\Zz}
  % % \]
  % $((X_{j,a})_j,N,C)$ is a KL-datum for~$(V_a,f|V_a)$, for invertible
  % functions~$g_a\colon V_a\to \Gg_m$ which are restrictions of
  % invertible functions $g\colon W\to \Gg_m$.
\end{theorem}

In fact, and this is crucial for the motivating application, we need to
go further by allowing more general summands than $\chi(g(x))$ in the
exponential sums.

More precisely, we want to allow trace functions ``adapted'' to a given
stratification of the summation space~$V$. This means the following: we
are given a prime number~$\ell$, a stratification $\mcV=(V_i)$ of~$V$
(this is simply a set-theoretic partition of $V$ into finitely many
reduced locally closed-subschemes, see~\cite[p.\,120]{fouvry-katz}) and
an object~$L$ of the derived category $D^b_c(V[1/\ell],\bQl)$\footnote{\
  There is a short intuitive introduction to the objects of the derived
  category $D^b_c(V[1/\ell],\bQl)$ in Section~\ref{ssec-weight}.} which
is adapted to~$\mcV$, in the sense that all of its cohomology sheaves
of~$L$ are lisse on all strata of~$\mcV$
(see~\cite[p.\,120]{fouvry-katz}). We then consider the family of sums
\[
  \sum_{x\in V(\Fp)}t_L(x)\psi(f(x)+h\cdot x),
\]
for $h\in\Aa^n(\Fp)=\Ff_p^n$, and we want to obtain stratification
statements for this sum similar to those in Theorem~\ref{th-fk}.

\begin{remark}
  This setting does indeed generalize the situation in
  Theorem~\ref{th-fk} (and the proofs of the results of Fouvry--Katz go
  through this step): if we take the ``trivial''
  stratification~$\mcV=\{V\}$ and the object~$L$ to be the single sheaf
  $[x\mapsto x^a]^*g^*\bQl$ (viewed as a complex by being in
  degree~$0$), then the trace function of $L$ is a combination of
  $\chi(g(x))$ for all multiplicative characters of order
  dividing~$a$. (Such an object is adapted to~$\mcV$ because it ``is''
  its unique non-zero cohomology sheaf, and because $g$ is invertible.)
  % because $\mcL_{\chi(g)}$ is the unique non-zero
  % cohomology sheaf in that case, and is lisse on~$V$ since~$g$ is
  % invertible.
\end{remark}

Thus we generalize the notion of Katz--Laumon stratification as
follows.\footnote{\ The terminology ``semiperverse'', ``mixed'',
  ``direct factor'' are explained, at least intuitively, in the Appendix
  (see Sections~\ref{ssec-perv},~\ref{ssec-weight-lisse}
  and~\ref{ssec-weight},~\ref{ssec-trace}, respectively).
  % , whereas the adjective ``fiberwise'' means that the condition holds
  % for all fibers over finite fields of characteristic different from
  % $\ell$ satisfy the stated property (see~\cite[\S\,3]{fouvry-katz}).
  Examples will also
  help to clarify the meaning.}

\begin{definition}[Generalized KL-datum]\label{def-kl-datum}
  Let $(V,f,\mcV)$ be as above, with $\dim_{\Zz}(V)=d$, let~$\ell$ be a
  prime number and let~$K$ be an object of $D^b_c(V[1/\ell],\bQl)$
  adapted to $\mcV$.  A KL-datum for $(V,f,\mcV)$ and $K$ is a triple
  $(\mcX,N,C)$, where $N\geq 1$ is an integer, $C>0$ is a real number
  and $\mcX=(X_j)$ is a family of closed subschemes of~$\Aa^n$ with
  $\dim_{\Zz[1/N]}(X_j)\leq n-j$ and
  \[
    \Aa^n_{\Zz}\supset X_1\supset \cdots\supset X_n,
  \]
  such that the following holds:
  \begin{itemize}
  \item for any prime number $p\nmid \ell N$ such that $V\otimes \Fp$
    has dimension~$d$, and such that~$K\otimes \Fp$ is semiperverse of
    weights~$\leq 0$,
  \item for any direct factor~$L$ of~$K\otimes\Fp$,
  \item for any $h\in (\Aa^n-X_j)(\Fp)$,
  \item for any non-trivial additive character $\psi$ of $\Fp$,
  \end{itemize}
  we have
  \[
    \Bigl|\sum_{x\in V(\Fp)}t_L(x)\psi(f(x)+h_1x_1+\cdots+h_nx_n)\Bigr|
    \leq Cp^{(j-1)/2}.
  \]
  % If these properties hold only for $K$ in a restricted class of
  % objects, we will say that we have a KL-datum for this class of
  % objects.
\end{definition}

\begin{remark}\label{rm-normalize}
  The form of the upper-bound seems different from what we discussed
  earlier in the case of multiplicative character sums. This is because
  of the requirement that~$K$ and~$L$ be semiperverse of
  weights~$\leq 0$; for the case of $\chi(g(x))$, this implies that the
  corresponding object~$L$ is not the sheaf $\mcL_{\chi(g)}$ in
  degree~$0$, but the shifted and twisted object
  \[
    \mcL_{\chi(g)}[d](d/2)
  \]
  (i.e., a twist of~$\mcL_{\chi(g)}$ viewed as a complex by being a
  single sheaf in degree~$-d$), which has trace function
  \[
    (-1)^{d}p^{-d/2}\chi(g(x)).
  \]

  In fact, this also explains the requirement that $V\otimes \Fp$ has
  dimension~$d$, since otherwise this object would \emph{fail} to be
  semiperverse.

  See also Remark~\ref{rm-chig} below.
  % (2) As in the statement of Theorem~\ref{th-algebraic-dependency}, we
  % may look at subsets of the potential objects~$K$, when stating uniform
  % statements. So we will speak, say, of a KL-datum ``for a certain class
  % of objects'' (or, by abuse of notation, of trace functions), meaning
  % that the properties above hold when~$K$ is restricted to be one of
  % these objects.
\end{remark}

Fouvry and Katz deduced Theorem~\ref{th-fk} from a theorem
(see~\cite[Th.\,3.1]{fouvry-katz}) which actually applies to provide a
KL-datum for a given stratification $\mcV$ (in fact one which applies to
all objects~$K$ subject to a suitable restriction of being ``fiberwise
semiperverse'').

We will establish that these more general stratifications also vary
algebraically in families. This however requires some restriction on the
objects $K$ which we consider, because the semiperversity condition
modulo~$p$ for an object~$K$ on the total space~$W$ of a family may not
imply that the restrictions~$K_a$ on~$V_a$ remain
semiperverse.\footnote{\ Up to a necessary uniform shift.} The following
definition is ad-hoc for this purpose.

\begin{definition}[Transverse semiperverse objects]
  \begin{enumerate}
  \item Let~$k$ be a finite field. Let $W$ be an algebraic variety
    over~$k$ and $V\subset W$ a closed subvariety. Let~$\ell$ be a prime
    number invertible in~$k$. A semiperverse object~$K\in D^b_c(W,\bQl)$
    is said to be \emph{$V$-transverse} if the object
    \[
      (K|V)[-\codim(V)]
    \]
    is semiperverse on~$V$.
    
  \item Let~$N\geq 1$ be an integer. Let $W$ be a finite-type scheme
    over~$\Zz[1/N]$ and $V\subset W$ a closed subscheme. An
    object~$K\in D^b_c(W,\bQl)$ is said to be \emph{fiberwise
      $V$-transverse} if, for all primes $p\nmid N$ such that
    $K\otimes \Fp$ is semiperverse, this object is
    $V\otimes\Fp$-transverse.
  \end{enumerate}
\end{definition}

This property holds, and is rather easy to check, in many cases of
interest, as we now illustrate.

\begin{example}\label{ex-transverse}
  Let $V\subset W$ be a closed subvariety of an irreducible algebraic
  variety $W$ over a finite field~$k$.

  (1) If $g$ is a function invertible on~$W$, and $\chi$ is a
  multiplicative character of~$k^{\times}$, then the shifted Kummer
  sheaf $\mcL_{\chi(g)}[\dim(W)]$ is semiperverse, and is
  $V$-transverse. Indeed, we have
  \[
    (\mcL_{\chi(g)}[\dim(W)]|V)[-\codim(V)]=\mcL_{\chi(g\mid
      V)}[\dim(V)],
  \]
  which is semiperverse.

  (2) In fact, quite generally, let $\mcF$ be a constructible sheaf
  on~$W$. %% with support equal to~$W$.
  Then $K=\mcF[\dim(W)]$ (placed in degree $-\dim(W)$) is semiperverse,
  and it is~$V$-transverse as soon as the support of the restriction
  $K|V$ is also equal to~$V$. Indeed, we get as before
  \[
    (K|V)[-\codim(W)]=(\mcF|V)[\dim(V)],
  \]
  which is semiperverse under our assumption.
  %% TODO: check support

  (3) On some intuitive level, as explained at the beginning of
  Section~\ref{ssec-perv}, to say that $K$ on~$W$ is semiperverse (and
  mixed of weights $\leq 0$) amounts to saying that
  \[
    \sum_{x\in W(k)}|t_K(x)|^2\asymp 1.
  \]

  If the values of~$t_K$ are ``well-distributed'', we then expect that
  for $V\subset W$, of codimension~$r\geq 0$, we should have
  \[
    \sum_{x\in V(k)}|t_K(x)|^2\asymp |k|^{-r},
  \]
  %% |t(x)| about q^{-dim(W)/2} --> sum of size about q^{dim(V)-dim(W)}
  and if this is so, then $(K|V)[-r](-r/2)$ should be semiperverse of
  weights $\leq 0$. The definition of $V$-transverse objects captures
  this intuition.
\end{example}

% This can be achieved under the assumption that~$L$ is
% fiberwise semiperverse and mixed of weights~$\leq 0$. More precisely,
% the most general algebraic uniformity statement we will prove is the
% following.

We can now state a generalization of Theorem~\ref{th-general-kl} to
trace functions.

\begin{theorem}[Algebraically uniform KL-stratifications for trace
  functions]\label{th-general-kl-trace}
  Let $W\subset \Aa^n_{\Zz}$ be a locally-closed subscheme,
  $\mathbf{M}\subset \Aa^r_{\Zz}$ a reduced closed subscheme of
  relative dimension $d+\dim_{\Zz}(\mathbf{M})$ and
  $\Delta\colon W\to \mathbf{M}$ a dominant morphism.  Let~$\mcW$ be a
  stratification of~$W$.  Let~$f\colon W\to \Aa^1_{\Zz}$ be a function
  on~$W$.

  Let~$\ell$ be a prime number and~$K$ an object of
  $D^b_c(W[1/\ell],\bQl)$. Let~$N_1\geq 1$ be an integer and
  $\mathbf{M}_1\subset \mathbf{M}$ a subscheme of codimension $\geq 1$
  such that
  \begin{enumth}
  \item[\emph{(a)}] $K\otimes \Fp$ is semiperverse of weights~$\leq 0$ on
    $W\otimes\Fp$ for all $p\nmid N_1$,
  \item[\emph{(b)}] $K$ is fiberwise $V_b$-transverse for all~$b$
    outside $\mathbf{M}_1$.
  \end{enumth}
  
  There exist
  \begin{enumth}
  \item An integer $N\geq 1$, divisible by $N_1$, and a real
    number~$C>0$,
  \item Closed subschemes
    $Y_j\subset \mathbf{M}\times \Aa^{n}_{\Zz}$ defined for
    $1\leq j\leq n+\dim_{\Zz}(\mathbf{M})$ with
    \[
      \dim_{\Zz}(Y_j)\leq \dim_{\Zz}(\mathbf{M})+n-j
    \]
    such that
    \[
      \mathbf{M}\times \Aa^{n}_{\Zz}\supset Y_1\supset \cdots \supset
      Y_{\dim_{\Zz}(\mathbf{M})+n},
    \]
  \item A closed subscheme $A$ of~$\mathbf{M}_{\Qq}$ of codimension
    $\geq 1$, containing $\mathbf{M}_{1,\Qq}$,
  \item A non-zero polynomial~$\varphi\in \Zz[X_1,\ldots,X_r]$,
  \end{enumth}
  which satisfy the following: defining closed subschemes $X_{j,a}$ for
  $a\in \mathbf{M}$ by
  \[
    X_{j,a}=\{h\in \Aa^{n}_{\Zz}\,\mid\, (h,a)\in Y_j\}\subset
    \Aa^n_{\Zz},
  \]
  and stratifications $\mcV_a$ of $V_a$ by
  \[
    \mcV_a=\{W_j\cap V_a\,\mid\, W_j\in \mcW\},
  \]
  then for all $a\in \mathbf{M}(\Qq)\cap\Zz^r$ outside of~$A$, the
  integer $\varphi(a)$ is non-zero and the triple
  $((X_{j,a})_{j\leq n},N\varphi(a),C)$ is a KL-datum for the restricted
  data $(V_a,f|V_a,\mcV_a)$ and for the
  object~$(K_a|V_a)[-\codim(V_a)]$.
  % and for
  % the function~$g_a\colon V_a\to \Gg_m$ obtained by restriction of~$g$
  % to~$V_a$.

  % There exist integers~$N$ and~$C$, and for
  % $1\leq j\leq n+\dim(\mathbf{M})$ there exist closed subschemes~$Y_j$
  % of $\mathbf{M}\times \Aa^{n}_{\Zz}$ with relative dimension
  % $\leq n+\dim(\mathbf{M})-j$ such that
  % \[
  %   \mathbf{M}\times\Aa^{n}_{\Zz}\supset Y_1\supset \cdots \supset
  %   Y_{n+\dim(\mathbf{M})}
  % \]
  % with the following properties. Define
  % \begin{gather*}
  %   X_{j,a}=\{h\in \Aa^{n}_{\Zz}\,\mid\, (h,a)\in Y_j\}\subset
  %   \Aa^n_{\Zz}
  %   \\
  %   \mcV_a=\{W_j\cap V_a\,\mid\, W_j\in \mcW\},
  % \end{gather*}
  % for $a\in\mathbf{M}$.

  % Let~$N_1\geq 1$ be an integer and let~$\mathbf{M}_1$ be a closed
  % subvariety of~$\mathbf{M}$ of codimension $\geq 1$. 
  
  % Then, for all~$a\in\mathbf{M}$ outside of a codimension~$1$ subvariety
  % of~$\mathbf{M}$, the triple $((X_{j,a})_j,NN_1,C)$ is a KL-datum
  % for~$(V_a,f|V_a,\mcV_a)$, restricted to the objects on~$V_a$ which
  % arise by restriction of some object~$K$ on~$W$ such that
  % \begin{enumth}
  % \item $K\otimes \Fp$ is semiperverse of weights~$\leq 0$
  %   for all $p\nmid N_1$,
  % \item $K$ is fiberwise $V_b$-transverse for all~$b$ outside
  %   $\mathbf{M}_1$.
  % \end{enumth}

  Moreover, we can find a single tuple $((Y_j),N,C,A,\varphi)$ which
  applies to all objects~$K$ for which the properties~\emph{(a)}
  and~\emph{(b)} above hold for given integer~$N_1$ and subscheme
  $\mathbf{M}_1$.
\end{theorem}

\begin{remark}
  The point of the introduction of the extra parameter $N_1$ is that
  many ``natural'' objects $K$ on $W$ will be semiperverse modulo~$p$
  for all but finitely many primes. This is a general fact: for
  instance, consider the object $K=\bQl[1]$ on an irreducible
  subscheme~$X\subset \Aa^2_{\Zz}$. Then the object $K\otimes \Fp$ is
  semiperverse modulo~$p$ if $X\otimes\Fp$ is of dimension~$1$, but is
  not if $X\otimes\Fp$ is the same as $\Aa^2_{\Fp}$. The meaning of our
  statement is that it is enough then to avoid such primes in
  considering exponential sums.
\end{remark}

% \begin{remark}
%   The precise meaning of the last restriction is the following: the
%   properties of the stratification $\mcV_a$ apply only to objects~$K$
%   on~$V_a$ of the form~$K=\widetilde{K}|V_a$, where~$\widetilde{K}$ is
%   an object of~$D^b_c(W,\bQl)$.
% \end{remark}

We have now stated a variety of \emph{algebraic} uniformity
statements. We will show how to use them to deduce our other main
results, concerning \emph{analytic} uniformity statements.

To state these results, we introduce a notion of height for varieties
and stratifications. First, define
\[
  \log^+(x)=\max(0,\log(x))
\]
for $x\geq 0$. The coefficient height of an integral polynomial
in finitely many variables $(T_i)$ is
\[
  h_c\Bigl(\sum_{i,j}a_{i,j}T_i^j\Bigr)=\log^+ \max_{i,j} |a_{i,j}|,
\]
and the coefficient height of a finite family $(f_j)_{j\in J}$ of
polynomials (in the same variables) is
\[
  h_c((f_j)_{j\in J})=\max_{j\in J}(h_c(f_j)).
\]

For a closed subscheme $X\subset \Aa^n_{\Zz}$, the height of~$X$ is
defined to be
\[
  h(X)=\deg(X)+\inf\{h_c((f_j)_{j\in J})\,\mid\, \text{the common zero
    set of the $f_j$ is~$X$}\}.
\]

Finally, for a finite family $\mathcal{X}=(X_j)$, we denote
\[
  h(\mathcal{X})=\max_{j} h(X_j).
\]

By considering ``universal'' families of subvarieties, we will deduce
for instance the following analytically uniform statement, which
corresponds to Theorem~\ref{th-fk} with $g=1$.

\begin{theorem}[Analytically uniform
  KL-stratifications]\label{th-quantitative-bounds}
  Let $n$ and $d$ be positive integers. Let $V$ be a closed subscheme of
  $\Aa^n_{\Zz}$ with $V_{\Cc}$ of dimension~$\leq d$ and given by
  vanishing %% or non-vanishing
  of~$\leq r$ polynomials of degree~$\leq \delta$.  Let
  $f\in\Zz[x_1,\ldots,x_n]$ be given.

  There exist positive integers~$N$ and~$C$, and a stratification
  $\mcX=(X_j)$ with~$X_j$ of relative dimension $\leq n-j$ and
  \[
    \Aa^n_\Zz \supset X_1 \supset \cdots \supset X_n,
  \]
  such that
  %% $X_i$ is a homogeneous subvariety of codimension $\geq i$ and such
  %% that
  \begin{enumth}
  \item for all primes $p$ not dividing $N$ such that the dimension of
    $V$ modulo~$p$ is $\leq d$,
  \item and for all $\uple{h}\in X_i(\Ff_p)\setminus X_{i+1}(\Ff_p)$,
  \end{enumth}
  the bound
  \[
    \Bigl|\sum_{x\in V(k)}\psi(f(x)+h\cdot x)\Bigr|\leq C
    p^{\frac{d+i}{2}}
  \]
  holds.  The integer~$C$ is bounded in terms of $(n,d,r,\delta)$, the
  integer~$N$ is bounded in terms of $(n,d,r,\delta)$ and $\log(N)$ is
  bounded in terms of $(n,d,r,\delta)$ and linearly in terms of the
  height of~$V$, and the height of~$\mcX$ is bounded linearly in terms
  of the height of~$V$.
  % ,N,C)$ satisfies the following bounds:
  % \begin{enumth}
  % \item $C\ll_{n,d,r,\delta}1$;
  % \item $N \ll_{n,d,r\delta} 1$;
  % \item $h(\mcX) \ll_{n,d,r\delta} 1$.
  %   for all $i$.
  % \item the number of irreducible components of $X_i$ is $\ll_{n,D} 1$;
  % \item $h(\mathcal{X})\ll_{n,D} h_c(F)$, i.e. one can write $X_j$ as
  %   the common zero set of polynomials $(G_{j,1},...,G_{j,k})$ in such a
  %   way that $h_c(G_{j,s})\ll_{n,D} h_c(F)$ for every indexing pair
  %   $(j,s)$.
  % \end{enumth}
\end{theorem}

%   Let $d$ and $n$ be positive integers. Let $V$ be a locally closed
%   subscheme of $\Aa^n_{\Zz}$ with $V_{\Cc}$ of dimension~$\leq d$ and
%   given by vanishing or non-vanishing of~$\leq r$ polynomials of
%   degree~$\leq \delta$.  Let $f\in\Zz[x_1,\ldots,x_n]$ be given.
%   There exists a KL-datum~$(\mathcal{X},N,C)$ of $(V,f)$ such that $N$,
%   and $C$ are bounded in terms of $(n,d,r,\delta, \deg(f))$, and the
%   height of~$\mathcal{X}$ is bounded in terms of
%   $(n,d,r,\delta,\deg(f))$ and the height of~$V$.
% \end{theorem}

This statement uses Theorem~\ref{th-general-kl}. One can similarly
exploit Theorem~\ref{th-general-kl-trace}, and this is indeed what is
needed for the work of Bonolis, Pierce and Woo. Since a general
statement would be quite involved, we will only record the particular
special case needed for~\cite{bpw} in Theorem~\ref{th-quantitative-bpw}
below.

\subsection*{Notation.} Given complex-valued functions $f$ and $g$
defined on a set $S$, we write $f \ll g$ if there exists a real number
$C \geq 0$ (called an ``implicit constant'') such that the inequality
$|f(s)|\leq C g(s)$ holds for all $s \in S$. We write $f\asymp g$ if
$f\ll g$ and~$g\ll f$. If $f$ and $g$ are defined on a topological
space~$X$, and $x_0\in X$, then we say that $f\sim g$ as $x\to x_0$ if
$\lim_{x\to x_0} f(x)/g(x)=1$.

For any complex number $z$, we write $e(z)=\exp(2i\pi z)$; the value
$e(a/q)$ is well-defined for $q\geq 1$ and $a\in\Zz/q\Zz$.

By \emph{variety} over a field~$k$, we mean a separated and reduced
scheme of finite type over~$k$, which may be reducible.

\subsection*{Acknowledgments.} L. Pierce and F. Thorne originally
raised the question of uniformity of the constant $C(V,f)$ in
stratification estimates for families of algebraic varieties, and
\'E. Fouvry emphasized the importance of establishing the algebraic
dependency of the stratification.

We thank L. Pierce further for encouragement and discussions concerning
the current text, in particular related to the work~\cite{bpw}. We also
thank W. Sawin for suggesting the iterative strategy used in the proof
of Theorems~\ref{th-quantitative-bounds} and~\ref{th-quantitative-bpw}.

We thank the referees for useful remarks, and we especially thank
F. Gundlach for pointing out a number of mistakes in the original
versions. We thank J. Fresán for help with some technical points of
algebraic geometry.

K.W. is partially supported by NSF GRFP under Grant
No. DGE-2039656. E.K. is partically supported by the SNF project
``Trace functions and arithmetic Fourier transforms'' (SNF grant number
SNF\_219220) and the joint ANR-SNF project ``Equidistribution in Number
Theory'' (SNF grant number 10.003.145 and ANR-24-CE93-0016).

\section{Examples of stratification results and
  applications}\label{sec-motivation}

We present here some applications of stratification results for
exponential sums in general. We begin by recalling how they occur
implicitly in classical results like the Burgess bound, and then
recall some of the results of Fouvry and Fouvry--Katz, before
presenting the recent work of Bonolis, Pierce and Woo which motivated
the search for the uniform statements.

\subsection{Burgess bound}

The celebrated Burgess bound states (in the special case of prime
moduli) that if $\chi$ is a non-trivial multiplicative character
modulo a prime number~$p$, then the bound
\[
  \sum_{1\leq n\leq N}\chi(n)\ll N^{1-1/r}p^{(r+1)/(4r^2)}(\log
  p)^{1/r}
\]
holds for $N\geq 1$ and~$r\geq 1$, where the implied constant is
absolute (see, for
instance~\cite[Th.\,12.6,\,(12.58)]{iwaniec-kowalski}). The key feature
of this bound is that, selecting~$r$ large enough, it provides a
non-trivial estimate provided $N$ is a bit larger than~$p^{1/4}$ in
logarithmic scale.  It is well-known that the proof depends essentially
on the Weil bound for character sums. In fact, it involves the whole
family of character sums of the type
\[
  \sum_{x\in\Ff_p}\chi((x-a_1)\cdots
  (x-a_r))\overline{\chi((x-b_1)\cdots\chi(x-b_r))}
\]
for $(a_1,\ldots,a_r,b_1,\ldots,b_r)\in\Ff_p^{2r}$. These satisfy the
optimal bound
\[
  \Bigl| \sum_{x\in\Ff_p}\chi((x-a_1)\cdots
  (x-a_r))\overline{\chi((x-b_1)\cdots\chi(x-b_r))}\Bigr|\leq
  (2r-1)p^{1/2}
\]
if at least one of the $a_i$ or $b_i$ occurs with multiplicity one among
all of them.  The final step of the proof of the Burgess bound then
depends on the simple stratification of the parameter space of all
$(a_i)$ and~$(b_i)$ according to whether this condition is satisfied or
not -- when it isn't the trivial bound (of size~$p$) is used for the
exponential sum.

Although this example is not of the exact same flavor as the results of
Fouvry--Katz--Laumon (and the proof is quite different), the general
philosophy remains the same. These ideas surrounding stratifications of
character sums have been recently studied by Xu \cite{xu}, with an
application to Burgess bounds for homogeneous forms by Pierce and Xu
\cite{pierce-xu}, and ongoing work in progress of Fouvry, Kowalski,
Michel and Sawin~\cite{fkms} for general bilinear forms with trace
functions.

% \subsection{Multiplicative stratification}

% TODO

\subsection{Class number properties}

One of the very first applications of the stratification theory of Katz
and Laumon was given by Fouvry in~\cite{fouvry2}, and concerns
divisibility properties of class numbers of quadratic fields -- one of
the most fascinating topics in all of number theory.  Fouvry's result
(see~\cite[Théorème]{fouvry2}) was improved significantly by Fouvry and
Katz~\cite[Cor.\,1.3]{fouvry-katz}, whose version we quote:

\begin{theorem}[Fouvry--Katz]
  There exist  $c_0>0$ and $x_0\geq 0$ such that for all $x\geq x_0$,
  the inequality
  \[
    |\{p\leq x\,\mid\, p\equiv 1\bmod{4}, p+4\text{ is squarefree and }
    3\nmid h(p+4)\}|\geq c_0\frac{x}{\log x}
  \]
  holds, where $h(d)$ denotes the class number of the quadratic field
  $\Qq(\sqrt{d})$.
\end{theorem}

We note, however, that the proof of this theorem also depends on a
refinement of Theorem~\ref{th-fk}, namely~\cite[Th.\,1.2]{fouvry-katz},
involving the non-vanishing of the so-called ``$A$-number'' for a
specific scheme. We will not discuss any issue related to this invariant
in this paper, but it would be interesting to understand if the
corresponding results can be extended in families.

\subsection{Equidistribution}

One of the consequences of Theorem~\ref{th-fk} proved by Fouvry and
Katz is a result of equidistribution of fractional parts of values of
general polynomials modulo primes. Precisely, let
\[
  P_1,\ldots, P_r\in \Zz[x_1,\ldots ,x_n]
\]
be polynomials whose linear span does not contain a non-zero
polynomial of degree $\leq 1$. In~\cite[Corollary\,1.4]{fouvry-katz},
Fouvry and Katz prove that for any function
$\phi\colon \Rr^+\rightarrow \Rr$ such that
$\phi(x)\rightarrow \infty$, the set
\[
  \left\{p^{-1}\left(P_1(x),\hdots, P_r(x)\right)\,\mid\,
    x=(x_i)\in\Zz^n \text{ with } 0\leq x_1,\hdots,x_n\leq p^{1/2}
    \log(p) \phi(p) \right\}
\]
becomes equidistributed modulo $1$ as $p\to\infty$.

Theorem \ref{th-quantitative-bounds} allows us to quantify the rate of
equidistribution with explicit dependence on the polynomials
$(P_i)$. Define first
\[
  w(x)=\sqrt{x}(\log x)\phi(x).
\]

Moreover, for $\alpha=(\alpha_i)$ and $\beta=(\beta_i)$ in $[0,1]^r$,
with $\alpha_i\leq \beta_i$ for all~$i$, define the discrepancy by
\begin{multline*}
  D(\alpha, \beta;p) = \Bigl|\ \Bigl|\{x=(x_i)\in \Zz^n\,\mid\, 0\leq
  x_i\leq w(p), \alpha_i\leq \Bigl\{\frac{P_i(x)}{p}\Bigr\}\leq \beta_i,
  \text{ for } 1\leq i\leq r\Bigr\}\Bigr| -
  \\
  w(p)^n\prod_{i=1}^r(\beta_i-\alpha_i) \Bigr|,
\end{multline*}
where $\{t\}$ denotes the fractional part of a real number~$t$.  A
multidimensional version of the Erd\H{o}s-Tur\'an inequality (see,
e.g.,~\cite[Lemma\,2]{gsz}) implies that for any $K\geq 1$, and for the
nontrivial additive character $\psi(n) = e(n/p)$ of $\Ff_p$, we have
\begin{equation}\label{eq-et}
  D(\alpha,\beta;p)
  \ll \frac{w(p)^n}{K} +
  \sum_{\substack{0<A_i\leq K\\ i=1,\hdots,r}}
  \prod_{i=1}^r\frac{1}{\max(k_i,1)} \Bigl|\sum_{|x|_{\infty}\leq
    w(p)}\psi(A_1P_1(x)
  + \hdots + A_rP_r(x))\Bigr|,
\end{equation}
where the implied constant depends only on~$r$.

For $(A_1,\hdots,A_r)\in\Zz^r$, we define
\begin{equation*}
  S(A_1,\hdots,A_r) =  \sum_{|x|_{\infty}\leq
    w(p)}\psi(A_1P_1(x)
  + \hdots + A_rP_r(x)).
\end{equation*}    
% S(A_1,\hdots,A_r) := \sum_{|\mathbf{x}|\leq
% w(p)}\psi(A_1P_1(\mathbf{x}) + \hdots + A_rP_r(\mathbf{x})).
% \end{equation*}

Following the proof of \cite[Corollary\,1.4]{fouvry-katz}, using the
completion method, we determine that for any KL-stratification
$\mathscr{X} = (X_j)$ associated to the family of exponential sums
\[
  \sum_{x\in\Ff_p^n}\psi(A_1P_1(x) + \hdots + A_rP_r(x)+h\cdot x),
\]
the following bound holds:
\begin{equation*}
  S(A_1,\hdots,A_r) \ll
  \frac{1}{p^n} \sum_{j=1}^n p^{\frac{n}{2}+\frac{j-1}{2}}
  \sum_{\mathbf{h}\in X_{j-1}(\Ff_p)} \min(w(p), \|h_i/p\|^{-1}).
\end{equation*}

By Theorem \ref{th-quantitative-bounds}, we know that one may assume
that $\log^+ h(\mathscr{X}) \ll_{n,r,D} h_c((P_1,\hdots,P_r)),$ where
$D$ is the total degree of $P_1\hdots P_r.$ Thus, by \cite[Lemma
9.5]{fouvry-katz}, we deduce that there exists an integer
$C_{n,r,D}\geq 1$ such that
\begin{equation*}
  S(A_1,\hdots, A_r) \ll C_{n,r,D}^{h_c((P_1,\hdots, P_r))} p^{1/2}
  w(p)^{n-1}\log(p).
\end{equation*}

Using this bound in the inequality~(\ref{eq-et}), we obtain
\begin{equation*}
  D(\alpha,\beta;p) \ll
  \frac{w(p)^n}{K} +
  \log(K)^r \cdot C_{n,r,D}^{h_c((P_1,\hdots,P_r))} \cdot
  p^{1/2} w(p)^{n-1}\log(p).
\end{equation*}

Thus we have the following quantitative version of the equidistribution
result of Fouvry and Katz:

\begin{corollary}
  Let $n$ and $r$ be positive integers and let $P_1$, \ldots,
  $P_r \in \Zz[x_1,...,x_n]$ be polynomials whose linear span contains
  no non-zero polynomial of degree $\leq 1$.  Let $\phi\colon\Rr\to \Rr$
  be a function such that $\phi(x) \rightarrow\infty$ as
  $x\rightarrow\infty$.  Then defining
  \[
    w(x) = \sqrt{x} \log(x) \phi(x)
  \]
  we have the bound
  \begin{multline*}
    \Bigl|\Bigl\{x\in \Zz^n\,\mid\, 0\leq x_i\leq w(p), \alpha_i\leq
    \Bigl\{\frac{P_i(x}{p}\Bigr\}\leq \beta_i, \text{ for } 1\leq i\leq
    r\Bigr\}\Bigr| \\ = w(p)^n\prod_{i=1}^r(\beta_i-\alpha_i) +
    o_{n,r,D}(C_{n,r,D}^{h_c((P_1,\hdots, P_r))}\cdot w(p)^n),
  \end{multline*}
  for $\alpha$ and $\beta\in [0,1]^r$, where $C_{n,r,D}>0$ is a
  positive constant dependening only on $n$, $r$ and the total
  degree~$D$ of the product of the polynomials $P_i$.
\end{corollary}

% \Comm{At the moment, $h_c(P)$ denotes the logarithmic height of the
%   coefficients. It might be more clear for this application to have also
%   a notation for $\max_i |a_i|$ (then the final dependence is polynomial
%   on $\max_i |a_i|$).}

\subsection{Integral points on thin sets}\label{ssec-thin}

The question of quantifying how a KL-stratification~$\mathcal{X}$
depends on the underlying subscheme~$V$, function $f$, and trace
function $t(x)$ arises in work of Bonolis, Pierce, and Woo~\cite{bpw}
concerning the polynomial sieve and bounds on counting integral points
on type II thin sets, namely subsets $M \subset \mathbf{P}^{n-1}(\Qq)$
contained in the image $\pi(Z(\Qq))$ for an irreducible projective
algebraic variety $Z$ defined over~$\Qq$ of dimension $n-1$ and a
generically surjective morphism $\pi\colon Z \to \mathbf{P}^{n-1}$ of
degree $D \geq 2$.
% (One motivation to show a thin set contains few points comes from
% specialization of Galois groups: as Serre observed, for $\uple{x}$
% outside of a thin set, $F(Y,\uple{x})$ has the same Galois group as
% $F(Y,\uple{X})$ \cite[p. 122]{Ser97}.)
This can also be generalized to the affine case, and Serre defines in
\cite[p.\,122]{Ser97} an affine thin set of type II in
$\mathbf{A}^n(\Qq)$ to be a subset of the form
\[
  M = \{\uple{x} \in \Qq^n\,\mid\, \exists y \in \Qq, F(y,\uple{x}) =0,
  \text{$y$ not a pole of any coefficient of $F$}\},
\]
for some absolutely irreducible polynomial
$F \in \Qq(x_1,\ldots,x_n)[y]$ of degree~$D\geq 2$ (in the variable
$y$).
% If $F$ satisfies appropriate homogeneity conditions, a thin set of
% type II in $\mathbf{P}^{n-1}(\Qq)$ can be described in an analogous
% way.

Serre conjectured in \cite[p.\,178]{Ser97} that if
$M\subset \mathbf{P}^{n-1}(\Qq)$ is a projective thin set of type II,
then the estimate 
\begin{equation*}
  |\{\uple{x}\in M\,\mid\,  H(\uple{x})\leq B\}| \ll_M B^{n-1}\log(B)^c,
\end{equation*}
holds for $B\geq 2$, where the height function is defined by
$H(\uple{x})=\max_{1 \leq i \leq n}|x_i|$ for
$\uple{x} = [x_1: \hdots : x_n] \in \mathbf{P}^{n-1}(\Qq)$.

The question then arises if an analogous bound would also hold for
affine thin sets of type II. This is known to be false in general, and
indeed counterexamples are easy to describe. For instance, if we define
$F(y,\uple{x})=y^2 - (x_1+x_2)$, it is elementary that
\[
  |\{\uple{x}\in \Zz^2\mid H(\uple{x})\leq B, \exists y \in \Zz,
  F(y,\uple{x})=0\}|\asymp B^{3/2}
\]
for $B\geq 1$.

Nevertheless, for \emph{suitable} polynomials $F\in \Zz[y,x_1,...,x_n]$
with $\deg_y F\geq 2$, it is expected that there should exist some
constant $c=c(F)$ such that
\[
  |\{\uple{x}\in \Zz^n\,\mid\, H(\uple{x})\leq B,\exists y\in \Zz,
  F(y,\uple{x}) = 0\}| \ll_{F} B^{n-1}( \log B)^c.
\]
%%\Comm{KW: added height back into the count}

In the paper~\cite{bpw} of Bonolis, Pierce and Woo, a certain set of
polynomials that are ``genuinely polynomials in $n+1$ variables'' are
considered and one of the aims is to track explicitly the dependence on
$F$ in the implied constant for the bound. The problem is studied using
the polynomial sieve; after an application of the sieve, sums of the
following form arise:
\begin{equation}\label{eq-bpw-sums}
  \sum_{\substack{p\neq q \leq P \\ p,q \textrm{ prime}}}
  \sum_{\substack{\uple{u}\in \Zz^n\\ |\uple{u}|\leq P^2/B}}
  \left|S_F(\overline{q}\uple{u},p)\right|
  \left|S_F(\overline{p}\uple{u},q)\right|,
\end{equation}
for some parameter~$P\geq 1$, where the exponential sum
$S(\uple{h},p)$ is defined as
\begin{equation}\label{eq-sfp}
  S_F(\uple{h},p) = \sum_{\uple{x}\in \Ff_p^n} \psi(\uple{h}\cdot
  \uple{x}) \cdot r_F(\uple{x}), \hspace{1cm} r_F(\uple{x}) =
  \sum_{\substack{y\in \Ff_p\\ F(y,\uple{x})=0}}1,
\end{equation}
and~$\bar{p}$ (resp.~$\bar{q}$) is the inverse of~$p$ modulo
$q$ (resp. of~$q$ modulo~$p$).

A key point is that $r_F$ is a trace function. By the extension of
Theorem~\ref{th-fk} to trace functions
(namely~\cite[Th.\,3.1]{fouvry-katz}), there exists a KL-datum
$(\mathcal{X},N,C)$ for the sums $S_F(\uple{h},p)$, and this precise
classification of the size of $|S(\uple{h},p)|$ allows
for~(\ref{eq-bpw-sums}) to be rewritten as
\[
  C^2\sum_{\substack{p\neq q \leq P \\ p,q \textrm{ prime}}}
  \sum_{j=1}^n \sum_{k=1}^n
  p^{(n+j-1)/2}q^{(n+k-1)/2}\sum_{\substack{\uple{u}\in \Zz^n\\
      |\uple{u}|\leq P^2/B\\ \overline{q}\uple{ u}\in X_j(\Ff_p)\\
      \overline{p}\uple{u}\in X_k(\Ff_q)}}1.
\]

Thus, the KL-stratification turns the question about exponential sums
into a problem of counting points on the strata.  To estimate
\textit{explicitly the dependence on $F$} for the internal count of
integral points whose reductions lie in $X_j(\Ff_p)$ and $X_k(\Ff_q)$,
the methods of~\cite{bpw} requires understanding how certain properties
of these strata, such as the degree and the ``height'' $h(\mathcal{X})$,
depend on the original trace function $r_F$; hence, the estimates of
Theorem~\ref{th-quantitative-bounds} become necessary.  Beyond the above
application, the authors expect that the quantitative version of
Theorem~\ref{th-fk} given in results like
Theorem~\ref{th-quantitative-bounds} will give uniform versions of other
applications of the general theory of Katz--Laumon stratifications to
analytic number theory.

We conclude by recording the statement which will be directly cited
in~\cite{bpw}.

\begin{theorem}\label{th-quantitative-bpw}
  Let $n$ and $D$ be positive integers. Let $F\in \Zz[y,x_1,,...,x_n]$
  be a polynomial of total degree $\leq D$ which is monic
  in~$y$.
  % \footnote{EK: this was missing in the previous version, but if I
  % understand correctly, this assumption is not a problem in your
  % application.}
  There exist positive integers~$N$ and~$C$ and a
  stratification $\mcX=(X_j)$ with
  \[
    \Aa^n_\Zz \supset X_1 \supset \cdots \supset X_n,
  \]
  such that $X_i$ is a homogeneous subscheme of codimension $\geq i$ and
  such that for all primes $p$ not dividing $N$ and for all
  $\uple{h}\in X_i(\Ff_p)\setminus X_{i+1}(\Ff_p)$, the bound
  \[
    |S_{F}(\uple{h},p)|\leq C p^{\frac{n+i}{2}}
  \]
  holds, and moreover, the data $(\mcX,N,C)$ satisfies the following
  bounds:
  \begin{enumth}
  \item $C$ is bounded in terms of $n$ and~$D$ only,
  \item $N$ is bonded in terms of~$n$ and~$D$ and $\log(N)$ is bounded
    linearly in terms of the height of~$F$,
    %% \ll_{n,D}1$ and $N \ll_{n,D} 1$,
  \item the degree of each~$X_i$ is bounded in terms of $n$
    and~$D$ only,
  \item the number of irreducible components of $X_i$ is $\ll_{n,D} 1$,
  \item we have $h(\mathcal{X})\ll_{n,D} h_c(F)$.
  \end{enumth}

  In particular, one can write $X_j$ as the common zero set of
  polynomials $(G_{j,1},...,G_{j,k})$ in such a way that
  $h_c(G_{j,s})\ll_{n,D} h_c(F)$ for every indexing pair $(j,s)$.
\end{theorem}

Since the proof requires the same ideas as that of
Theorem~\ref{th-quantitative-bounds}, we will give it in
Section~\ref{sec-quantitative-bounds}.

% \Comm{TODO: fix corollary to have $N(\mathscr{X}) \ll_{n,D} 1$ to match Theorem \ref{th-quantitative-bounds} when that finalizes.} 

\section{Examples of stratifications for explicit families}\label{sec-examples}

We present here a few examples in which suitable stratifications can
be explicitly computed.

\subsection{Linear spaces}

As an example of Theorem \ref{th-fk} and the indexing of the strata
$(X_j)$, let us consider the example when $f=0$, $g=1$, and $V$ is a
linear subspace of $\Aa^n_\Zz$ of relative dimension $d=n-2$. Let
$V^\perp$ denote the ``orthogonal'' space
\[
  V^{\perp}=\{h\in \Aa^n\,\mid\, h_1x_1+\cdots +h_nx_n=0\text{ for all }
  x\in V\},
\]
which is a linear subspace of relative dimension~$2$. We have
\begin{equation}\label{eq: planar example}
  \sum_{x\in V(\Ff_p)} \psi(h_1x_1+\cdots +h_nx_n) = \begin{cases}
    p^{n-2}, & h\in V^\perp(\Ff_p) \\ 
    0, & \text{otherwise}.
  \end{cases} 
\end{equation}

We can easily  interpret the above in terms of a KL-stratification,
taking 
$$
X_1 = X_2 = \cdots = X_{n-2} = V^\perp, \text{ and }X_{n-1} = X_n =
\emptyset,
$$
and $N=C=1$. Then Theorem \ref{th-fk} is verified via the computation
(\ref{eq: planar example}), observing that for $j \in \{n-1,n\}$, the
estimate (\ref{eq-sum}) is the trivial bound.

\subsection{Quadratic forms in $n$ variables}
\label{quadratic}

Let now $F=x_{1}^{2}+\cdots+x_{n}^2$ be a diagonal quadratic form in
$n$ variables. Assume that~$p$ is odd. Taking $\psi(x)=e(x/p)$, and
defining
\[
  T(F,\uple{v};p)=
  \sum_{\substack{\uple{x}\in\Ff_{p}^{n}\\F(\bfx)=0}}\psi(\uple{v}\cdot\bfx)
\]
we obtain
\begin{align*}
  T(F,\uple{v};p) &=\frac{1}{p}\sum_{a\in\Ff_{p}^{\times}}\sum_{\bfx\in
    \Ff_{p}^{n}}\psi(aF(\bfx)+\uple{v}\cdot\bfx)+p^{n-1}
  \delta_{\uple{v}=\boldsymbol{0}}\\
  &=\frac{1}{p}\sum_{a\in\Ff_{p}^{\times}}\prod_{i}
  \sum_{x}\psi(ax^{2}+v_{i}x)+p^{n-1}\delta_{\uple{v}=\boldsymbol{0}}
  \\&=\frac{\tau^{n}}{p}\sum_{a\in\Ff_{p}^{\times}}\chi_2(a)^{n
  }\psi(-F(\uple{v})/(4a))+p^{n-1}\delta_{\uple{v}=\boldsymbol{0}},
\end{align*}
where $\chi_2$ is the Legendre symbol modulo $p$ and $\tau$ is the Gauss
sum associated to $\chi_2$. From this computation, it follows that a
Katz--Laumon stratification in the case of diagonal quadratic forms
depends on the number of variables:
\begin{itemize}
\item[$i)$] when $n$ is odd, we can take
  \begin{itemize}
  \item[$\bullet$] $X_1 = X_2 = \cdots = X_{n-1} = \{\boldsymbol{0}\}$,
    since for all $\uple{v}\neq \boldsymbol{0}$, we have
    $T(F,\uple{v};p) = O(p^{\frac{n-1}{2}})$ and
    $T(F,\boldsymbol{0};p) = p^{n-1}$.
  \end{itemize}
\item[$ii)$] when $n$ is even, we can take
  \begin{itemize}
  % \item[$\bullet$] $X_0 = \Aa^n$, and additionally for
  %   $\uple{v}\in (X_0\setminus X_1)(\Ff_p)$, we get
  %   $T(F,\uple{v};p) = O(p^{\frac{n-2}{2}}).$
  \item[$\bullet$]
    $X_{1}=\{\uple{v}\in\Ff_{p}^{n}\,\mid\, \uple{v}\not=0\text{ and }
    F(\uple{v})=0\}$.
  \item[$\bullet$] $X_2=\cdots = X_{n-1}=\{\boldsymbol{0}\}$.
  \end{itemize}
  Indeed, for $\uple{v}\notin X_1$, so that $\uple{v}\not=0$ and
  $F(\uple{v})\not=0$, we have $T(F,\uple{v};p) = O(p^{\frac{n-2}{2}})$,
  and for $\uple{v}\in X_2\setminus X_1$, so that $F(\uple{v})=0$ but
  $\uple{v}\not=0$, we get $T(F,\uple{v};p)=O(p^{n/2})$. Finally, for
  $\uple{v}=0$, we have
  \[
    T(F,0;p)=p^{n-1}+(1-1/p)\tau^n.
  \]
\end{itemize}

Observe also that since the stratification remains unchanged when
passing to finite extensions of $\Ff_{p}$, one deduces that the
stratification for any full rank quadratic form $F$ in $n$ variables
satisfies the same properties.

\subsection{The case of smooth forms}

Let now $F\in\Zz[T_1,\ldots,T_n]$ be an homogeneous polynomial in $n$
variables. Take again $\psi(x)=e(x/p)$. For every
$\uple{v}\in\Ff_{p}^{n}$, let
\[
  T(F,\uple{v};p)=\sum_{\substack{\bfx\in\Ff_{p}^{n}\\F(\bfx)=0}}
  \psi(\bfx\cdot\uple{v}).
\]

Note that for every $a\in\Ff_{p}^{\times}$, we have
\[
  T(F,a\uple{v};p)=\sum_{\substack{\bfx\in\Ff_{p}^{n}\\F(\bfx)=0}}
  \psi(\bfx\cdot(a\uple{v}))
  =\sum_{\substack{\bfy\in\Ff_{p}^{n}\\F(\bfy)/a^n=0}}
  \psi(\bfy\cdot\uple{v})=T(F,\uple{v};p),
\]
and hence
%% \Ques{Should the second sum be over $a\in \Ff_p$ instead of
%% $a\in \Ff_p^\times$?}
\begin{align*}
  (p-1)T(F,\uple{v};p)
  &=\sum_{a\in\Ff_{p}^{\times}}T(F,a\uple{v};p)\\
  &=\sum_{a\in\Ff_{p}}\sum_{\substack{\bfx\in\Ff_{p}^{n}\\F(\bfx)=0}}
  \psi(a(\bfx\cdot\uple{v}))-
  |\{\bfx\in\Ff_{p}^{n}\,\mid\,F(\bfx)=0\}|\\
  &=
    p|\{\bfx\in\Ff_{p}^{n}\,\mid\, F(\bfx)=\bfx\cdot\uple{v}=0\}|-
    |\{\bfx\in\Ff_{p}^{n}\,\mid\,F(\bfx)=0\}|,
\end{align*}
which implies 
\[
  T(F,\uple{v};p)= \frac{p}{p-1}|\{\bfx\in\Ff_{p}^{n}\,\mid\,
  F(\bfx)=\bfx\cdot\uple{v}=0\}|
  -\frac{1}{p-1}|\{\bfx\in\Ff_{p}^{n}\,\mid\,F(\bfx)=0\}|.
\]

If we furthermore assume that the associated projective hypersurface
$V(F)\subset\mathbf{P}^{n-1}$ is non-singular (in particular,~$F$ is
irreducible), then we know that
\[
  \frac{1}{p-1}|\{\bfx\in\Ff_{p}^{n}\,\mid\, F(\bfx)=0\}|
  = p^{n-2}+O(p^{\frac{n-2}{2}})
\]
(see, e.g.,~\cite[App.,\,Th.\,1]{hooley}), while for
$\uple{v}\neq \mathbf{0}$, we have
\[
  \frac{1}{p-1}|\{\bfx\in\Ff_{p}^{n}\,\mid\,
  F(\bfx)=\uple{v}\cdot\bfx=0\}|= p^{n-3}+E(F,\uple{v};p),
\]
where 
\[
  E(F,\uple{v};p)=
  \begin{cases}
    O(p^{\frac{n-3}{2}})
    &\text{if $V(F)\cap V(\uple{X}\cdot\uple{v})$ is smooth,}\\
    O(p^{\frac{n-2}{2}})
    &\text{if $V(F)\cap V(\uple{X}\cdot\uple{v})$
      is singular.}
  \end{cases}
\]
(loc. cit.)

%%\Comm{Typo: changed $p$'s in exponent to $n$'s.}
On the other hand, $V(F)\cap V(\uple{X}\cdot\uple{v})$ is singular if
and only if $\uple{v}$ is a point on the dual variety
$V(F)^{\ast}$. Moreover, when $F$ defines a smooth projective variety,
the dual variety $V(F)^{\ast}$ is known to be an irreducible
hypersurface. Thus we conclude that, assuming that $V(F)$ is smooth, we
obtain a Katz--Laumon stratification with
\begin{itemize}
\item[$\bullet$] $X_{1}=V(F)^{\ast}$,
\item[$\bullet$] $X_2 = \cdots = X_{n-1}=\{\boldsymbol{0}\}$.
\end{itemize}

Indeed, the above shows that $T(F,\uple{v};p)=O(p^{\frac{n-1}{2}})$ for
$\uple{v}\notin V(F)^{\ast}$, and $T(F,\uple{v};p)=O(p^{\frac{n}{2}})$
holds when $\uple{v}\in X_{1}\setminus X_{n-1}$, whereas
$T(F,\boldsymbol{0};p)=p^{n-1}+O(p^{n/2})$.

\begin{remark}
  This computation is compatible with the case of quadratic forms, since
  $V(F)^{\ast}=V(F)$ if $F$ is a quadratic form.
\end{remark}

% \Ques{Is there a way to see through this analysis the difference between the $n$ odd and $n$ even case for quadratic forms?}\New{Mmmmm, good question. I am not sure: it is not clear how to conclude something about the weight of the exponential sum when $n$ is even.}

\subsection{A case with a deep stratification}

Let $n\geq 1$ be an integer, and let $F_{0}$, \ldots,
$F_{n-1}\in\Zz[x,y,z,w]$ be non-degenerate quadratic forms in $4$
variables. Let $V$ be the subscheme of $\mathbf{P}^{4n-1}_{\Zz}$ defined
by the equations
\[
  V=\{[x_{0}:\cdots :x_{4n-1}]\,\mid\,
  F_{j}(x_{4j},x_{4j+1},x_{4j+2},x_{4j+3})=0\text{ for } 0\leq j\leq
  n-1\}.
\]

Still with~$\psi(x)=e(x/p)$ for $x\in\Ff_p$, consider
\[
  T(V,\uple{v};p)=\sum_{\substack{\bfx\in
      V(\Ff_{p})}}\psi(\bfx\cdot\uple{v}).
\]

For every $k\in\{1,...,n\}$ we denote
$S_{k}=\{I\subset\{1,...,n\}\,\mid\, |I|=k\}$. Then we define
%% $X_0=V$ and
\[
  X_{k}=\bigcup_{I\in S_{k}}\{[x_{0}:\cdots :x_{4n-1}]\,\mid\,
  F_{j}(x_{4j},x_{4j+1},x_{4j+2},x_{4j+3})=0\text{ for every }j\in I\}
\]
for $1\leq k\leq 3n-2$ and $X_{3n-1}=\{\boldsymbol{0}\}$. From the
computations made in Section $\ref{quadratic}$, we deduce that
generically $T(V,\uple{v};p)=O(p^{(3n-1)/2})$ and that
\[
  T(V,\uple{v};p)=
  \begin{cases}
    O(p^{\frac{3n-1+k}{2}}) &
    \text{if $k\leq n-1$ and $\uple{v}\in X_{k}\setminus X_{k+1}$},\\
    O(p^{2n-1/2}) &
    \text{if $\uple{v}\in X_{n}\setminus \{\boldsymbol{0}\}$},\\
    p^{3n-1} &\text{if $\uple{v}=\boldsymbol{0}$}.
\end{cases}
\]

\subsection{Stratification in families}
  
We are going to use the stratification for diagonal quadratic forms to
illustrate Theorem~\ref{th-algebraic-dependency}. Let $n\geq 2$ be an
odd integer (the case $n$ even is similar), and consider the
``universal'' diagonal quadratic form
\[
F(\uple{A},\uple{T})=\sum_{i=1}^{n}A_{i}T_{i}^{2}\in
\Zz[A_1,\ldots,A_n,T_1,\ldots, T_n].
\]

We view this in the setting of Theorem~\ref{th-algebraic-dependency}
with $r=n$, and
\[
  W=\{(\uple{a},\uple{x})\in \Aa^r\times\Aa^r\,\mid\,
  F(\uple{a},\uple{x})=0\}
\]
with morphism $\Delta(\uple{a},\uple{x})=\uple{a}$. Note that, in this
case, all varietes $V_{\uple{a}}$ are hypersurfaces with relative
dimension $n-1$, except when $\uple{a}=0$, and that $V_{\uple{a}}$
modulo~$p$ is a hypersurface unless $p$ divides all $a_i$.
In what follows, we fix a prime~$p$ and we denote by
$\delta(\uple{a},\bfx)$ the characteristic function of the solution set
of the quadratic form, i.e., we put
\[
  \delta (\uple{a},\uple{x})=\begin{cases}
    1,&\text{if $F(\uple{a},\uple{x})=0$,}\\
    0,&\text{otherwise.}
\end{cases} 
\]
for $\uple{a}$ and $\uple{x}$ in~$\Ff_p^n$. We still denote
$\psi(x)=e(x/p)$, and define the function
\[
  \varphi(\uple{a},\bfb,\uple{x})=\delta(\uple{a},\uple{x})
  \psi(-\uple{a}\cdot\bfb)
\]
on~$\Ff_p^{3n}$.  Further, for $\uple{d}\in\Ff_p^n$, we denote by
$F_{\uple{d}}$ the specialized quadratic form, so that
\[
  F_{\uple{d}}(\uple{x})=F(\uple{d},\bfx)=\sum_{i=1}^{n}d_{i}x_{i}^{2}.
\]

We compute the discrete Fourier transform of~$\varphi$. It is given by
the formula
\begin{align*}
  \widehat{\varphi}(\uple{c},\uple{d},\uple{v})
  &=\sum_{(\uple{a},\bfb,\uple{x})\in\Ff_p^{3n}}\psi(\uple{a}\cdot\uple{c}+
  \bfb\cdot\uple{d}+\uple{v}\cdot\uple{x})\varphi(\uple{a},\bfb,\uple{x})
  \\
  &=\sum_{(\uple{a},\bfb,\uple{x})\in\Ff_p^{3n}}\psi
  (\uple{a}\cdot\uple{c}+\bfb\cdot\uple{d}+
  \uple{v}\cdot\uple{x}-\uple{a}\cdot\bfb)\delta(\uple{a},\uple{x})
  \\
  &=\sum_{\substack{(\uple{a},\uple{x})\in\Ff_p^{2n}\\F(\uple{a},\uple{x})=0}}\psi
  (\uple{a}\cdot\uple{c}+\uple{v}\cdot\uple{x})
  \sum_{\bfb\in\Ff_p^n}\psi(\bfb\cdot(\uple{d}-\uple{a}))\\
  &=p^{n}\psi(\uple{d}\cdot\uple{c})
  \sum_{\substack{\uple{x}\in\Ff_p^{n}\\F(\uple{d},\uple{x})=0}}\psi
  (\uple{v}\cdot\uple{x})=p^{n}
  \psi(\uple{d}\cdot\uple{c})T(F_{\uple{d}},\uple{v};p)
\end{align*}
for $(\uple{c},\uple{d},\uple{v})\in\Ff_p^{3n}$.

Since the equality
\[
  |\widehat{\varphi} (\uple{c},\uple{d},\uple{v})|=
  |\widehat{\varphi}(0,\uple{d},\uple{v})|=
  p^{n}|T(F_{\uple{d}},\uple{v};p)|,
\]
holds for all $(\uple{c},\uple{d},\uple{v})$, we can use the
stratification computed in Section~\ref{quadratic} to define a
stratification $\mathscr{X}$
%% \Typo{
\begin{equation}
  \mathbf{A}^{3n}_{\Zz}\supset X_{1}\supset\cdots\supset X_{3n-1},
\label{eq : stratall}
\end{equation}
% \Ques{DB: the stratification is actually stopping at
% $X_{2n-1-3/2(n)}$ or something similar... but maybe it is better to
% leave it as it is?}\Comm{KW: I think it makes sense to leave it
% going to $X_{3n-1}$}
by defining
%%\Typo{
\[
  X_{j}=\{(\uple{c},\uple{d},\uple{v})\in \Aa^{3n}\,\mid\, \uple{v}\in
  X_{j,\uple{d}}\},
\]
%}
where $X_{j,\uple{d}}$ is the $j$-th stratum of the stratification
associated to the exponential sums $T(F_{\uple{d}},\uple{v};p)$ computed
in Section~\ref{quadratic}.  On the other hand, we can use the
computation in Section~\ref{quadratic} to describe $X_{j,\uple{d}}$,
which leads to the following descriptions:
\begin{itemize}
\item[$i)$] if $j$ is even, then
  \[
    X_{j}=\bigcup_{\substack{I\subset\{1,...,n\}\\
        |I|=n-j-1}}\{(\uple{d},\uple{v})\,\mid\, d_{i}=0\text{ if
    }i\not\in I, F(\uple{d},\uple{v})=0\}
    \cup\bigcup_{\substack{J\subset\{1,...,n\}\\
        |J|=n-j}}\{(\uple{d},\uple{v})\,\mid\, d_{i}=0\text{ if
    }i\not\in J\},
  \]
  or in other words, $X_{j}$ is the set of solutions of the equations
  % \Comm{I think we have switched variables here from $\uple{d}$ to
  % $\uple{A}$ -- I think it makes more sense to keep it as $\uple{d}$
  % for now and specialize to the coefficients $\uple{A}$ later?}
  \[
    \bigcup_{\substack{I\subset\{1,...,n\}\\ |I|=n-j-1}}\{ A_{i}=0\text{
      if }i\not\in I,
    F(\uple{A},\uple{v})=0\}\bigcup_{\substack{J\subset\{1,...,n\}\\
        |J|=n-j}}\{(\uple{d},\uple{v}): A_{i}=0\text{ if }i\not\in J\}.
  \]
\item[$ii)$] if $j$ is odd, then
  \[
    X_{j}=\bigcup_{\substack{I\subset\{1,...,n\}\\
        |I|=n-j-2}}\{(\uple{d},\uple{v})\,\mid\, d_{i}=0\text{ if
    }i\not\in I, F(\uple{d},\uple{v})=0\},
  \]
  i.e. $X_{j}$ is defined by the equations
  \[
    \bigcup_{\substack{I\subset\{1,...,n\}\\ |I|=n-j-2}}\{ A_{i}=0\text{
      if }i\not\in I, F(\uple{A},\uple{v})=0\}.
  \]
\end{itemize}

Notice that by specializing the variables $\uple{A}=\uple{d}$ in the
definition of the $X_{j}$, one obtains the defining equation for the
strata $X_{j,\uple{d}}$ associated to $T(F_{\uple{d}},\uple{v};p)$. In
other words, we can obtain a stratification $\mathcal{X}_{\uple{d}}$ for
$T(F_{\uple{d}},\uple{v};p)$ by specializing the stratification
$\mathcal{X}$ for $\widehat{\varphi}$.  But whereas we constructed in
this example the stratification $\mathcal{X}$ for $\widehat{\varphi}$ by
putting together the stratification~$\mathcal{X}_{\uple{d}}$ for
each~$\uple{d}$, the idea behind Theorem~\ref{th-algebraic-dependency}
is to follow the opposite strategy: one first constructs a
stratification $\mathcal{X}$ for $\widehat{\varphi}$, following the work
of Katz-Laumon and Fouvry-Katz and then one shows that, for a dense open
subset $U\subset\mathbf{A}^{n}$, the stratification
$\mathcal{X}_{\uple{d}}$ can be obtained by ``specializing'' each stratum
of $\mathcal{X}$ by $\uple{A}=\uple{d}$.

\section{Algebraic uniformity: general strategy}\label{sec-algebraic}

This section is devoted to the description of the general strategy
behind the various uniformity theorems. We note that the use of such
tools as the derived category of constructible $\ell$-adic sheaves,
and its formalism, are genuinely unavoidable here, due to the
essential need to work with higher-dimensional trace functions. We
refer analytic number theorists who are yet unfamiliar with this
theory to the accessible intuitive discussion by Forey, Fresán and
Kowalski found in the Appendix, which is reproduced (with the kind
permission of the authors) from~\cite[App.\,E]{ffk} with minor
changes.

% proof of
% Theorem~\ref{th-algebraic-dependency}.
%%\begin{remark}
We also note that, as is often the case in applications of $\ell$-adic
machinery to exponential sums, there is an interplay between concrete
computations with trace functions and exponential sums, and their
interpretations in terms of sheaf theory. We will usually present the
argument using the former, before establishing the algebraic
incarnation.
%%\end{remark}

The basic strategy for the stratification theorems is encapsulated in
the following ``local'' lemma, which builds on the basic properties of
trace functions. We recall that the objects involved are presented
intuitively in the Appendix.
% which is implicit in the original work of Fouvry and
% Katz~\cite[p.\,124,\,125]{fouvry-katz} (and certainly well-known in
% the algebraic-geometry community).

% \Comm{$\mathcal{H}$ currently denotes both the stratification of $X$ and the cohomology sheafs.}
% \Ques{At the moment, the appendix does not give a definition for
%   $D^b_c(X,\bQl)$. Should we add one to the appendix or one here?}

\begin{lemma}\label{lm-fk}
  Let~$X/k$ be an algebraic variety over a finite field~$k$. Let~$L$ be
  an object of $D^b_c(X,\bQl)$ for some prime $\ell$ invertible
  in~$k$. Let $\mcX=(X_i)$ be a finite stratification of~$X$ with
  smooth, equidimensional strata, such that all the cohomology
  sheaves~$\mcH^j(L)$ of~$L$ are lisse on each~$X_i$. Suppose further
  that $L$ is semiperverse and mixed of weights $\leq w$, for some
  integer~$w$.

  Let 
  \[
    C=\sum_{j}\rank (\mcH^j(L)|X_i).
  \]
  
  For every $i$ and every $x\in X_i(k)$, we have
  \[
    |t_L(x)|\leq C|k|^{(w-\dim(X_i))/2},
  \]
  and in particular, if~$Y_j$ denotes the closure of the union of the
  strata with dimension~$\leq \dim(X)-j$, then we have
  \[
    |t_L(x)|\leq C|k|^{(w-\dim(X)+j-1)/2},
  \]
  for any~$x\in (X\setminus Y_j)(k)$.
\end{lemma}

\begin{proof}
  We have, more or less by definition, the formula
  \[
    t_L(x)=\sum_{j}(-1)^k\Tr(\frob_x\mid \mcH^j(L)_x).
  \]

  Since~$L$ is mixed of weights~$\leq w$, each $\mcH^j(L)$ is mixed of
  weights~$\leq j+w$ (again by definition of weights in this context).

  Consider a stratum~$X_i$ and an integer~$j$. The key remark is that
  if the restriction of~$\mcH^j(L)$ to $X_i$ is non-zero, then because
  this cohomology sheaf is lisse on~$X_i$, its support contains a
  connected component of~$X_i$, and hence has dimension
  $\geq \dim(X_i)$. On the other hand, by definition of a semiperverse
  object, the support of~$\mcH^j(L)$ has dimension $\leq -j$. What
  this means is that if $x\in X_i(k)$, then we have $\mcH^j(L)_x=0$
  unless $j\leq -\dim(X_i)$. Therefore, we get the bound
  \[
    |t_L(x)|=\Bigl|\sum_{j\leq -\dim(X_i)}(-1)^k\Tr(\frob_x\mid
    \mcH^j(L)_x)\Bigr| \leq C|k|^{(w-\dim(X_i))/2},
  \]
  as claimed.

  For the last statement, if~$x\in (X\setminus Y_j)(k)$, then
  picking~$i$ so that $x\in X_i(k)$, we must
  have~$\dim(X_i)\geq \dim(X)-j+1$ by definition of~$Y_j$, so
  \[
    |t_L(x)|\leq C|k|^{(w-\dim(X)+j-1)/2}.
  \]
\end{proof}

\begin{remark}
  We can check this for consistency: in a ``generic'' stratum with
  $\dim(X_i)=\dim(X)$, we obtain a bound of size
  $|k|^{(w-\dim(X))/2}$. (Equivalently, note that $Y_1$ has
  dimension~$\leq \dim(X)-1$, hence is a proper closed subvariety
  of~$X$, and thus for $x$ in the dense open set $X\setminus Y_1$, we
  recover the bound of size $|k|^{(w-\dim(X))/2}$ again.)

  If $w=0$, then this becomes $|k|^{-\dim(X)/2}$, and this corresponds
  to ``square-root cancellation'' when the trace function of~$L$ is
  associated to an exponential sum over~$X$, taking into account the
  normalization involved in the definition of weights in this context
  (see Remark~\ref{rm-normalize}, (1)).
\end{remark}

It is a basic fact about objects of~$D^b_c(X,\bQl)$ that, over a fixed
finite field, one can \emph{always} find a stratification with the
properties in this lemma. However, in applications, one starts with
``integral'' objects (over~$\Zz$), and a crucial remark going back to
Fouvry~\cite[p.\,85]{fouvry} is that it is essential that this
stratification be also ``defined over~$\Zz$'', i.e., that the strata
$H_i$ for the situation modulo a prime~$p$ be defined by reduction
modulo~$p$ of fixed varieties defined by the vanishing of integral
polynomials.
% In other words, the basic
% applications of stratification (as in~\cite{fouvry-katz}) are already
% examples of a type of algebraic uniformity, and this type of
% uniformity goes back to the fundamental work of Katz and
% Laumon~\cite{katz-laumon}.  \Comm{I think this last sentence is
%   slightly confusing because it refers to a different type of
%   algebraic uniformity than the algebraic uniformity in Theorem
%   \ref{th-algebraic-dependency}. Here is a proposed new sentence:
In other words, the original applications of stratification (as
in~\cite{fouvry-katz} and going back to the fundamental work of Katz and
Laumon ~\cite{katz-laumon}) already exhibit a type of \emph{algebraic
  uniformity} as the prime~$p$ varies; this is a different type of
algebraic uniformity than the versions of
Theorems~\ref{th-algebraic-dependency} and~Theorem \ref{th-general-kl},
where the underlying variety varies.

The strategy we will use to prove our main results is the following,
which elaborates on the argument of Katz--Laumon and Fouvry--Katz, and
provides an approach to uniform versions of the lemma above.

\begin{enumerate}
\item[] \textbf{Step~1} (\emph{Representation}) We represent the
  family of exponential sums of interest as the trace function of some
  object~$R$ on the parameter space.
\item[] \textbf{Step~2} (\emph{Stratification}) We construct a finite
  stratification of the parameter space with smooth equidimensional
  strata such that the cohomology sheaves of~$R$ are lisse on each
  stratum; this stratification is constructed ``over~$\Zz$'', so that it
  restricts to suitable stratifications modulo all (but finitely many)
  primes.
\item[] \textbf{Step~3} (\emph{Semiperversity}) We show that~$R$ is fiberwise
  semiperverse and mixed of some bounded weights.
\item[] \textbf{Step~4} (\emph{Betti}) Finally, we find an estimate
  for the constant~$C$ that is uniform with respect to the desired
  parameters.
\end{enumerate}

In general, the Representation Step is a matter of applying the
formalism of étale cohomology and the Grothendieck--Lefschetz trace
formula. The most general and precise version of the Stratification Step
is found in~\cite[Th.\,2.1]{fouvry-katz}; it relies on the previous
foundational work of Katz and Laumon~\cite{katz-laumon}. On the other
hand, in the cases under consideration at least, the Semiperversity Step
is based on the properties of Deligne's geometric Fourier transform over
finite fields (this step concerns the situation over finite fields, in
contrast with the Stratification Step, which is algebraic geometry
over~$\Zz$; the interplay of these two features is quite
interesting). Finally, the Betti Step follows in fact automatically from
the Stratification Step, if the version of Fouvry--Katz is used, but it
would usually also be possible to derive it (sometimes in slightly less
precise form) from Sawin's \textit{Quantitative Sheaf
  Theory}~\cite{qst}.\footnote{\ Note that this tool was not available
  at the time of~\cite{fouvry-katz}.}

In the case of Theorem~\ref{th-algebraic-dependency}, the principles are
the same, and in particular we perform a Fourier transform not only
over~$h$ but also over the parameter variable~$a\in\Aa^r$. This however
only gives a stratification~$(Y_j)$ of the whole space~$W$, and not of
the single fibers~$V_a$. Defining $X_{j,a}$ to be the fiber over~$a$
gives what we want, but only in general for~$a$ outside a codimension
one subscheme, as in the statement of the theorem.

There is an additional difficulty in Theorem~\ref{th-general-kl}, in
applying the Semiperversity Step of this strategy: because the parameter
space in that case is not in general an affine space, we cannot perform
a Fourier transform over~$a$ to imitate the previous case. We will work
around this with an additional trick involving an ``extra'' parameter
variable~$b$, allowed to range over all of~$\Aa^r$ and thus suitable for
Fourier analysis.
% is to first construct the stratification by applying one of the basic
% results of Fouvry--Katz, namely~\cite[Th.\,2.1]{fouvry-katz}, to the
% auxiliary variety $W\times \Aa_{\Zz}^k$. This leads to a first
% description of the relevant exponential sums. Once the stratification
% has been obtained, we produce an alternative interpretation of the
% sums, using the Fourier transform. We then deduce the desired
% cancellation properties using the properties of the algebraic Fourier
% transform, and especially it's compatibility with (semi)perverse
% sheaves.
\par
We will give the detailed proofs of our results in the next
sections. % We begin with Theorem~\ref{th-algebraic-dependency},
% although it is a special case of Theorem~\ref{th-general-kl}, because
% the setting is a bit easier.

\section{Algebraically uniform stratification, I}

In this section, we start by proving the simpler
Theorem~\ref{th-algebraic-dependency}, which avoids some of the
complications of the most general statement. However, since it doesn't
really involve any more work, we proceed in the setting of
Theorem~\ref{th-general-kl-trace}, i.e., with a general adapted trace
function instead of simply a function of the form $\chi(g(x))$. In other
words, we prove Theorem~\ref{th-general-kl-trace} with parameter space
$\mathbf{M}=\Aa^r_{\Zz}$.

\subsection{Data}\label{ssec-data}

%% Proof of Theorem~\ref{th-algebraic-dependency}}
% \Comm{The proof is currently written for general trace functions
% $t_L(x)$ whereas Theorem \ref{th-algebraic-dependency} is stated for
% specifically for $\chi(g(x))$. Potentially discuss at the next
% meeting when is the right time to introduce more general exponential
% sums (with $t_L(x)$) in intro (or later) and define
% KL-stratifications more generally.}

We thus assume we are given the data $(W,\Delta)$, as well as a
stratification $\mcW=(W_i)$ of~$W$.

We then define or assume given the following \emph{global} data and
notation:
%% We apply~\cite[Th.\,2.1]{fouvry-katz} with the following data:
\begin{itemize}
\item $T=\Aa^n_{\Zz}\times\Aa^r_{\Zz}$, with coordinates
  $(h,a)=(h_1,\ldots,h_n,a)$, where $a\in \Aa^r$,
\item $X=\Aa^n_{\Zz}\times W$, with coordinates
  $(h,x)=(h_1,\ldots, h_n,x_1,\ldots,x_n)$, and with ``input''
  stratification $\mathcal{X}$ defined by $(\Aa^n_{\Zz}\times W_i)_i$,
\item $\pi\,:\, X\lra T$ is the map given by $\pi(h,x)=(h,\Delta(x))$,
\item $F\,:\, X\lra \Aa^1_{\Zz}$ is the function given by
  \[
    F(h,x)=f(x)+x\cdot h=f(x)+\sum_{i=1}^n x_ih_i
  \]
  viewed as a $T$-morphism $X\lra \Aa^1_T$,
\item $\ell$ is some fixed prime number,
\item $K$ is an object of the derived category $D^b_c(W[1/\ell],\bQl)$,
  adapted to the stratification~$\mcW$ %% $\{W\}$ of $W$.
  of~$W$.
  % , and moreover fiberwise semiperverse and fiberwise mixed of
  % weights $\leq 0$.
\end{itemize}

We also fix an isomorphism between~$\bQl$ and~$\Cc$, and below we
identify these two fields, so that in particular exponential sums (which
are elements of~$\Cc$) can be identified with trace functions of
$\ell$-adic objects (which are in~$\bQl$).

In addition, we will consider the following \emph{local} data:
\begin{itemize}
\item a finite field $k$ of characteristic not dividing $\ell$, with an
  algebraic closure~$\bar{k}$, such that $W\otimes \Fp$ has dimension
  $d$ and moreover $K\otimes \Fp$ is semiperverse of weights $\leq 0$,
\item a non-trivial additive character $\psi\colon k\to\bQl^{\times}$,
\item a direct factor $L$ of $K\otimes k$ (which is then also
  semiperverse of weights~$\leq 0$).
\end{itemize}

We will then study the exponential sums
\[
  \sum_{x\in V_a(k)}\psi(f(x)+h\cdot x)t_L(x) 
\]
where $V_a=\Delta^{-1}(a)$, viewed as functions of
$(h,a)\in k^n\times W(k)$.

\begin{remark}\label{rm-chig}
  Readers who are mostly interested in the case of
  Theorem~\ref{th-algebraic-dependency} may assume that~$K$ and~$L$ are
  chosen so that $t_L(x)=\chi(g(x))$ for a given invertible function~$g$
  and some non-trivial multiplicative character~$\chi$, which is adapted
  to the stratification~$\mcW=\{W\}$. More precisely, to satisfy the
  last conditions on~$K$ (the semiperversity and the fact that the
  weights are $\leq 0$), we should take
  \[
    t_L(x)=(-1)^{d}|k|^{-d/2}\chi(g(x))
  \]
  corresponding to the object $\mcL_{\chi(g)}[d](d/2)$, where the shift
  by $d$ ensures the semiperversity condition (using the assumption that
  $W\otimes\Fp$ has dimension~$d$) and the twist by $d/2$ is then used
  to have weights $\leq 0$.
  % Indeed, as explained in~\cite[proof of
  % Cor.\,3.2]{fouvry-katz}, this object over~$k$ is a direct summand of
  % $g^*(x\mapsto x^M)_*\bQl[d](d/2)$, where~$M$ is the order of the
  % character~$\chi$, and this object is semiperverse of weights~$\leq 0$.
\end{remark}

\subsection{Representation}\label{sec-r}

We assume given the global data above, and consider the exponential
sums for a given choice of local data $(k,\psi,L)$.

We claim that for $(h,a)\in T(k)$, the sum
% Let $k$ be a finite field of characteristic not dividing $\ell$ and let
% $\psi\colon k\to\bQl$ be a non-trivial additive character of $k$. For a
% direct factor $L$ of $K\otimes k$ and $t=(h,a)\in T(k)$, the
% corresponding exponential sum is
\[
  \sum_{x\in V_a(k)}\psi(f(x)+h\cdot x)t_L(x) = \sum_{\substack{x\in
      W(k)\\ \Delta(x)=a}}\psi(f(x)+h\cdot x)t_L(x)
\]
is the value at $(h,a)$ of the trace function of the object
\begin{equation}\label{eq-r}
  R\pi_{k,!}(p_{2,k}^*L\otimes \sheaf{L}_{\psi(F)})
\end{equation}
where $p_2\,:\, X\lra W$ is the projection. This is an object of the
category $D^b_c(T_k,\bQl)$, which we denote\footnote{\ There should be
  no confusion with the $R$ notation for derived functors.} by~$R$
(here and below, in notation such as $\pi_k$, $p_{2,k}$ or~$T_k$, we
use the subscript $k$ to indicate the base change to~$k$).
\par
This is a straightforward check: by the formalism of trace functions
and of $\ell$-adic cohomology (especially the meaning of $R\pi_!$ as
``summing over points of the fiber''), the trace function~$t_R$ of~$R$
satisfies
\begin{align*}
  t_R(h,a) &=\sum_{(h',x)\in \pi^{-1}(h,a)(k)}\psi(F(x,h))
             t_L(x)\\
           &=\sum_{\substack{x\in
             W(k)\\\Delta(x)=a}}\psi(f(x)+h\cdot x) t_L(x)
  =\sum_{x\in V_a(k)}\psi(f(x)+h\cdot x)t_L(x)
\end{align*}
for all $(h,a)\in T(k)$.
% which is the basic sum of interest for $V_a$, provided~$K$ and~$L$ are
% chosen so that
% \[
%   t_L(x)=\chi(g(x)).
% \]

\subsection{Stratification}\label{ssec-stratif}

We apply the stratification theorem of Fouvry and
Katz~\cite[Th.\,2.1]{fouvry-katz} to the global data above (i.e., that
of Section~\ref{ssec-data}).  The output is a triple
$(N,C,\mathcal{Z})$ where:
\begin{enumerate}
\item $N\geq 1$ is an integer,
\item $C\geq 0$ is a real number,
\item $\mathcal{Z}=(Z_i)_{i\in I}$ is a finite stratification of
  $T=(\Aa^n\times\Aa^r)_{\Zz[1/N]}$.
\end{enumerate}

All of these depend (only) on $(W,\Delta,\mcW,f,\ell)$, and satisfy
various properties.  The most important ones are:
\begin{enumerate}
\item each stratum~$Z_i$ is smooth and surjective over~$\Zz[1/N]$;
\item the geometric fibers of each stratum of $Z_i$ over $\Zz[1/N]$ are
  equidimensional (see~\cite[Th.\,2.1,\,1)]{fouvry-katz});
\item for any finite field $k$ of characteristic not dividing
  $\ell N$, and any non-trivial additive character~$\psi$ as in the
  Representation Step, the restriction to $Z_i$ (over~$k$) of the
  cohomology sheaves of the corresponding object~$R$ in~(\ref{eq-r})
  are \emph{adapted} to~$\mcZ$, i.e., their restriction to
  each~$(Z_i)_k$ are lisse (see~\cite[Th.\,2.1,\,2)]{fouvry-katz}).
\end{enumerate}

In particular, we see that this stratification fits into the pattern
of Lemma~\ref{lm-fk}, when we restrict to finite fields of
characteristic not dividing~$\ell N$.

\begin{remark}
  (1) We note that although $C$ is mentioned in the statement
  of~\cite[Th.\,2.1]{fouvry-katz}, it does not appear in the published
  statements of the properties that $(N,C,\mathcal{Z})$ are stated to
  satisfy. This is a typographical mistake, and the right-hand side of
  the main inequality in property (2) of loc. cit. should be
  $C\sup_{x\in X_t}\|L\|(x)$ instead of $\sup_{x\in X_t}\|L\|(x)$.

  (2) Note that the stratification is a \emph{global} object, but the
  objects~$R$ considered in the first step are \emph{local},
  constructed over a given finite field~$k$, because additive
  characters (and their associated Artin--Schreier sheaves) only make
  sense as algebraic objects over a single finite field (and its
  extensions).
\end{remark}
\par
% Consider an object $K$ of $D^b_c(W[1/M\ell],\bQl)$ for some $M\geq 1$
% and some prime $\ell$, adapted to the stratification $\{W\}$ of $W$.

\subsection{Semiperversity and weights}\label{ssec-perv-weights}

As in the general strategy, we claim that for suitable finite field~$k$
and any non-trivial additive character~$\psi$, the object~$R[n]$ is
semiperverse. Although the category of semiperverse sheaves enjoys many
good stability and formal properties, this property is not directly
obvious from the definition of~$R$ (due to the application of $R\pi_!$,
which does not in general preserve this property). To prove the claim,
we apply Fourier analysis. The crucial facts are: (1) that the Fourier
transform of a trace function is again a trace function (due to
Deligne): (2) that the Fourier transform of an object~$M$ is
semiperverse if and only if the object~$M$ is semiperverse, if the
Fourier transform is properly normalized.\footnote{\ Analytically, this
  normalization is similar to the unitary normalization of the discrete
  Fourier transform.}

% The assumptions on~$k$, as in the definition of a general KL-datum
% (Definition~\ref{def-kl-datum}) is that $

We will thus compute the Fourier transform. Although one needs to do
this \emph{algebraically} (at the level of sheaves or objects of the
derived category), it is intuitively much clearer to start by doing the
computation with trace functions, and then to upgrade it to a
sheaf-theoretic statement.

Let~$(\eta,\alpha) \in T(k)$ (viewed as the ``dual'' space for the
Fourier transform). Then we compute the value at~$(\eta,\alpha)$ of
the Fourier transform of~$t_R$, namely
\begin{align}
  \widehat{t}_R(\eta,\alpha)
  &=
    \sum_{(h,a)\in T(k)}t_R(h,a)\psi(-a\cdot \alpha-h\cdot \eta)
    \notag
  \\
  &=
    \sum_{(h,a)\in T(k)}\psi(-a\cdot \alpha-h\cdot \eta)
    \sum_{\substack{x\in W(k)\\\Delta(x)=a}} \psi(f(x)+h\cdot x)t_L(x)
  \label{eq-ft3}
  \\
  &= \sum_{x\in W(k)}t_L(x) \psi(f(x)) \psi(-\Delta(x)\cdot
    \alpha)\sum_{h\in k^n}\psi(h\cdot (x-\eta))\label{eq-ft4}
  \\
  &= |k|^nt_L(\eta)\psi(f(\eta)) \psi(-\Delta(\eta)\cdot \alpha)\label{eq-ft5}
\end{align}
by, in turn, the definition, exchanging the order of summation and
finally orthogonality of characters of~$k^n$.
% , this vanishes unless $\eta\in W(k)$, and
% in that case this becomes
% \[
%   |k|^nt_L(\eta)\psi(f(\eta)) \psi(-\Delta(\eta)\cdot \alpha).
% \]

Up to the normalization (which is important), this is the trace
function of the object
\[
  q_1^*(i_{W,!}(L)\otimes \sheaf{L}_{\psi(f)})\otimes
  \sheaf{L}_{\psi(G)},
\]
of $D^b_c(T_k,\bQl)$, where
\begin{enumerate}
\item $q_1\,:\, T_k=\Aa^{n}_k\times\Aa^r_k\lra \Aa^n_k$ is the first
  projection $(\eta,\alpha)\mapsto \eta$;
\item the map $i_W\colon W_k\to \Aa^n_k$ is the immersion of~$W$
  in~$\Aa^n_k$;
\item the function $G$ is defined by
  $G(\eta,\alpha)=-\Delta(\eta)\cdot \alpha$.
\end{enumerate}

The precise sheaf-theoretic version of this computation is the
following lemma, which we will prove in Section~\ref{sec-iso}.

\begin{lemma}\label{lm-sheaf}
  With notation as above, the algebraic Fourier transform~$\ft_{\psi}R$
  of~$R$ is isomorphic to
  \[
    (q_1^*(i_{W,!}(L)\otimes \sheaf{L}_{\psi(f)})\otimes
    \sheaf{L}_{\psi(G)})[r-n]
  \]
  in $D^b_c(T_{\bar{k}},\bQl)$.
\end{lemma}

Since semiperversity is a geometric property, it can be tested after
base change to~$\bar{k}$.  The object
$i_{W,!}(L)\otimes \sheaf{L}_{\psi(f)}$ is semiperverse (because~$L$ is
semiperverse and a closed immersion is $t$-exact, see
e.g.~\cite[Cor.\,4.1.3]{BBD-pervers}) and $\sheaf{L}_{\psi(f)}$ is
lisse, therefore $q_1^*(i_{W,!}(L)\otimes \sheaf{L}_{\psi(f)})[r]$ is
semiperverse since the morphism $q_1$ is smooth of relative
dimension~$r$ (the functor $q_1^*$ is $t$-exact, see,
e.g.,~\cite[p.\,108]{BBD-pervers}). Thus, it follows from the lemma that
$\ft_{\psi}R[n]$ is semiperverse, hence~$R[n]$ is semiperverse, as we
claimed.

We then need to control the weights of~$R[n]$. Since~$L$ is mixed of
weights~$\leq 0$ and~$\mcL_{\psi(F)}$ is pure of weight~$0$, the tensor
product $p_{2,k}^*L\otimes \mcL_{\psi(F)}$ is mixed of weights~$\leq 0$,
and Deligne's most general version of the Riemann
Hypothesis~\cite{deligne} implies that
$R=R\pi_{k,!}(p_{2,k}^*L\otimes\mcL_{\psi(F)})$ is also mixed of
weights~$\leq 0$. Consequencely, $R[n]$ is mixed of weights~$\leq n$.

\subsection{Conclusion}

We have now obtained the expected properties for the object~$R$, or its
twist~$R[n]$ of weights~$\leq n$. Applying Lemma~\ref{lm-fk}, this has
immediate consequences for the stratification of the exponential sums
$t_R(h,a)$ in terms of the stratification $\mathcal{Z}$ of~$T$ obtained
in the Stratification step: if $Z_i$ is a stratum of~$\mathcal{Z}$ with
geometric fibers of dimension~$\eta_i$, then we have
\begin{equation}\label{eq-final-bound}
  \Bigl| \sum_{x\in V_a(k)}t_L(x)\psi(f(x)+h\cdot x)\Bigr|=
  |t_{R[n]}(h,a)|\leq 
  C|k|^{(n-\eta_i)/2}
\end{equation}
for $(h,a)\in Z_i(k)$, where the constant~$C$ only depends on the
complexity of~$R$, and hence on the complexity of~$L$ and the degree
of~$f$, according to Quantitative Sheaf Theory\cite[\S\,6]{qst}.

\begin{remark}
  We check this for consistency in the case of sums with
  multiplicative characters. According to Remark~\ref{rm-chig}, this
  means that
  \[
    t_L(x)=(-1)^{\dim_{\Zz}(W)}|k|^{-\dim_{\Zz}(W)/2}\chi(g(x)),
  \]
  and hence we obtain
  \[
    \sum_{x\in V_a(k)}\chi(g(x))\psi(f(x)+h\cdot
    x)=O(|k|^{(n+\dim_{\Zz}(W)-\eta_i)/2}).
  \]
  
  Since~$\dim_{\Zz}(W)=d+r$, in a generic stratum with fibers of
  dimension $\eta_i=r+n$, this gives
  \[
    \sum_{x\in V_a(k)}\chi(g(x))\psi(f(x)+h\cdot
    x)=O(|k|^{(n+(d+r)-(r+n))/2})=O(|k|^{d/2}),
  \]
  indicating the expected square-root cancellation whenever $V_{a,k}$
  has dimension~$d$, which is a generic condition.
\end{remark}

Although we have obtained a good stratification result, we are not quite
done in the proof of Theorem~\ref{th-algebraic-dependency}. The issue is
that we want to obtain stratifications for the individual
varieties~$V_a$.

As in~Lemma~\ref{lm-fk} and \cite[p.\,126]{fouvry-katz}, we wish to
define the subvarieties $(X_{j,a})$ in the stratification as the closure
of the fiber $Z_{i,a}=\{h\,\mid\, (h,a)\in Z_i\}$ of those strata which
have relative dimension $\leq r+n-j$ over~$\Zz[1/N]$. However, this will
only give the correct estimates when those fibers themselves have
relative dimension $\leq n-j$, which may not be the case, depending on
the value of~$a$.
% To finish, we must handle the possibility that, for some strat $H_i$
% and some $a\in\Zz$, the fiber $H_{i,a}=\{h\,\mid\, (h,a)\in H_i\}$
% could be of relative dimension $<\eta_i-r$, implying that $X_j(a,f)$
% is not of dimension $\leq n-j$.
However, this property holds generically, and this provides the
statement we want. We now explain this.

We first construct the data $((Y_j),N,C,A,\varphi)$ described in
Theorem~\ref{th-general-kl-trace}. The integer~$N$ and the real
number~$C$ have already been described.

Next, recall that $\eta_i$ denotes the common dimension of all
geometric fibers of $Z_i/T[1/N]$.  We denote
\[
  Y_j=\overline{\bigcup_{\eta_i\leq n+r-j}\overline{Z}_i}
\]
as in~\cite[p.\,126]{fouvry-katz}, the schematic closure\footnote{\
  See, e.g.,~\cite[\S\,10.8]{g-w} for a definition of the schematic
  closure.}  in~$T$ of the Zariski closure in $T[1/N]$ of the union of
the strata with $\eta_i\leq n+r-j$.
% We claim that the triple
% \[
%   ((Y_j),N,C)
% \]
% has the properties claimed in Theorem~\ref{th-general-kl-trace}
% (recalling that we are proving this result in the case
% $\mathbf{M}=\Aa^r_{\Zz}$).
By construction, $Y_j$ has relative dimension
$\leq n+r-j$, and
\[
  \Aa^r_{\Zz}\times\Aa^{n}_{\Zz}\supset Y_1\supset \cdots \supset
  Y_{n+r}.
\]
\par
Next, for each $i$, we consider the projection
\[
  \pi_{i}\,:\, Z_i\lra \Aa^r_{\Zz[1/N]}
\]
on the coordinate $a$, so that $Z_{i,a}=\pi_i^{-1}(a)$. Let
$J\subset I$ be the subset of those $j\in I$ where $\pi_j$ is
\emph{not} dominant, i.e., such that the image of $\pi_j$ is not
Zariski-dense in $\Aa^r_{\Zz[1/N]}$. For $j\in J$, the Zariski-closure
of $\pi_j(Z_j)$, denoted as $A_j\subset \Aa^r_{\Zz[1/N]}$, has
relative dimension $<r$.
% of the image of $\pi_j$ is such that
% $A_{j,\Cc}(\Cc)$ is a finite set of values. 
We denote by $A_0$ the union of the Zariski-closures
(in~$\Aa^r_{\Zz[1/N]}$) of $\pi_j(Z_j)$ for~$j\in J$. This is a
(possibly reducible) subscheme of relative dimension~$<r$, and it
depends only on $(W,\Delta,f)$.
\par
For $i\notin J$, the morphism $\pi_i$ is dominant by definition.  By a
standard result of algebraic geometry (see,
e.g.,~\cite[Cor.\,10.85]{g-w}), there exists an open dense
subset~$U_i$ of~$\Aa^r_{\Zz[1/N]}$ such that $\pi_i$ is flat
over~$U_i$, and we may assume that~$U_i$ is a standard open subset
defined by inverting a non-zero
polynomial~$\varphi_i\in \Zz[1/N,X_1,\ldots,X_r]$, i.e.
\[
  U_i=\Spec(\Zz[1/N,X_1,\ldots, X_r,1/\varphi_i]).
\]

Multiplying $\varphi_i$ by a suitable non-zero integer, we may assume
that in fact~$\varphi_i\in \Zz[X_1,\ldots,X_r]$.  We then define
\[
  \varphi=\prod_{i\notin J}\varphi_i\in \Zz[X_1,\ldots, X_r].
\]

By properties of flat morphisms (see
e.g.,~\cite[Cor.\,14.116\,(1)]{g-w}), we have
% there exists a dense open subset $U_i\subset \Aa^r_{\Qq}$ such that
\[
  \dim(Z_{i,\Qq})=\dim(\pi_i^{-1}(a))+\dim(\Aa^{r}_{\Qq})=
  \dim(\pi_i^{-1}(a))+r
\]
for $a\in U_i$ (note the base change to~$\Qq$). Since the stratum
$Z_i/\Qq$ is equidimensional with geometric fibers of
dimension~$\eta_i$, we have $\dim(Z_{i,\Qq})=\eta_i$, and therefore
% Moreover, the fiber $\pi_1^{-1}(a)$ is the base change to the
% residue field of~$a$ of $Z_{i,a}$, so its dimension is the relative
% dimension of $Z_{i,a}$. Thus we obtain the relative dimension
% formula
\[
  \dim_{\Qq}(Z_{i,a})=\eta_i-r
\]
for $a\in U_i$.
% $\dim_{\Zz} Z_{i,a}=\eta_i-r$ for all $a\in U_i$.  The complement
% $\tilde{A}_i$ of $U_i$ is $\tilde{A}_{i,\Cc}(\Cc)$ is finite.

We denote by $A_1$ the union of $\Aa^r_{\Zz[1/N]}\setminus U_i$ for
$i\notin J$; this is a subscheme of~$\Aa^r_{\Zz[1/N]}$ of
codimension~$\geq 1$. We also denote by $A_2$ the Zariski-closure
in~$\Aa^r_{\Zz[1/N]}$ of the set of $a$ such that $V_a$ has dimension
different from~$d$ over~$\Qq$. This is again a proper subscheme. We
next define $\widetilde{A}=A_0\cup A_1\cup A_2$, a subvariety
of~$\Aa^r_{\Zz[1/N]}$ of relative dimension~$<r$, depending only on
$(W,\Delta,f)$.
% The parameter space $T$ is covered by the strat $H_i$ for $i\notin J$,
% and the ``vertical'' subschemes $\{a\}\times \Aa^n$ for $a\in A$. For
% $a\in\Aa^1-A$, and $i\notin J$, the fiber $H_{i,a}$ is either empty
% (if $a\notin \pi_i(H_i)$) or has dimension $\eta_i-1$ by the above.
\par
Let finally $N_1\geq 1$ be an integer, and let $\mathbf{M}_1$ be a
closed subscheme of~$\mathbf{M}$ of codimension at least~$1$.  Let~$K$
be an object on~$W$ which is fiberwise $V_b$-transverse for all~$b$
outside $\mathbf{M}_1$, and fiberwise semiperverse of weights~$0$ for
primes $p\nmid N_1$. We define~$A=\widetilde{A}\cup \mathbf{M}_1$.
\par
\medskip
\par
\textbf{Claim.} The data
\[
  ((Y_j),N,C,A,\varphi)
\]
satisfies the conditions of Theorem~\ref{th-general-kl-trace}.
\par
\medskip
\par
Let $a\in \Zz^r$ not in~$A(\Qq)$. By construction, the fiber $Z_{i,a}$
is empty for $i\in J$. If~$i\notin J$, on the other hand, then the
polynomials~$\varphi_i$ satisfy $\varphi_i(a)\not=0$,
so~$\varphi(a)\not=0$. If we define~$N_a=\varphi(a)$, then~$a$ defines
(by evaluation) a morphism
\[
  \Spec(\Zz[1/(NN_a)]\to U_i.
\]

The pullback of the flat morphism $\pi_i^{-1}(U_i)\to U_i$ along this
morphism is still flat, and in particular all its fibers have the same
dimension~$\eta_i$ (see, e.g,~\cite[Th,\,14.114]{g-w}, and the
following remarks). But this pullback is the base change
of~$Z_{i,a}\to V_a$ to~$\Zz[1/(NN_a)]$, and hence $Z_{i,a}$ has
relative dimension~$\eta_i-r$ over~$\Zz[1/(NN_a)]$.  It follows that
the subschemes~$Y_{j,a}$ of $Y_j$ are of relative dimension $\leq n-j$
over~$\Zz[1/(NN_a)]$.

For $p\nmid \ell NN_1N_a$ with $\dim(V_{a,\Fp})=d$ and for~$h\in\Ff_p^n$
such that $(a,h)\in \Ff_p^{n+r}-Y_j(\Fp)$, we deduce
from~(\ref{eq-final-bound}) that
\[
  \Bigl|\sum_{x\in V_a(k)}\psi(f(x)+h\cdot x)t_L(x)\Bigr|\leq C|k|^{(j-1)/2}.
\]

Since~$K$ is $V_a$-transverse modulo~$p$ by assumption, this estimate
translates to the desired conclusion that the triple
$((X_{j,a}),NN_a,C)$ is a KL-datum for $(V_a,f|V_a,\mcV_a)$ and the
object $(K_a|V_a)[-\codim(V_a)]$, after replacing $K$ and $L$ by the
necessary shifts.

\subsection{Proof of Lemma~\ref{lm-sheaf}}\label{sec-iso}

This section, which may safely be skipped in a first reading, gives the
proof of the claimed isomorphism of Lemma~\ref{lm-sheaf}.  We observe
that such results are quite standard (see for instance in the
works~\cite{katz-act} or~\cite{katz-mmp} of Katz), and can be
interpreted as the outcome of following line by line the ``classical''
computation (namely, the equalities~(\ref{eq-ft3}),~(\ref{eq-ft4})
and~(\ref{eq-ft5})) using the function-sheaf dictionary.

These computations are in some sense straightforward, except for two
complications:
\begin{itemize}
\item the use of the standard notation from algebraic geometry
  involves a lot of bookkeeping and
  % potentially confusing proliferation of morphisms, which
  often hide the close parallel with the computation of trace
  functions.
\item the necessity to keep careful track of shifts in the various
  steps, since these are (up to sign) invisible in the trace function,
  but carry crucial information, e.g. in terms of determining when an
  object is semiperverse.
\end{itemize}

We will present the proof in a way which, we hope, illustrates how
similar computations can be performed in fair generality.  It would be
useful to have a rigorous and usable formulation of the heuristic
principle used here that any ``standard'' computation of this type has
a sheaf-theoretic version, and we go a little bit in this direction by
explaining, in general, how certain steps are done on the
sheaf-theoretical level.  Moreover, we note that we stated
Lemma~\ref{lm-sheaf} as a geometric isomorphism, i.e.,
in~$D^b_c(T_{\bar{k}},\bQl)$, but it would not be difficult to upgrade
this to an arithmetic isomorphism, with a suitable Tate twist. We omit
this since this is not needed for our purpose.

We will use a short-hand notation to represent the steps of the
computation of the Fourier transform, similar to one used for similar
purposes in~\cite[Proof of Prop.\,9.20]{ffk}:
\begin{enumerate}
% \item Given a $k$-variety~$X$ and an object~$M$ of~$D^b_c(X,\bQl)$, if
%   $x$ is a choice of coordinate on~$X$, we write $M(x)$
\item For any morphism $f\colon X\to Y$ of~$k$-varieties, with
  coordinates denoted $x$ and~$y$ respectively, and object~$M$ of
  $D^b_c(Y,\bQl)$, we write $M(f(x))$ for~$f^*M$; for instance,
  if~$X=Y\times Z$ and $f$ is the projection $(y,z)\mapsto y$, we
  write $M(y)$ for~$f^*M$.
\item For Artin--Schreier sheaves $\mcL_{\psi(f)}$, we write
  $\psi(f)$.
\item We often drop the $\otimes$ sign, representing multiplication.
\item Given a subvariety~$Y$ of a variety~$X$, with immersion
  $j\colon Y\to X$, we write~$\delta_Y$ for the object~$j_!\bQl$
  on~$X$. 
\item We drop the $R$ prefix before derived functors (this is in fact
  a fairly usual convention).
\item For a morphism $f\colon X\to Y$ of~$k$-varieties, with
  coordinates denoted $x$ and~$y$ respectively, and for an object~$N$
  of~$D^b_c(X,\bQl)$, we write $\sigma_{f(x)=y}M(x)$ for $Rf_!M$; for
  a projection $Y\times Z\to Y$, we write $\sigma_{z}M(y,z)$.
\item We write $=$ for the existence of an isomorphism in the derived
  category, over~$\bar{k}$, i.e., a geometric isomorphism.
\end{enumerate}

These conventions simplify the bookkeeping involved in constructing
various objects, and are more closely related to the corresponding
notation for the trace functions.

With these conventions, the definition of~$R$ becomes
\[
  R=R\pi_{k,!}(p_{2,k}^*L\otimes
  \mcL_{\psi(F)})=\sigma_{\Delta(x)=a}L(x)\psi(F(h,x)). 
\]

Moreover, given an object~$M$ of~$D^b_c(T_k,\bQl)$, with coordinates
$(h,a)$ on~$T$, Deligne's algebraic Fourier transform\footnote{\ More
  precisely, this is the inverse of Deligne's transform, which was
  originally defined with $\psi(\eta\cdot h+\alpha\cdot a)$ instead of
  $\psi(-\eta\cdot h-\alpha\cdot a)$.} of~$M$ is the object of the
same category, but with coordinates $(\eta,\alpha)$ on~$T$, given by
\[
  \ft_{\psi}M(\eta,\alpha)=R\pi_{2,!}(\pi_1^*M\otimes
  \sheaf{L}_{\psi(-\eta\cdot h-\alpha\cdot
    a)})[n+r]=\sigma_{h,a}M(h,a)\psi(-\eta\cdot h-\alpha\cdot a)[n+r]
\]
(we refer to Laumon's paper~\cite[\S\,1]{laumon} for the fundamental
properties of the algebraic Fourier transform; see also
Example~\ref{ex-fourier}).

To prove Lemma~\ref{lm-sheaf}, our goal is to compute this object
for~$M=R$. This is done by copying the computation of the trace
function. More precisely, by definition we get
\begin{align*}
  \ft_{\psi}R(\eta,\alpha)&=\sigma_{h,a}R(h,a)\psi(-\eta\cdot
  h-\alpha\cdot a)[n+r]\\
  &= \sigma_{h,a}\sigma_{\Delta(x)=a}L(x)\psi(f(x)+h\cdot x)
  \psi(-\eta\cdot h-\alpha\cdot a)[n+r].
\end{align*}

Then, using the functoriality of the direct image with compact support
(see, e.g.,~\cite[Prop.\,5.5.1,\,(iii)]{lei-fu}), we can exchange the
two sums to get
\[
  \ft_{\psi}R(\eta,\alpha)=\sigma_x \sigma_{h} L(x)\psi(f(x)+h\cdot x)
  \psi(-\eta\cdot h-\alpha\cdot \Delta(x))[n+r].
\]

(To be more precise: we use the fact that the map
\[
  (\eta,\alpha,h,a,x)\mapsto (\eta,\alpha)
\]
from the subvariety of
$\Aa^n\times\Aa^r\times\Aa^n\times \Aa^n\times W$ defined by
$\Delta(x)=a$ to the affine space~$\Aa^n\times\Aa^r$ has two
factorizations:
\begin{gather*}
  (\eta,\alpha,h,a,x)\mapsto (\eta,\alpha,h,a)\mapsto (\eta,\alpha)
  \\
  (\eta,\alpha,h,a,x)\mapsto (\eta,\alpha,a,x)\mapsto (\eta,\alpha).
\end{gather*}

This implies that the compositions of higher direct images along these
two morphisms coincide; the first composition corresponds to the
definition of~$\ft_{\psi}R$, and the second corresponds to the claim.)

The projection formula (see, e.g.,~\cite[Th.\,7.4.7]{lei-fu}) allows
us to ``pull out'' factors that are independent of a summation
variable, i.e.
\[
  \ft_{\psi}R(\eta,\alpha)=\sigma_x L(x)\psi(f(x)-\alpha\cdot
  \Delta(x))\sigma_h\psi(h\cdot (x-\eta)) [n+r].
\]

Now we apply the following key fact, which is the analogue of the
orthogonality of characters:

\begin{lemma}\label{lm-key}
  We have
  \[
    \sigma_h\psi(h\cdot (x-\eta)) =\delta_{x=\eta}[-2n].
  \]
\end{lemma}

\begin{proof}
  The key point here is to have the correct shift. To check it, note
  that by the Künneth formula (see, e.g.,~\cite[Cor.\,7.4.9]{lei-fu}),
  and by an elementary translation which allows us to assume
  that~$\eta=0$, it suffices to consider the case $n=1$ and to prove
  that (in more traditional sheaf notation)
  \begin{equation}\label{eq-ft-delta}
    Rp_!\mcL_{\psi(ab)}=\delta_0[-2]
  \end{equation}
  where $p\colon \Aa^1\times \Aa^1\to \Aa^1$ is the second projection
  and $\delta_0$ coincides with the skyscaper sheaf at~$0\in\Aa^1$.
  This can be checked directly by computing cohomology, but this can
  be remembered and recovered as follows: the left-hand side is, up to
  shift, the Fourier transform of the trivial sheaf on~$\Aa^1$.  We
  know that the result is a shift of $\delta_0$, and to check which is
  the right normalization, we use the fact that the Fourier transform
  of a perverse sheaf is perverse
  (see~\cite[Th.\,1.3.2.3]{laumon}). In this case, this means that
  $\ft_{\psi}\bQl[1]$ is perverse, hence we must have
  \[
    \ft_{\psi}\bQl[1]=(Rp_!\mcL_{\psi(ab)}[1])[1]=\delta_0,
  \]
  which implies~(\ref{eq-ft-delta}).
\end{proof}

Inserting this result in the previous computation, we get
\[
  \ft_{\psi}R(\eta,\alpha)=L(\eta)\psi(f(\eta)) \psi(\alpha\cdot
  \Delta(\eta)) [-n+r].
\]

If we transcribe this back in the classical notation, this concludes the
proof of Lemma~\ref{lm-sheaf}.

\section{Algebraically uniform stratification, II}

We now consider the full statement of Theorem~\ref{th-general-kl-trace},
and we begin the proof in the same manner as before. The global data
is now
\begin{itemize}
\item $T=\Aa^n_{\Zz}\times \mathbf{M}$, with coordinates
  $(h,a)=(h_1,\ldots,h_n,a)$, where
  $a\in \mathbf{M}\subset \Aa^r_{\Zz}$,
\item $X=\Aa^n_{\Zz}\times W$, %%$X=W\times \Aa^n_{\Zz}$,
  with coordinates
  $(h,x)=(h_1,\ldots, h_n,x_1,\ldots,x_n)$;
  % \footnote{Adapt this  everywhere.}
  for the input stratification $\mathcal{X}$ of~$X$, we take
  $(\Aa^n_{\Zz}\times W_i)_i$,
\item $\pi\,:\, X\lra T$ is given by $\pi(h,x)=(h,\Delta(x))$,
\item the function $F\,:\, X\lra \Aa^1$ given by
  \[
    F(h,x)=f(x)+x\cdot h=f(x)+\sum_{i=1}^n x_ih_i
  \]
  viewed as a $T$-morphism $X\lra \Aa^1_T$,
\item $\ell$ is some fixed prime number,
\item $K$ is an object of the derived category $D^b_c(W[1/\ell],\bQl)$,
  adapted to the stratification $\mcW$ of~$W$.
\end{itemize}

Moreover, we may assume that~$\mathbf{M}$ is an integral
scheme.\footnote{\ This is useful to apply later the theorem of
  generic flatness as in~\cite[Cor.\,10.85]{g-w}, where this is an
  assumption for the target scheme.} Indeed, recall that a scheme is
integral if and only if it is reduced and irreducible (see,
e.g.,~\cite[Def.\,3.26]{g-w}), and we can argue for each irreducible
component of~$\mathbf{M}$ separately (recall that we assume
that~$\mathbf{M}$ is reduced).

We apply~\cite[Th.\,2.1]{fouvry-katz} and obtain data $(N,C,\mcM)$
where~$\mcM=(M_i)$ is a finite stratification
of~$(\mathbf{M}\times\Aa^n)_{\Zz[1/N]}$, satisfying properties
analogue to those in Section~\ref{ssec-stratif}.

Exactly as in~\ref{sec-r}, given the local data $(k,\psi,L)$ (which is
the same as previously), we then construct the object
\[
  R\pi_{k,!}(p_{1,k}^*L\otimes\mcL_{\psi(F)})
\]
on $T_k$, which has trace function the desired exponential sums, and we
denote it again by~$R$.

In order to complete the strategy of Section~\ref{sec-algebraic}, we
need to prove as in Section~\ref{ssec-perv-weights} that the object
$R[n]$ is semiperverse on~$T_k$.  As we already hinted, there is a new
difficulty here: we cannot perform a Fourier transform on the parameter
space, which is not in general an affine space, to check this in the
same manner as before. However, once this is proved, the final parts of
the argument reproduce exactly those given in the case
$\mathbf{M}=\Aa^r_{\Zz}$ in the previous section.

To prove the semiperversity, we introduce additional auxiliary variables
ranging over all of~$\Aa^r_k$ to perform Fourier analysis.  Precisely,
we define the object
\[
  R'=\mcL_{\psi(a\cdot b)}\otimes p_{12}^*R
\]
on~$\Aa^n\times\mathbf{M}\times\Aa^r$ (over~$k$) with coordinates
$(h,a,b)$, and with projection $p_{12}\colon (h,a,b)\mapsto (h,a)$.
The trace function of~$R'$ is
\[
  (h,a,b)\mapsto \psi(a\cdot b)t_R(h,a),
\]
(so that $|t_{R'}(h,a,b)|=|t_R(h,a)|$ for all $(h,a,b)$, which explains
the fact that bounds for exponential sums based on~$R$ or~$R'$ are
equivalent).  Moreover, we have isomorphisms
\[
  \mcH^i(R')\simeq \mcL_{\psi(a\cdot b)}\otimes\mcH^i(p_{12}^*R)
\]
for all~$i$ (because $\mcL_{\psi(a\cdot b)}$ is lisse on
$\mathbf{M}\times \Aa^r$), and thus
\begin{align*}
  \supp(\mcH^i(R'))&=\supp(\mcH^i(p_{12}^*R))\\
  &=
  \supp(p_{12}^*(\mcH^i(R)))=
  p_{12}^{-1}(\supp(\mcH^i(R)))=\supp(\mcH^i(R))\times\Aa^r,
\end{align*}
hence
\[
  \dim \supp(\mcH^i(R'))=r+\dim\supp(\mcH^i(R)).
\]

By definition of semiperverse objects, we deduce that~$R[n]$ is
semiperverse if (in fact, also only if) $R'[n+r]$ is semiperverse:
indeed, we then get
\begin{align*}
  \dim\supp(\mcH^i(R[n]))&= \dim\supp(\mcH^{n+i}(R)) \\& = \dim
  \supp(\mcH^{n+i}(R'))-r\\
  &= \dim \supp(\mcH^{i-r}(R'[n+r]))-r \leq -(i-r)-r=-i.
\end{align*}

% It follows that $R$ is semiperverse if $R'$ is
% semiperverse.
% %% \footnote{Check carefully.}
% If this is so, then we can complete the proof exactly as in the previous
% case.
The following lemma allows us therefore to conclude.

\begin{lemma}
  If~$K$ is semiperverse, then the object $R'[n+r]$ is semiperverse.
\end{lemma}

\begin{proof}
  We will prove this by computing the Fourier transform
  $S=\ft_{\psi/\mathbf{M}}(R'[n+r])$ of $R'[n+r]$ \emph{relative
    to~$\mathbf{M}$}, i.e., summing only over the $h$ and~$b$ variables,
  and checking that it is semiperverse.

  As before, we explain the proof using the ``numerical'' Fourier
  transform, identifying it as the trace function of some ``natural''
  object~$S'$; we can then check that $S$ and~$S'$ are indeed isomorphic
  (geometric isomorphism is enough), up to the correct shift, by
  ``sheafifying'' the computation as in Section~\ref{sec-iso}.

  Let~$(\eta,a,\beta)\in (\Aa^n\times\mathbf{M}\times\Aa^r)(k)$, where
  $\beta$ and~$\eta$ are the Fourier variables dual to~$b$ and~$h$,
  respectively. Then the trace function of~$S$ at~$(\eta,a,\beta)$ is
  \begin{multline*}
    \sum_{\substack{b\in k^r\\h\in k^n}} \psi(-b\cdot \beta-h\cdot
    \eta)\psi(a\cdot b)t_R(a,h)= \sum_{\substack{b\in k^r\\h\in k^n}}
    \psi(-b\cdot \beta-h\cdot \eta)\psi(a\cdot b)\sum_{\substack{x\in
        W(k)\\\Delta(x)=a}}\psi(x\cdot h)t_L(x)\\
    =
    \sum_{\substack{x\in W(k)\\\Delta(x)=a}}t_L(x)\sum_{b\in
      k^r}\psi(b(a-\beta))
    \sum_{h\in k^n}\psi(h\cdot (x-\eta)).
  \end{multline*}

  By orthogonality of characters, the sum over~$b$ vanishes
  unless~$\beta=a$. When this is the case, the trace function becomes
  \[
    |k|^r\sum_{\substack{x\in W(k)\\\Delta(x)=a=\beta}}t_L(x)
    \sum_{h\in k^n}\psi(h\cdot (x-\eta)).
  \]

  The sum over~$h$ is then zero unless $x=\eta$. Thus the trace
  function at~$(\eta,a,\beta)$ is also zero unless~$\eta\in W(k)$
  and~$\Delta(\eta)=a=\beta$. Hence
  \[
    t_{S}(\eta,a,\beta)=
    \begin{cases}
      0&\text{ unless $\beta=a$, $\eta\in W(k)$
        and~$\Delta(\eta)=\beta$},\\
      |k|^{r+n} t_L(\eta)&\text{ if $\beta=a$, $\eta\in W(k)$
        and~$\Delta(\eta)=\beta$.}
    \end{cases}
  \]

  Let~$W'\subset \Aa^n\times\mathbf{M}\times \Aa^r$ be the closed
  subvariety defined by
  \[
    W'=\{(\eta,a,\beta)\,\mid\, \beta=a,\quad \eta\in W,\quad
    \Delta(\eta)=\beta\},
  \]
  and~$i$ the corresponding closed immersion. There is an isomorphism
  \[
    \varphi\colon W\to W'
  \]
  given by $\eta\mapsto (\eta,\Delta(\eta),\Delta(\eta))$, with inverse
  $(\eta,a,\beta)\mapsto \eta$. The computation of trace functions
  suggests that~$S$ and $S'$ are isomorphic, up to shifts and twists,
  where $S'=j_!\varphi_*L$. This object $S'$ is semiperverse, since
  $\varphi$ is an isomomorphism and $j$ is a closed immersion (see,
  e.g.,~\cite[Cor.\,4.1.3]{BBD-pervers}).  Thus, once we check this
  isomorphism, as in Section~\ref{sec-iso}, \emph{mutatis mutandis}, the
  proof is completed.
\end{proof}

\section{Analytically uniform stratifications}\label{sec-quantitative-bounds}

We now prove Theorem~\ref{th-quantitative-bounds}, as well as
Theorem~\ref{th-quantitative-bpw}.

We begin with the former. As already indicated, this will be deduced
from Theorem~\ref{th-general-kl}. The basic idea is to apply the latter
to the \emph{universal} family~$W$ which parameterizes subvarieties
$V\subset \Aa^n_{\Zz}$ of suitable ``complexity''. The KL-datum thus
obtained can be specialized to $V$, and by specialization of the
(fixed!) equations defining it, we obtain equations whose coefficients
are polynomials in the parameters of~$V$. The fact that the KL-datum
does not \emph{always} specialize, however, requires an additional
iterative argument.

First, we state formally the elementary lemma which derives analytic
uniformity from algebraic uniformity.

\begin{lemma}\label{lm-simple}
  Let~$r$, $s\geq 0$ be integers.
  Let~$g\in \Zz[y_1,\ldots, y_r,x_1,\ldots,x_s]$ be a
  polynomial. For~$\uple{y}=(y_i)\in \Zz^r$, let
  $g_{\uple{y}}\in\Zz[x_1,\ldots,x_s]$ be the specialized
  polynomial. Then
  \[
    h_c(g_{\uple{y}})\leq h_c(g)+\deg(g)h(\uple{y})+O(1)
  \]
  for all~$\uple{y}\in\Zz^r$, where the implied constant depends only
  on~$r$ and the degree of~$g$.
\end{lemma}

\begin{proof}
  This is almost tautological: we write
  \[
    g=\sum_{I,J}\lambda_{I,J}y^Ix^J,
  \]
  for some~$\lambda_{I,J}\in\Zz$, all but finitely many of which are
  zero, where~$I$ runs over $r$-tuples of non-negative integers
  and~$J$ over $s$-tuples of non-negative integers, and we use the
  multi-index notation for monomials. Then
  \[
    g_{\uple{y}}=\sum_{J}\lambda_{\uple{y},J}x^J,
  \]
  with
  \[
    \lambda_{\uple{y},J}=\sum_{I}\lambda_{I,J}y^I,
  \]
  and hence
  \[
    h_c(g_{\uple{y}})=\log^+ \max_J|\lambda_{\uple{y},J}|
  \]

  Since
  \[
    |\lambda_{\uple{y},J}|\ll
    (\max_{i}|y_i|)^{\deg(g)}(\max_{I,J}|\lambda_{I,J}|),
  \]
  for any~$J$, where the implied constant depends only on~$r$ and the
  degree of~$g$ (it can be bounded by the number of monomials of
  degree~$\leq \deg(g)$ in~$r$ variables), we get
  \[
    h_c(g_{\uple{y}})\leq \deg(g)h(\uple{y})+h_c(g)+O(1),
  \]
  which is the result.
\end{proof}

\begin{proof}[Proof of Theorem~\ref{th-quantitative-bounds}]
  Let $n$, $r$, $\delta$ be given non-negative integers.
  %% locally-
  % closed subscheme of $\Aa^n_{\Zz}$, with $V_{\Cc}$ of
  % dimension~$\leq d$ and given by vanishing
  % %% ornon-vanishing
  % of~$\leq r$ polynomials of degree~$\leq \delta$. 
  We denote by~$\mathbf{M}$ the (affine) space parameterizing
  $r$-tuples of integral polynomials in~$n$ variables of
  degree~$\leq \delta$; its points will be denoted
  $\uple{g}=(g_i)_{1\leq i\leq r}$, the coordinates corresponding to
  the coefficients of the polynomials. The relative dimension~$m$
  of~$\mathbf{M}=\Aa^m_{\Zz}$ is the number of these coefficients, and
  depends only on $(n,r,\delta)$. Let $W\subset \Aa^{m+n}_{\Zz}$ be
  the ``universal'' closed subscheme of the given type and
  $\Delta\colon W\to \Aa^m_{\Zz}$ the corresponding projection: we
  have
  \[
    W=\{((g_i),x)\in \Aa^{m+n}\,\mid\, g_i(x)=0\text{ for all } i\},
  \]
  and $\Delta((g_i),x)=(g_i)$.  Let further~$f\in \Zz[x_1,\ldots,x_n]$
  be given; this defines a morphism
  $\widetilde{f}\colon W\to\Aa^1_{\Zz}$ given by
  $\widetilde{f}(\uple{g},x)=f(x)$.
  
  Note that for a given tuple $\uple{g}$, the fiber $W_{\uple{g}}$ is
  naturally identified with the closed subscheme~$V$ given by the
  vanishing of the $g_i$'s, and the restriction of~$\widetilde{f}$ to this
  fiber is then identified with the polynomial~$f$ on this closed
  subscheme~$V$. In particular, the data of \emph{any} given closed
  subscheme~$V$ of this type and the polynomial~$f$ appears as one of
  these $(W_{\uple{g}},\widetilde{f}|W_{\uple{g}})$.

  We apply Theorem~\ref{th-algebraic-dependency} with $\widetilde{f}$
  and with the invertible function~$g=1$ (not to be confused with the
  polynomials $\uple{g}$). We thus obtain the data
  \[
    ((Y_j),N,C,A, \varphi)
  \]
  from the theorem, such that 
  \[
    (Y_{j,\uple{g}},N\varphi(\uple{g}),C)
  \]
  is a KL-datum for $(W_{\uple{g}},\widetilde{f}|W_{\uple{g}})$ for
  all~$\uple{g}$ outside a codimension one subscheme, which we
  denote~$\mathbf{M}_1$ here. Since the equations defining any stratum
  $Y_{j,\uple{g}}$ can be obtained simply by specializing some of the
  variables in a \emph{fixed} set of equations defining $Y_j$, it
  follows that for $\uple{g}\notin \mathbf{M}_1$, the degree
  of~$Y_{j,\uple{g}}$ is bounded in terms of~$(n,r,\delta)$ only, and
  the coefficients of the equations are bounded polynomially (see
  Lemma~\ref{lm-simple}). Moreover, the value of~$N\varphi(\uple{g})$
  is also bounded polynomially in terms of the coefficients of the
  polynomial~$\uple{g}$. This provides the conclusion of
  Theorem~\ref{th-quantitative-bounds} for all~$V$ such that the
  corresponding tuple~$\uple{g}$ is \emph{not} in~$\mathbf{M}_1$, up
  to changing notation (e.g., the integer~$N$ in that statement is the
  value~$N\varphi(a)$).
  
  In order to also obtain a similar result for~$\mathbf{M}_1$, we
  iterate this process by considering the restricted family
  $W_1=\Delta^{-1}(\mathbf{M}_1)$, with the morphism
  \[
    \Delta_1\colon W_1\to \mathbf{M}_1.
  \]
  
  Applying Theorem~\ref{th-general-kl}, we obtain other data
  $(\mcY_1,N_1,C_1,A_1,\varphi_1)$, which is suitable for $\uple{g}$
  in $\mathbf{M}_1$ outside a codimension~$1$ subscheme. Since we only
  need to repeat this process finitely many times to obtain KL-datum
  for all $\uple{g}$ (note that when the exceptions form a
  dimension~$0$ subscheme, i.e., a finite set, we just pick an
  individual KL-stratification of each of these exceptions), which
  will involve a fixed number of constants $(N_i,C_i)$, we obtain the
  conclusion.
\end{proof}

The proof of Theorem~\ref{th-quantitative-bpw} is similar. We refer
the reader to Section~\ref{ssec-thin} for the notation.

\begin{proof}[Proof of Theorem~\ref{th-quantitative-bpw}]
  Fix an integer~$d$ with $0\leq d\leq D$.  We apply
  Theorem~\ref{th-general-kl-trace} to the following data:
  \begin{enumerate}
  \item the space~$\mathbf{M}$ is the parameter space (via the
    coefficients) of polynomials
    \[
      F\in \Zz[y,x_1,\ldots,x_n]
    \]
    of total degree~$d$, monic in~$y$; we denote by $m$ its relative
    dimension,
  \item $W=\mathbf{M}\times \Aa^n_{\Zz}$ and
    $\Delta\colon W\to \mathbf{M}$ is the canonical projection,
  \item $\mcW$ is a stratification of $W$ adapted to the object
    \[
      K=R\pi_!\bQl[m+n]((m+n)/2)
    \]
    of~$D^b_c(W,\bQl)$, where $\pi\colon \widetilde{W}\to W$ is the
    projection from the subscheme
    \[
      \widetilde{W}=\{(F,\uple{x},y)\in W\times \Aa^1_{\Zz}\,\mid\,
      F(\uple{x},y)=0\}.
    \]
  \end{enumerate}

  The assumption that~$F\in\mathbf{M}$ is monic in~$y$ is used in the
  following way: it implies that the morphism $\pi$ is quasi-finite,
  i.e., that the fibers over all~$x\in \Aa^n$ are finite. Indeed, for
  given $(F,\uple{x})\in W$, the equation
  \[
    F(\uple{x},y)=0
  \]
  for the values of~$y\in\Aa^1$ defining the fiber is monic in~$Y$,
  hence has finitely many solutions. This property of~$\pi$ implies that
  the object~$K$ is semiperverse (in fact, even perverse) and mixed of
  weights $\leq 0$ (see for instance~\cite[Cor.\,4.1.3]{BBD-pervers} for
  the first property). 
  
  The other key point is that the trace function of~$K$ is, up to
  normalization, given by
  \[
    (F,\uple{x})\mapsto \sum_{\substack{y\in \Fp\\F(y,\uple{x})=0}}1,
  \]
  so that, by the proper base change theorem, looking at the fiber
  (of~$\Delta$ now) above a point $F\in \mathbf{M}$, we obtain an object
  $K_F$ with trace function given (up to normalization again) by
  \[
    \uple{x}\mapsto \sum_{\substack{y\in
        \Fp\\F(y,\uple{x})=0}}1=r_F(\uple{x}),
  \]
  and hence such that the exponential sums are exactly the sums
  $S_F(\uple{h},p)$ of~(\ref{eq-sfp}).

  Finally, we note that (by Example~\ref{ex-transverse}), the object
  $K$ is $V_F$-transverse for any~$F$.
 
  After applying Theorem~\ref{th-general-kl-trace}, we obtain a KL-datum
  $(\mcX,N,C)$; arguing as in the proof of
  Theorem~\ref{th-quantitative-bounds} with Lemma~\ref{lm-simple}, we
  have the desired quantitative bounds for all parameters $F$ outside of
  a codimension~$1$ subscheme. We then iterate this procedure with the
  corresponding subfamily, once more in the same manner as the proof of
  Theorem~\ref{th-quantitative-bounds}.

  Finally, the added information that the strata may be defined by
  homogeneous polynomials is deduced by the same argument
  as~\cite[p.\,131]{fouvry-katz}, which shows that (for sums of the type
  we consider), any stratum $X_{j,F}$ may be replaced by the homogeneous
  subscheme of $X_{j,F}$ defined by the vanishing of the homogeneous
  components of a generating set of the ideal defining $X_{j,F}$. This
  procedure will preserve the boundedness properties of the original
  stratification.
\end{proof}

\appendix
\numberwithin{equation}{section}

\section{Intuition for analytic number theorists (by
  A. Forey, J. Fresán and E. Kowalski)}

This Appendix reproduces, with slight modifications, Sections~1 to~4
of Appendix~E of the book~\cite{ffk} of Forey, Fresán and Kowalski.
The goal is to provide readers who have a background in analytic
number theory with some intuition and feeling for objects such as
$\ell$-adic complexes and perverse sheaves.

The focus here concerns trace functions of \emph{more than one
  variable}.  The theory of trace functions in \emph{one} variable is
more accessible, as the algebraic objects can be presented more
concretely using Galois theory of function fields.  Some familiarity
with this point of view will certainly also be very helpful in
developing intuition. A concise introduction can be found in the Pisa
survey of Fouvry, Kowalski and Michel~\cite{pisa}, and a more detailed
treatment is contained in the lectures of Michel at the 2016 Arizona
Winter School~\cite{arizona}.

We fix a finite field $k$, and denote by $k_n$ the extension of $k$ of
degree~$n$ inside a fixed algebraic closure~$\bar{k}$.  For simplicity
of notation, we will mostly speak about trace functions on the affine
space~$\Aa^m$ for some integer $m\geq 0$. However, it will be implicit
that most of what we discuss can be done for any algebraic variety $Y$
over $k$ (and this is needed, for instance because we often naturally
wish to restrict a trace function to a subvariety, where some particular
property holds), for instance for powers of the multiplicative group
$\Gg_m$ (i.e., $Y$ such that $Y(k_n)=(k_n^{\times})^d$ for some
$d\geq 0$). The reader should keep in mind that for such a subvariety,
of dimension~$d\leq m$, the size of the finite set~$Y(k_n)$ of points
of~$Y$ with coordinates in~$k_n$ is approximately $|k_n|^{d}$ when $n$
is large.

Throughout, we fix a non-trivial additive character
$\psi\colon k\to \Cc^{\times}$ and, for $n\geq 1$, we define
\begin{align*}
  \psi_n \colon k_n &\longrightarrow \Cc^{\times} \\
  x &\longmapsto \psi(\Tr_{k_n/k}(x)).
\end{align*}

We finally note that we will completely ignore (here) the distinction
between $\Qlb$ and~$\Cc$.

\subsection{Trace functions}\label{ssec-trace}

The concrete origin for the use of methods of algebraic geometry and
étale cohomology in analytic number theory lies in trace functions,
and especially in exponential sums. Properly speaking, a trace
function on~$\Aa^m$ is the data of a family $(t_n)_{n\geq 1}$ of
functions $k_n^m\to \Cc$, and it is associated to some algebraic
object~$M$, which we call a ``coefficient object''.  This object is
not uniquely determined by~$(t_n)$, but we will not worry about this
matter in this appendix.

The first examples of trace functions arise from polynomials
$f\in k[X_1,\ldots,X_m]$ by means of 
\begin{equation}\label{eq-simplest-trace}
  t_n(x_1,\ldots,x_m)=\psi_n(f(x_1,\ldots,x_m));
\end{equation}
the corresponding coefficient object is denoted by $\mcL_{\psi(f)}$. Many other examples are then obtained by applying various operations,
which are known to preserve the set of trace functions (these are
operations on the coefficient objects, which are reflected in a specific
operation at the level of trace functions). These operations include the
following, where we indicate the algebraic notation for the
corresponding coefficient objects:
\begin{itemize}
\item The constant function $1$ is associated to the coefficient object 
  $M=\bQl$.
\item The sum of the trace functions associated to $M_1$ and $M_2$ is
  associated to $M_1\oplus M_2$ (and if a coefficient object~$M$ has
  this form, we also say that~$M_1$ and~$M_2$ are \emph{direct
    factors} of~$M$).
\item If $(t_n)$ is a trace function associated to~$M$, then
  $((-1)^kt_n)$ is a trace function for each integer $k\in \Zz$,
  associated to a coefficient denoted by $M[k]$ and called a ``shift''
  of $M$.
 
\item If $(t_n)$ is a trace function associated to~$M$, then
  $(|k_n|^r t_n)$ is a trace function for each integer~$r \in \Zz$,
  associated to a coefficient denoted by $M(-r)$ and called a ``(Tate)
  twist'' of $M$.
  
\item The product of the trace functions associated to $M_1$ and $M_2$
  is associated to $M_1\otimes M_2$.
\item If $f=(f_1,\ldots,f_d)\colon \Aa^m\to \Aa^d$ is a tuple of
  polynomials in $k[X_1,\ldots,X_m]$, and $s=(s_n)$ is a trace
  function on~$\Aa^d$ associated to a coefficient~$N$, then
  $$
  t_n(x_1,\ldots,x_m)=s_n(f(x_1,\ldots,x_m))
  $$
  defines a trace function $(t_n)$ on~$\Aa^m$, which we also denote by 
  $s\circ f$. The corresponding coefficient is $f^*N$.
\item If $f=(f_1,\ldots,f_d)\colon \Aa^m\to \Aa^d$ is a tuple of
  polynomials in $k[X_1,\ldots,X_m]$, and $t=(t_n)$ is a trace function
  on~$\Aa^m$, associated to a coefficient object~$M$, then
  \begin{equation}\label{eq-sum-fiber}
    s_n(y_1,\ldots, y_d)=\sum_{\substack{x\in k_n^m\\f(x)=y}}t_n(x)
  \end{equation}
  defines a trace function on~$\Aa^d$; the associated coefficient object is
  denoted by $Rf_!M$.
\end{itemize}

\begin{example}[Fourier transform]\label{ex-fourier}
  This formalism is already sufficient to explain Deligne's Fourier
  transform. Let $m\geq 1$ be an integer, and consider the projections
  $$
  p_1,\ p_2\colon \Aa^{2m}\to \Aa^m
  $$
given by 
  $$
  p_1(x_1,\ldots,x_m,y_1,\ldots,y_m)=(x_1,\ldots,x_m),\quad\quad
  p_2(x_1,\ldots,x_m,y_1,\ldots,y_m)=(y_1,\ldots,y_m).
  $$
  We write 
  \[
  X\cdot Y=X_1Y_1+\cdots+X_mY_m
  \]
  for variables $X_i$ and $Y_j$. This is a polynomial with coefficients
  in~$k$, so the functions
  $$
  F_n(x,y)=\psi_n(-x_1y_1-\cdots -x_my_m)
  $$
  define a trace function $F=(F_n)$ on~$\Aa^{2m}$, associated to the
  coefficient object $\mcL_{\psi(-X\cdot Y)}$.
  
  Let $t=(t_n)$ be a trace function on~$\Aa^m$ with coordinates $(x_1, \ldots, x_m)$. Then the discrete
  Fourier transforms $(\widehat{t}_n)$, which are defined for $n\geq 1$
  and $y\in k_n^m$ by
  $$
  \widehat{t}_n(y)=\sum_{x\in k_n^m}t_n(x)F_n(x,y)=
  \sum_{x\in k_n^m}t_n(x)\psi_n(-x\cdot y), 
  $$
  also define a trace function $\widehat{t}=(\widehat{t}_n)$. Indeed,
  for any $y$, the set of all $x\in k_n^m$ can be identified with the
  set of $(x,y)\in k_n^{2m}$ such that $p_2(x,y)=y$, and we have
  $t_n(x)=t_n(p_1(x,y))$, so that if~$t$ is associated to the
  coefficient object~$M$, then the formalism above shows that $\widehat{t}$ is
  associated to
  $$
  \widehat{M}=Rp_{2!}(p_1^*M\otimes \mcL_{\psi(-X\cdot Y)}).
  $$
\end{example}

\subsection{Weights and purity: lisse sheaves}\label{ssec-weight-lisse}

The formalism of trace functions is useful in analytic
number theory \emph{because} of Deligne's Riemann hypothesis over finite
fields. This also leads to some understanding of the important
qualitative differences between various types of trace functions---corresponding to classes of coefficients which may (for instance) be
lisse sheaves, constructible sheaves, complexes of constructible
sheaves, or perverse sheaves. We will try in this and the following
sections to provide the readers with some intuition of the concrete meaning of these
notions.

The key concept (due to Deligne) is that of a coefficient~$M$ which is
\emph{punctually pure}, or \emph{pure}, of some weight~$w\in\Zz$. The
main conceptual difficulty is that the meaning of this property for the
corresponding trace function is not straightforward in general.

The simplest case (from which the others will be derived) is that of $M$
which is a single ``lisse sheaf''. \emph{In that case}, the concrete
meaning\footnote{\ But not exactly the precise definition.} of $M$ being
\emph{punctually pure of weight~$w$}, in terms of the trace function $t=(t_n)$,
is that there exist
\begin{itemize}
\item an integer $r\geq 0$, the \emph{rank} of~$M$, 
\item for each $n\geq 1$ and~$x\in k_n^m$, a unitary matrix
$\Theta_M(x;k_n)\in \Un_r(\Cc)$, well-defined up to conjugacy, 
\end{itemize}
so that the following equality holds:  
$$
t_n(x)=|k_n|^{w/2}\Tr(\Theta_M(x;k_n)).
$$
In particular, note that this implies the estimate
$$
|t_n(x)|\leq r|k_n|^{w/2}
$$
for all~$n$ and~$x\in k_n^m$.

In the remainder of this appendix, we will sometimes say that a lisse
sheaf, or its trace function, is ``pure'' instead of the more correct
``punctually pure''.

\begin{remark}
  The reader may ask why we speak of a conjugacy class of unitary
  matrices instead of its multiset of eigenvalues, which contains the
  same information. The point is that if one is studying the
  \emph{distribution} of values of the trace function, as in Deligne's
  Equidistribution Theorem (or those studied in~\cite{ffk}), then it is
  essential to have this interpretation.

  The matrix $\Theta_M(x;k_n)$ is not arbitrary in~$\Un_r(\Cc)$. For
  instance, its eigenvalues (which of course determine the trace) are
  Weil numbers of weight~$0$, i.e., algebraic numbers in~$\Cc$ for which
  all Galois conjugates have modulus~$1$. Moreover, if $n'$ is a
  multiple of~$n$, then $x\in k_n^m$ can also be viewed as an element of
  $k_{n'}^m$ through the inclusion $k_n \subset k_{n'}$, and the formula
  $$
  \Theta_M(x;k_{n'})= \Theta_M(x;k_{n})^{n'/n}
  $$
  holds (\ie the eigenvalues of the matrix $\Theta_M(x;k_{n'})$ are
  those of $\Theta_{M}(x;k_n)$ raised to the power~$n'/n$).
\end{remark}

As one can expect, the trace functions defined by the
formulas~(\ref{eq-simplest-trace}), associated to $\mcL_{\psi(f)}$, are of this type, with $r=1$, $w=0$, and the matrix $\Theta(x;k_n)$ reduced
to the single complex number of modulus one $\psi_n(f(x))$. Moreover, it
is also intuitively clear (and true) that some of the operations
discussed above will respect the special class of trace functions
associated to pure lisse sheaves.

For instance:
\begin{itemize}
\item If $t$ and~$t'$ are trace functions associated to objects $M$
  and~$N$ which are both lisse sheaves pure of (the same) weight~$w$,
  then $t+t'$ is also pure of weight~$w$; we have 
  \[
  \Theta_{M\oplus N}(x;k_n)=\Theta_M(x;k_n)\oplus \Theta_N(x;k_n).
  \]
\item If $t$ and~$t'$ are trace functions associated to objects $M$
  and~$N$ which are both lisse sheaves pure of weights~$w$ and~$w'$,
  respectively, then $tt'$ is also pure of weight~$w+w'$. In other
  words, $M\otimes N$ is still a lisse sheaf, pure of that weight; in
  fact, we have 
  \[
  \Theta_{M\otimes N}(X;k_n)=\Theta_M(x;k_n)\otimes \Theta_N(x;k_n).
  \]
\item If $f=(f_1,\ldots,f_d)\colon \Aa^m\to \Aa^d$ is a tuple of
  polynomials in $k[X_1,\ldots,X_m]$, and $s$ is a trace function
  on~$\Aa^d$ associated to a lisse sheaf of weight~$w$, then $s\circ f$
  is also pure of weight~$w$. In other words, $f^*N$ is still a lisse
  sheaf, pure of weight~$w$; in fact, we have
  \[
  \Theta_{f^*N}(x;k_n)=\Theta_N(f(x);k_n).
  \]
\end{itemize}

But elementary examples show that the crucially important operation of
``summing over the fiber'' (see~(\ref{eq-sum-fiber})) does not always
send a single lisse sheaf to a lisse sheaf, and may also not map a trace
function which is pure of some weight to another one.

\begin{example}\label{ex-trace-fns}
  (1) Let $m=d=1$ and $f\in k[X]$ a polynomial of degree $2$, viewed
  as a map from $\Aa^1$ to itself. Assume that the characteristic
  of~$k$ is different from~$2$. We consider the trace function $(t_n)$
  with $t_n(x)=\psi_n(x)$, associated to the lisse sheaf
  $\mcL_{\psi(X)}$ (of weight~$0$), and the trace function $(s_n)$
  defined by
  $$
  s_n(x)=\sum_{\substack{y\in k_n\\ f(y)=x}}t_n(y)=\sum_{\substack{y\in
      k_n\\ f(y)=x}}\psi_n(y),
  $$
  for $n\geq 1$ and $x\in k_n$, which is associated to the coefficient
  object $Rf_!\mcL_{\psi(X)}$. For most $x$, the value of $s_n(x)$ is
  either $0$ (if $f(y)=x$ has no solutions in $k_n$) or a sum of two
  roots of unity, but for the single point $x_0=f(y_0)$, where $y_0$
  is the unique zero of the derivative of~$f$, the value $s_n(x_0)$ is
  a single root of unity (note that $y_0$, and hence $x_0$, belongs to
  $k$, so it also belongs to $k_n$ for all $n$, but the value of
  $s_n(x_0)$ does vary with $n$).

  (2) We consider $m=2$ and the trace function $(t_n)$ defined by
  $t_n(x,y)=\psi_n(xy^2)$ for $(x,y)\in k_n^2$. It is associated to the coefficient object 
  $\mcL_{\psi(XY^2)}$, which is pure of weight~$0$. Let $d=1$ and
  $f=X$. Then $Rf_!\mcL_{\psi(XY^2)}$ has the trace function $(s_n)$
  such that
  $$
  s_n(x)=\sum_{y\in k_n} \psi_n(xy^2)=\begin{cases}
    \text{a quadratic Gauss sum}& \text{ if } x\not=0,\\
    |k_n|&\text{ if } x=0.
  \end{cases}
  $$
\end{example}

Neither of these examples of trace functions are associated to a
single punctually pure lisse sheaf. However, it turns out that the
underlying reason is not the same. In Example~\ref{ex-trace-fns}, (1),
the issue is that $(s_n)$ is associated to a single constructible
sheaf which is ``not lisse'' at the point~$x_0$. In
Example~\ref{ex-trace-fns}, (2), the issue is that $(s_n)$ is
associated to a ``complex'' of constructible sheaves, i.e., not to a
single sheaf.

\subsection{Weights and purity: constructible sheaves and complexes}\label{ssec-weight}

In fact, the most general source of trace functions are \emph{(bounded)
  mixed complexes of constructible sheaves}.  We now try to outline the
concrete interpretation of these more general conditions.

The first step goes from a single lisse sheaf to a \emph{single
  constructible sheaf}. Such a sheaf is \emph{(punctually) pure of
  weight~$w$} if there is a ``stratification''
$$
\emptyset=X_0\subset X_1\subset \cdots\subset X_q=\Aa^m
$$
of $\Aa^m$, where $X_i$ is a closed subvariety of~$X_{i+1}$, so
that the restriction of~$M$ to each of the pieces~$X_{i+1}\setminus X_i$
is a single lisse sheaf, punctually pure of weight~$w$, and of some
rank~$r_i\geq 0$ (which in general depends on~$i$).

Concretely, for a given $x\in k_n^m$, there exists a unique $i$ such that
$x\in X_{i+1}\setminus X_i$, and then there exists a unitary matrix
$\Theta_M(x;k_n)$ of size $r_i$ such that
\begin{equation}\label{eq-trace-tn}
  t_n(x)=|k_n|^{w/2}\Tr(\Theta_M(x;k_n)).
\end{equation}

\begin{example}
  Example~(1) above is of this kind, with the stratification
  $$
  \emptyset\subset \{x_0\} \subset \Aa^1,
  $$
  and with $r_0=1$ and $r_1=2$. On $\{x_0\}$, the unique eigenvalue is
  $s_n(x_0)=\psi_n(y_0)$, viewing $x_0$ as belonging to $k_n$. On $\Aa^1\setminus \{x_0\}$, the two
  eigenvalues are either opposite (hence the trace is zero) if
  $x\notin f(k_n)$, or are given by $\psi_n(y)$, for $y$ ranging over
  the two roots of the quadratic equation~$f(y)=x$.
\end{example}

More generally, Deligne defined a \emph{mixed constructible sheaf} of
weights $\leq w$ by the condition that there is a filtration with
associated punctually pure quotients~$M_j$, each of some weight
$w_j\leq w$. Concretely, this implies that the trace function $t=(t_n)$
is given by
$$
t_n(x)=\sum_{j\in J}t_{n,j}(x)
$$
for some finite set~$J$, where each family $(t_{n,j})_{n\geq 1}$ is the
trace function of a constructible sheaf which is pure of weight
$w_j\leq w$.
 
Finally, the most general type of trace functions arises from
objects~$M$ that are \emph{complexes of constructible sheaves}. Such a
complex gives in particular rise to a \emph{sequence}
$(\mcH^i(M))_{i\in\Zz}$ of constructible sheaves, with $\mcH^i(M)=0$ for
all but finitely many $i$, in such a way that
$$
t_n(x)=\sum_{i\in\Zz}(-1)^it_{n,i}(x)
$$
for all $n\geq 1$ and $x\in k_n$, where $(t_{n,i})_{n\geq 1}$ is the
system of trace functions for the constructible
sheaf~$\mcH^i(M)$. (These sheaves are called the \emph{cohomology
  sheaves} of the complex~$M$.)  These complexes can be seen as defining
the objects of the derived category $D_c^b(X,\bQl)$ on an algebraic
variety~$X$, which appears throughout the paper.\footnote{\ Although the
  full definition involves describing also morphisms, and especially
  possible isomorphisms, among these objects, which are quite subtle.}

\begin{example}
  Example~(2) above is obtained from a complex of constructible
  sheaves~$M$, where there are two non-zero pieces, namely $\mcH^1(M)$
  and~$\mcH^2(M)$.

  The sheaf~$\mcH^1(M)$ is constructible for the stratification
  $$
  \emptyset\subset \{0\}\subset \Aa^1,
  $$
  with the piece on~$\{0\}$ of rank~$0$, and the piece
  on~$\Aa^1\setminus \{0\}$ of rank~$1$, pure of weight~$1$, with the
  corresponding unique eigenvalue equal to the quadratic Gauss sum
  $$
  \sum_{y\in k_n} \psi_n(xy^2)
  $$
  for $x\in k_n\setminus \{0\}$.

  The sheaf $\mcH^2(M)$ is also constructible, for the same
  stratification (but this is not a general feature), with
  the lisse sheaf of rank~$0$ on $\Aa^1\setminus \{0\}$, and a piece of
  rank~$1$ of weight~$2$ at $\{0\}$, where the corresponding unitary
  matrix has eigenvalue $1$ (so the value of the trace function is
  $|k_n|$, see~(\ref{eq-trace-tn}).
\end{example}

However, for a complex~$M$, the definition of what it means that $M$ is
\emph{pure of weight~$w$} is much more subtle than for a single
sheaf. In particular, it does \emph{not} mean that each piece $\mcH^i(M)$ is itself a punctually pure sheaf of
  weight~$w$. More precisely, one defines first the \emph{mixed
  complexes of weights $\leq w$}, which are those such that $\mcH^i(M)$
is a mixed constructible sheaf of weights $\leq w+i$ for any
$i\in\Zz$. There is then furthermore defined another complex $\dual(M)$,
called the \emph{Verdier dual} of~$M$, and $M$ is said to be pure of
weight~$w$ if $M$ is mixed of weights~$\leq w$ and~$\dual(M)$ is mixed
of weights $\leq -w$.

\begin{remark}
  (1) For a single lisse sheaf~$M$ which is punctually pure of
  weight~$0$, and viewed as a complex placed in degree~$0$, the
  corresponding complex has $\mcH^0(M)=M$ and $\mcH^i(M)=0$ for all
  $i\not=0$. One can prove that the Verdier dual is a
  complex~$\dual(M)$ such that $\mcH^{-2m}(\dual(M))$ is a lisse sheaf
  which is pure of weight $-2m$ and all the other cohomology sheaves
  vanish, so that the two definitions of purity coincide for lisse
  sheaves. In fact, the trace function of $\dual(M)$ is \emph{in this
    case} the complex conjugate of the trace function of~$M$.
  \par
  (2) In practice, if an analytic number theorist is interested in a
  single trace function (e.g., one that represents a concrete family of
  exponential sums which one is interested in estimating) and one is not applying further
  operations like $Rf_!$, then one can quite often reduce to the case of
  a single lisse sheaf. This is for example the case for the
  hyper-Kloosterman sums in two variables
  $$
  \Kl_3(x;k_n)=\frac{1}{|k_n|}\sum_{\substack{a,b,c\in
      k_n^{\times}\\abc=x}} \psi_n(a+b+c),
  $$
  or the famous sums
  $$
  FI(x,y;k_n)=\sum_{z\in k_n^{\times}} \Kl_3(xz;k_n)
  \Kl_3(yz;k_n)\psi_n(z)
  $$
  which arose in the work of Friedlander and Iwaniec on the ternary
  divisor function~\cite{friedlander-iwaniec}, and reappeared in the
  work of Zhang~\cite{zhang}.
  
  Indeed, if the exponential sum is mixed, this will often be clear from
  the definition, or from a preliminary analysis, and one can
  ``isolate'' the part of most interest (of highest weight usually),
  which will be associated to a punctually pure constructible
  sheaf. Then by restricting the set of definition according to a
  suitable stratification, one will ensure that one handles a lisse
  sheaf.
  
  For $m=1$, this second step means avoiding finitely many values of $x$
  where the sheaf has unusual behavior; for $m\geq 2$, this means
  avoiding those that satisfy some non-trivial polynomial equation
  $g(x_1,\ldots,x_m)=0$. These special parameters can then be handled
  separately---giving rise to a kind of inductive process which
  reflects exactly the algebraic stratification of the corresponding
  coefficient~$M$.
\end{remark}

One good explanation for the focus on mixed objects with bounded weights
can be found (a posteriori) from the statement of Deligne's most general
form of the Riemann hypothesis. In our context, it can be stated as
follows:

\begin{theorem}[Deligne]\label{thm:DRH}
  Let $(t_n)$ be a trace function on~$\Aa^m$ associated to a complex~$M$
  which is mixed of weights~$\leq w$. Let $f=(f_1,\ldots, f_d)$ be a
  tuple of polynomials in $k[X_1,\ldots,X_m]$.  The complex $Rf_!M$ is
  \emph{mixed} of weights $\leq w$ , and so its trace functions
  $$
  s_n(y)=\sum_{\substack{x\in k_n^m\\f(x)=y}}t_n(x)
  $$
  have the properties of trace functions of mixed objects of
  weights~$\leq w$.
\end{theorem}

\begin{remark}
  On the other hand, even if $M$ is a single lisse sheaf, punctually
  pure of weight~$w$, it is \emph{not always the case} that $Rf_!M$ is
  pure.
\end{remark}

A benefit of introducing these more general definitions is that all
operations now respect the property of being mixed for any trace
function, with a good understanding of how the weights may~change:
\begin{itemize}
\item The lisse sheaf $M=\bQl$ is pure of weight~$0$.
\item If $M_1$ and $M_2$ have weights $\leq w_1$ and $\leq w_2$,
  respectively, then $M_1\oplus M_2$ has weights \hbox{$\leq \max(w_1,w_2)$}
  and $M_1\otimes M_2$ has weights $\leq w_1+w_2$.
\item If $M$ has weights $\leq w$, then for any $k\in\Zz$, the shifted
  complex $M[k]$ has weights $\leq w+k$.
  \item If $M$ has weights $\leq w$, then for any $r\in\Zz$, the twisted
  complex $M(r)$ has weights $\leq w-2r$.
\item If $f=(f_1,\ldots,f_d)\colon \Aa^m\to \Aa^d$ is a tuple of
  polynomials in $k[X_1,\ldots,X_m]$, and $s=(s_n)$ is a trace function
  on~$\Aa^d$ associated to a mixed complex~$N$ of weights~$\leq w$, then
  $f^*N$ has weights $\leq w$.
\item If $f=(f_1,\ldots,f_d)\colon \Aa^m\to \Aa^d$ is a tuple of
  polynomials in $k[X_1,\ldots,X_m]$, and if $M$ has weights $\leq w$,
  then   $Rf_!M$ has weights $\leq w$ (this is again Deligne's Theorem \ref{thm:DRH}).
\end{itemize}
\par
All objects that occur in practice in analytic number theory\footnote{\
  And indeed more generally in algebraic geometry.} are mixed
complexes. This means that any trace function $(t_n)$ has a
decomposition
$$
t_n=\sum_{a\leq w\leq b} t_{n,w}
$$
for some $a$ and $b$ (independent of~$n$), where $(t_{n,w})_{n\geq 1}$
is a trace function associated to a complex which is pure of weight~$w$
on some subvariety.

\subsection{Perverse sheaves}\label{ssec-perv}

There remains the task of attempting to explain a further fundamental
subclass of trace functions (hence of complexes), those associated to
\emph{perverse sheaves}. This is a distinguished class of complexes with
remarkable geometric and arithmetic properties. For analytic purposes,
the most important of these is maybe that the \emph{simple} perverse
sheaves provide a \emph{canonical basis} of the abelian group of trace
functions, and that if we restrict to pure perverse sheaves, then this
is in a natural sense a \emph{quasi-orthogonal basis} for the trace
functions of pure complexes of weight $0$. We will now explain these
properties.

The rigorous definition of perverse sheaves is of a similar nature to
that of pure complexes: it is a relatively simple condition for both the complex~$M$
and its Verdier dual $\dual(M)$, called
\emph{semiperversity}.\footnote{\ The complication is that the Verdier
  dual is often difficult to compute.} The condition of
semiperversity concerns the size of the support of the cohomology
sheaves $\mcH^i(M)$ (which are intuitively the points~$x$ where
$\mcH^i(M)$ does not vanish; in the stratification in terms of lisse
sheaves, this is where these sheaves have non-zero rank): for any
$i\in\Zz$, the support of $\mcH^i(M)$ should be of dimension at most
$-i$. (In particular, if $i\geq 1$, then the support should be empty, so
$\mcH^i(M)$ should be zero then.)

Remarkably, this condition can be recovered intuitively from basic
analytic intuition (which highlights that it is extremely natural).
  
Thus consider a trace function $t=(t_n)$ associated to a complex $M$
on~$\Aa^m$ and assume that it is mixed of weights $\leq 0$. From the
analytic point of view, we are often in the situation where the
mean-square of the values of the trace function $t_n$ are bounded (after
some normalization maybe), and bounded away from zero, i.e., for $n$
large enough, we have
\begin{equation}\label{eq-perverse-norm}
  \sum_{x\in k_n^m}|t_n(x)|^2\asymp 1.
\end{equation}

For $i\in\Zz$, the cohomology sheaf $\mcH^i(M)$ should be
``essentially'' pure of weight~$i$ (rigorously, we only know that it is
mixed of weights $\leq i$). So the contribution to the sum above of the
$x$ in the support~$S_i$ of $\mcH^i(M)$ should be expected to be of
order of magnitude
$$
|k_n|^{2\cdot i/2}\times |S_i(k_n)|\approx |k_n|^{i+d_i}
$$
if $S_i$ has dimension $d_i$. Hence the
estimate~(\ref{eq-perverse-norm}) only has a chance to hold if
$i+d_i\leq 0$ for all~$i$, and this is \emph{precisely} the
semiperversity condition.

\begin{example}
  Consider a family of exponential sums of type
  $$
  \frac{1}{|k_n|^{m}} \sum_{y\in k_n^m}\psi_n(f(y)+x_1y_1\cdots +x_my_m)
  $$
  with parameters $(x_1,\ldots,x_m)\in k_n^m$ (these functions of $x$
  are the trace functions of a complex~$M$ which is a normalized form of
  Deligne's Fourier transform of the lisse sheaf $\mcL_{\psi(f)}$).
  \par
  We expect ``generic'' square-root cancellation, so as $n$ varies,
  for ``most'' choices of $x\in k_n^m$, this sum should be of size
  about $|k_n|^{-m/2}$. Since $\mcH^i(M)$ is of weight~$\leq i$, and
  hence contributes terms of size typically expected to
  be~$|k_n|^{i/2}$, this expectation corresponds to the fact that
  $\mcH^i(M)$ should be ``generically'' zero unless~$i=m$, while
  $\mcH^{-m}(M)$ contributes a fixed number of complex numbers of
  modulus $\leq |k_n|^{-m/2}$.

  But for special values of $x$, those satisfying some non-trivial
  polynomial equation $g(x)=0$, one may obtain a larger sum than
  square-root cancellation. Experience teaches that usually this size
  only jumps by one factor $|k_n|^{1/2}$ (so the sum is about
  $|k_n|^{-m/2+1/2}$) if only this one condition is imposed; if it is
  bigger (say of size $|k_n|^{-m/2+1}$), this should mean that a second
  (independent) equation $h(x)=0$ holds, and so on.

  This ``stratification'' of bounds getting steadily worse only on
  smaller subsets corresponds to cohomology sheaves $\mcH^i(M)$
  (contributing terms of size $|k_n|^{i/2}$) vanishing outside of
  subvarieties of dimension at most $-i$.

  In the extreme case, the exponential sum is of size~$1$ (i.e., there
  is no cancellation at all) at worse for finitely many values of the
  parameters, corresponding to $\mcH^0(M)$ being supported on finitely
  many points.
  \par
  This particular example is at the root of the results of Katz, Laumon
  and Fouvry (see Theorem \ref{th-fk}) on stratification for additive exponential sums which we
  considered in this paper.
  % i~\cite{fouvry,KL-fourier-exp-som,fouvry-katz}.
  It should suggest to analytic readers that semiperversity is a
  relatively easy condition to check, and that it should be natural and
  ubiquitous in analytic number theory.
\end{example}

The following statement provides a concrete illustration of the
advantages of perverse sheaves.

\begin{theorem}\label{th-appendix-basis}
  The $\Zz$-module of trace functions on~$\Aa^m$ over~$k$ is generated
  by the trace functions of perverse sheaves, and the trace functions of
  \emph{simple} perverse sheaves form a basis.
\end{theorem}

The first statement is in fact very explicit. Indeed, if $t=(t_n)$ is an
arbitrary trace function, associated to a complex~$M$, one can define
(in addition to its ``usual'' cohomology sheaves $\mcH^i(M)$) its
\emph{perverse cohomology sheaves} $\pH^i(M)$, which are perverse
sheaves, zero for $|i|>m$, such that their trace functions
$(\pt_{i,n})_{n\geq 1}$ satisfy the equation
$$
t_n=\sum_{i\in\Zz}(-1)^i\ \pt_{i,n}
$$
for all $n\geq 1$. Furthermore, a complex $M$ is mixed of
weights~$\leq w$ if and only if each $\pH^i(M)$ is also mixed of
weights~$\leq w+i$ (similarly to the cohomology sheaves;
see~\cite[Th.\,5.4.1]{BBD-pervers}).

\begin{remark}
  To say that a complex $M$ is perverse is to say that its perverse cohomology sheaves are $M=\pH^0(M)$ and
  $\pH^i(M)=0$ for all~$i\not=0$.
\end{remark}

Up to the terminology and notation, the second statement of
Theorem~\ref{th-appendix-basis} is proved by Laumon
in~\cite[Th.\,1.1.2]{laumon} (though it was known before, at least to
Deligne). To understand it, one must explain what are the simple
perverse sheaves which are mentioned there. We will content ourselves
with stating the quasi-orthonormality property which holds for a simple
perverse sheaf that is pure of weight~$0$.  It is another consequence of
Deligne's Riemann Hypothesis, proved by Katz, that if $t=(t_n)$ is the
trace function of a perverse sheaf~$M$, then
\begin{equation}\label{eq-quasi-ortho}
  \limsup_{n\to +\infty}\sum_{x\in k_n^m}|t_n(x)|^2=1
\end{equation}
\emph{if and only if} $M$ is \emph{simple}.

\begin{remark}
  One of the fundamental results of Beilinson, Bernstein, Deligne and
  Gabber \cite[Cor.\,5.3.4]{BBD-pervers} is that a simple perverse sheaf
  which is mixed, as a complex, is in fact \emph{pure} of some weight;
  since non-mixed complexes do not appear in practice, this means that
  simple perverse sheaves in analytic number theory are always pure of
  some weight, and the quasi-orthonormality characterization can be
  extended to all simple perverse sheaves, up to normalization.
\end{remark}

\begin{example}
  We can illustrate how useful this quasi-orthonormality statement can
  be to guess or understand some properties of perverse sheaves by
  noting that it strongly suggests a non-trivial property of simple
  perverse sheaves. Namely, let $M$ be a simple perverse sheaf, pure of
  weight~$0$, and generically non-zero (i.e., the support of~$M$ is all
  of~$\Aa^m$). If we repeat the argument leading to the guess of the
  semiperversity condition, we see that we expect that the contribution
  to
  $$
  \sum_{x\in k_n^m}|t_n(x)|^2
  $$
  of each non-zero cohomology sheaf $\mcH^i(M)$ should be of size
  $$
  \alpha_i |k_n|^{i+d_i}
  $$
  for some integer $\alpha_i\geq 1$, and comparison
  with~(\ref{eq-quasi-ortho}) indicates that $i+d_i$ will be $<0$ except
  for one single value of~$i$. Moreover, one knows that the cohomology
  sheaf $\mcH^{-m}(M)$ is generically non-zero, so this value must be
  $i=-m$, so that we expect that
  $$
  d_i\leq -i-1\quad \text{ for }\quad i\not=-m,
  $$
  which is stronger than the condition $d_i\leq -i$ derived from
  semiperversity only. This property is indeed true, and it is a very
  useful fact in applications.
  % (it is the improved support
  % condition of Proposition~\ref{pr-support-property}).
\end{example}


\begin{thebibliography}{CC}
\bibitem{bpw} D. Bonolis, L. Pierce and K. Woo: \textit{Counting
    integral points on thin sets of type II: singularities, sieves and
    stratification}, preprint, \url{arXiv:2505.11226}.

\bibitem{deligne} P. Deligne: \textit{La conjecture de Weil, II},
  Publ. Math. IHÉS 52 (1980), 137--252.
  
\bibitem{ffk} A. Forey, J. Fresán and E. Kowalski: \textit{Arithmetic
    Fourier transforms over finite fields}, to appear in Astérisque;
  \url{arXiv:2109.11961}.

\bibitem{BBD-pervers} A.A. Be{\u \i}linson, J. Bernstein, P. Deligne and
  O. Gabber: \textit{Faisceaux pervers}, in ``Analysis and Topology on
  Singular Spaces'', Luminy, Astérisque 100, S.M.F, 1982.
  
\bibitem{fouvry} \'E. Fouvry: \textit{Consequences of a result of
    N. Katz and G. Laumon concerning trigonometric sums}, Israel
  J. Math. 120 (2000), 81--96.

\bibitem{fouvry2} \'E. Fouvry: \textit{Sur les propriétés de
    divisibilité des nombres de classes des corps quadratiques},
  Bulletin S. M. F. 127 (1999), 95--113.
  
\bibitem{fouvry-katz} \'E. Fouvry and N. Katz: \textit{A general
    stratification theorem for exponential sums, and applications},
  Crelle 540 (2001), 115--166.

\bibitem{pisa} É. Fouvry, E. Kowalski and Ph. Michel: \textit{Trace
    Functions over Finite Fields and Their Applications}, in Colloquium
  De Giorgi 2013--14, Ed. Norm. Pisa 5, 2015.

\bibitem{arizona} É. Fouvry, E. Kowalski, Ph. Michel and W. Sawin:
  \textit{Lectures on applied $\ell$-adic cohomology}, in ``Analytic
  methods in arithmetic geometry'', Contemp. Math. 740, A.M.S, 2019.

\bibitem{fkms} É. Fouvry, E. Kowalski, Ph. Michel and W. Sawin:
  \textit{Bilinear forms with trace functions}, preprint (2025),
  \url{arXiv:2511.09459}.
  
\bibitem{friedlander-iwaniec} J. Friedlander and H. Iwaniec:
  \textit{Incomplete {K}loosterman sums and a divisor problem}, Ann. of
  Math 212 (1985), 319--350; with an appendix by B. J. Birch and
  E. Bombieri.
  
\bibitem{lei-fu} L. Fu: \textit{Étale cohomology theory}, Nankai
  Tracts in Math. 13, World Scientific (2010).
  
\bibitem{g-w} U. Görtz and T. Wedhorn: \textit{Algebraic geometry, I:
    schemes}, Vieweg-Teubner 2010.

\bibitem{gsz} A. Granville, I. Shparlinski and A. Zaharescu:
  \textit{On the distribution of rational functions along a curve
    over~$\Ff_p$ and residue races}, J. Number Theory 112 (2005),
  216--237.

\bibitem{hooley} C. Hooley: \textit{On the number of points on a
    complete intersection over a finite field}, J. Number Theory 38
  (1991), 338--358; with an Appendix by N.M. Katz.

\bibitem{ireland-rosen} K. Ireland and M. Rosen: \textit{A classical
    introduction to modern number theory}, 2nd edition, Grad. Texts in
  Math. 84, Springer, 1992.

\bibitem{iwaniec-kowalski} H. Iwaniec and E. Kowalski:
  \textit{Analytic number theory}, A.M.S Colloquium Publ. 53, 2004.

\bibitem{katz-gkm} N. M. Katz: \textit{Gauss sums, Kloosterman sums and
    monodromy groups}, Annals of Math. Studies 116, Princeton
  Univ. Press (1988).
  
\bibitem{katz-betti} N. M. Katz: \textit{Sums of Betti numbers in
    arbitrary characteristics}, Finite Fields and Their Applications
  7 (2001), 29--44.

\bibitem{katz-act} N. M. Katz: \textit{Affine cohomological
    transforms, perversity and monodromy}, Journal AMS  6 (1993),
  149--222. 

\bibitem{katz-mmp} N. M. Katz: \textit{Moments, monodromy and
    perversity}, Annals of Math. Studies 159, Princeton Univ. Press
  (2005).

\bibitem{katz-laumon} N.M. Katz and G. Laumon: \textit{Transformation
    de Fourier et majoration de sommes exponentielles},
  Publ. Math. IHÉS 62 (1985); 145--202; Corrigendum 69 (1989), 233.

\bibitem{laumon} G. Laumon: \textit{Transformation de Fourier,
    constantes d’équations fonctionnelles et conjecture de Weil},
  Publ. Math. IHÉS 65 (1987), 131--210.

\bibitem{pierce-xu} L.B. Pierce, J. Xu: \textit{Burgess bounds for short character sums evaluated at forms}, Algebra Number Theory 14 (2020), 1911--1951.

\bibitem{qst} W. Sawin; A. Forey, J. Fresán and E. Kowalski:
  \textit{Quantitative sheaf theory}, Journal of the AMS 36 (2023),
  653--726.

\bibitem{Ser97} J-P. Serre: \textit{Lectures on the {M}ordell-{W}eil
    theorem}, Aspects of Mathematics, Vieweg, third edition, 1997.

\bibitem{xu} J. Xu: \textit{Stratification for multiplicative character sums}, Int. Math. Res. Not. 10 (2020), 2881--2917.

\bibitem{zhang}
  Y. Zhang: \textit{Bounded gaps between primes}, Ann. of Maths 179
  (2014), 1121--1174.

\end{thebibliography}
\end{document}